\documentclass[11pt]{article}
\usepackage{amsmath,amssymb,graphicx}
\usepackage{subcaption}
\usepackage{amsfonts, amsmath, amssymb,latexsym,amsthm, mathrsfs, mathtools}
\usepackage[margin=0.9in]{geometry}
\usepackage{url}
\usepackage[english]{babel}
\usepackage{epsfig}
\usepackage{pifont}
\usepackage{tikz-cd}
\usepackage{authblk}

\newtheorem{thm}{Theorem}[section]
\newtheorem*{thm*}{Theorem}
\newtheorem*{prop*}{Proposition}
\newtheorem*{corol*}{Corollary}
\newtheorem*{lemma*}{Lemma}
\newtheorem{prop}[thm]{Proposition}
\newtheorem{corol}[thm]{Corollary}
\newtheorem{lemma}[thm]{Lemma}
\newtheorem{claim}{Claim}[thm]
\newtheorem{sublemma}[claim]{Lemma}
\usepackage[all]{xy}
\newtheorem*{claim*}{Claim}
\newtheorem*{question*}{Question}

\usepackage{chngcntr}
\counterwithin{figure}{section}

\usepackage{enumitem}
\theoremstyle{definition}
\newtheorem{defi}[thm]{Definition}
\newtheorem*{defi*}{Definition}

\theoremstyle{remark}
\newtheorem{remark}[thm]{Remark}

\newtheorem*{remark*}{Remark}
\newtheorem*{remarks*}{Remarks}
\newtheorem*{notation*}{Notation}

\newcounter{main}

\theoremstyle{plain}
\newtheorem{maintheorem}[main]{Theorem}  
\newtheorem{maincor}[main]{Corollary}  
\newtheorem{mainprop}[main]{Proposition}

\usepackage{caption}
\usepackage{subcaption}

\usepackage{fancyhdr}

\usepackage[hidelinks]{hyperref}

\title{Pesin theory for transcendental maps and applications}
\author[1]{Anna Jov\'e\thanks{This work is supported by the Spanish government grants FPI PRE2021-097372 and PID2023-147252NB-I00. Corresponding author. \url{annajove@ub.edu}.}}
\affil[1]{\small Departament de Matemàtiques i Informàtica, Universitat de Barcelona, Barcelona, Spain}
\begin{document}
\maketitle
\begin{abstract}
	In this paper, we develop Pesin theory for the boundary map of some Fatou components of  transcendental functions, under certain hyptotheses on the singular values and the Lyapunov exponent. That is, we prove that generic inverse branches for such maps are well-defined and conformal. In particular,  we study in depth the 
	Lyapunov exponents with respect to harmonic measure, providing results which are of independent interest. 
	As an application of our results, we describe in detail generic inverse branches for centered inner functions, and we prove density of periodic boundary points for a large class of Fatou components. 
\end{abstract}

\section{Introduction}

In the setting of smooth dynamical systems, {\em hyperbolic dynamical systems} play a distinguished role, since they are the easiest to study and exhibit the simplest possible behaviour. Indeed, hyperbolic dynamics are characterized by the presence of expanding and contracting directions for the derivative at every point, which provides strong local, semilocal or even global information about the dynamics. However, the assumption of hyperbolicity is quite restrictive.
A weaker (and hence, more general) form of hyperbolicity, known as {\em non-uniform hyperbolicity}, was initially developed by Yakov Pesin in his seminal work \cite{Pesin76, Pesin77}. Since then,  Pesin's approach to hyperbolicity, also known as {\em Pesin theory}, has been extended, generalized and  refined in numerous articles and  research books (see e.g. \cite{Pollicott}, \cite[Supplement]{KatokHasselblat}, \cite{BarreiraPesin}). Although results apply to both discrete and continuous dynamical systems, in this paper we focus on the discrete ones.

Roughly speaking, Pesin studied originally  $ \mathcal{C}^1 $-diffeomorphisms on compact smooth Riemannian manifolds. Under the assumption that such a map is measure-preserving and ergodic, and no Lyapunov exponent vanishes except on a set of zero measure,  the forward and backwards contraction or expansion around almost every point is controlled asymptotically by the Lyapunov exponents. Applications of this theory include periodic points, homoclinic points, and stable manifold theory \cite[Part II]{Pollicott}.

One of the natural generalizations of Pesin theory is to the setting of iteration of rational maps in the Riemann sphere $ \widehat{\mathbb{C}} $. That is, let $ f\colon\widehat{\mathbb{C}}\to\widehat{\mathbb{C} } $ be holomorphic, and consider the discrete dynamical system generated by $ f $. The phase space $ \widehat{\mathbb{C}} $ is commonly split into two totally invariant sets: the Fatou set $ \mathcal{F}(f) $, where the family of iterates is normal, and hence the dynamics are in some sense stable; and its complement, the Julia set $ \mathcal{J}(f) $. Although the Fatou set is well-understood, the dynamics in the Julia set are more intricate  and worthy of study. For general background in rational iteration we refer to \cite{CarlesonGamelin,milnor}.
In contrast with the setting of $ \mathcal{C}^1 $-diffeomorphisms considered by Pesin, now the iterated function is no longer bijective, which is overcome by assuming a higher degree of regularity on the function. 

A rational map is said to be {\em hyperbolic}  if all orbits of critical values (i.e. images of zeros of $ f' $) are compactly contained in the Fatou set, which already implies that \textit{all} inverse branches around points in $ \mathcal{J}(f) $ are well-defined and uniformly contracting (see e.g. \cite[Sect. V.2]{CarlesonGamelin}, \cite[Sect. 19]{milnor}). Hence, following Pesin's approach for diffeomorphisms, it is natural to ask whether, for a general map (not necessarily hyperbolic), generic inverse branches are well-defined and contracting.
Note that one should make precise the notion of generic inverse branches, by defining  the abstract space of backward orbits for points in $ \mathcal{J}(f) $ and endow it with a measure (using Rokhlin's natural extension, see Sect. \ref{subsect-Rokhlin's-natural-ext}). 

One can prove that, under the assumption of existence of an ergodic invariant probability with positive Lyapunov exponent, for almost every backward orbit $ \left\lbrace x_n\right\rbrace _n $ there exists a disk around the initial point $ x_0 $, such that the corresponding inverse branches of $ f^n $ are well-defined and contracting in this disk (see \cite{Ledrappier,DobbsP}, and also \cite[Sect. 11.2]{PrzytyckiUrbanski}, \cite[Chap. 9.3]{KotusUrbanski}, \cite[Sect. 28.3]{UrbanskiRoyMunday3}, among others). The proof relies strongly on the fact that $ \widehat{\mathbb{C} } $ is compact (and hence, $ \mathcal{J}(f) $ is also compact), and the finiteness of the set of critical values.

We note that the existence of ergodic invariant probabilities supported on the Julia set of rational maps has been historically a topic of wide interest, in connection with the measure of maximal entropy. For polynomials, the existence of such a measure  was already proved by  Brolin \cite{Brolin65},  whereas for rational maps it was done by Freire, Lopes and Mañé \cite{Freire-Lopes-Mane}, and  Lyubich  \cite{Ljubich_1983}, independently. Such a measure of maximal entropy is known to be an ergodic invariant probability, and hence it can be used as an initial cornerstone to develop Pesin theory. Moreover, Lyapunov exponents with respect to any  ergodic invariant probability supported on $ \mathcal{J}(f) $  have been studied in depth \cite{Przytycki-Hausdorff_dimension,Przytycki-Lyapunovexponents, DobbsP}.

The goal of this paper is to extend these well-known results for rational maps to the transcendental setting, that is, for maps $ f\colon\mathbb{C}\to\widehat{\mathbb{C}} $ (transcendental) meromorphic, including the entire case. Although under the presence of poles some orbits get truncated, one can define the Fatou and Julia set for $ f $ in a similar way as for rational maps (for precise definitions and properties, see Sect. \ref{subsect-components-Fatou-K} and references therein). As in the rational case, the question we want to address is whether generic inverse branches are well-defined and contracting around points in $ \mathcal{J}(f) $.

Note that the cornerstones from which the rational Pesin theory is built (namely, compact phase space, finitely many critical values, and existence of ergodic invariant probabilities) no longer hold in general. Indeed, first, the phase space is now $ \mathbb{C} $, which is no longer compact, and nor is the Julia set. In fact, this lack of compactness causes difficulties even for the extension of the notion of hyperbolicity from the rational setting  \cite{Rempe-Sixsmith}.

Additionally, critical values are not the only values where inverse branches fail to be defined. Indeed, one must consider the set of {\em singular values} (i.e. critical and asymptotic values, and accumulation thereof, denoted by $ SV$),  and it may be uncountable.
%Note also that the functions we iterate have always infinite degree.

Finally, the existence of invariant measures on the Julia set is much more delicate and remains somewhat unexplored, as well as  Lyapunov exponents (which depend on the existence of the previous measures).
{Indeed, although the existence of invariant ergodic probabilities supported on the Julia set has been proved for certain families (such as the hyperbolic exponential family \cite{UrbanskiZdunik}), in other cases it is known that they do not exist \cite{Dobbs}. Hence, in contrast with  rational maps, the existence of an ergodic invariant measure supported in the Julia set is unknown in the general setting. }

In order to overcome some of  these difficulties, we restrict ourselves to some forward invariant subsets of the Julia set which are of special interest: the boundaries of invariant (or periodic) connected components of the Fatou set (known as {\em Fatou components}). If we let $ U $ be an invariant Fatou component for $ f $, then its boundary $ \partial U$ is forward invariant under $ f $. In the seminal work of Doering and Mañé \cite{DoeringMane1991}, invariant ergodic measures for $ f\colon \partial U \to \partial U $ supported on $ \partial U $ are given, following the  approach initiated by Przytycki to study rational maps restricted to the boundary  of attracting basins  \cite{Przytycki-Hausdorff_dimension}.

Taking advantage of these invariant measures, under some mild assumptions on the singular values, we are able to overcome the difficulties arising from the lack of compactness, the infinite degree and the presence of infinitely many singular values. Our techniques include refined estimates on harmonic measure and the construction of an appropriate conformal metric.  In this manner, we can develop Pesin theory in the boundary of some transcendental Fatou components  in a quite successful way, which is presented next.
\subsection*{Statement of results}

Let $ f\colon \mathbb{C}\to\widehat{\mathbb{C}} $ be a transcendental meromorphic function, i.e. so that $ \infty $ is an essential singularity for $ f $, and let $ U $ be an invariant Fatou component for $ f $. It is well-known that such an invariant Fatou component is either an attracting basin, a parabolic basin, a rotation domain or a Baker domain (see Sect. \ref{subsect-components-Fatou-K}). In the sequel, we denote by $ \partial U $ the boundary of $ U $ in $ \mathbb{C} $, and $ \widehat{\partial} U $ the boundary  in $ \widehat{\mathbb{C}} $. All the derivatives and absolute values are understood to be with respect to the spherical metric in $ \widehat{\mathbb{C}} $, and hence $ \left| f'\right|  $ is bounded on compact subsets of the plane.

Attracting basins are the natural candidates to perform Pesin theory on their boundary, since the harmonic measure $ \omega_U $ (with basepoint the fixed point $ p\in U $) is invariant under $ f $ and ergodic. The universal assumption throughout the paper is that singular values are `not too dense' on $ {\partial} U $, a condition we make formal by requiring $$ \int_{\partial U} \log\left| x-SV\right|^{-1} d\omega_U(x)<\infty  $$ (see Sect. \ref{subsect-sparseSV}). We note that this assumption is always satisfied if there are  only finitely many singular values on $ \widehat{\partial } U $ (Remark \ref{lemma-finitely-manySV}).

Our main result is the following.

\begin{maintheorem}\label{teo:A} {\bf (Pesin theory for  attracting basins of transcendental maps)}
	Let $ f\colon\mathbb{C}\to\widehat{\mathbb{C}}$ be a meromorphic function, and let $U$ be a simply connected attracting basin, with  fixed point $ p\in U $. Let $ \omega_U $ be the harmonic measure on $ \partial U $ with base point $ p $. Assume $ f $ has positive Lyapunov exponent, that is  $ \log \left| f'\right| \in L^1(\omega_U) $ with $ \int_{\partial U} \log \left| f'(x)\right| d\omega_U(x)>0$. Suppose also that  $ \int_{\partial U} \log\left| x-SV\right|^{-1} d\omega_U(x)<\infty  $.
	
	\noindent Then,  for every countable collection of measurable sets $ \left\lbrace A_k\right\rbrace _k\subset\partial U $ with $ \omega_U(A_k)>0 $, and for $ \omega_U $-almost every $ x_0\in\partial U $,
	 there exists a backward orbit $ \left\lbrace x_n\right\rbrace _n \subset\partial U$ and $ r>0 $ such that
	\begin{enumerate}[label={\em (\alph*)}]
		\item  $ x_{n_k}\in A_k $ for some sequence $ n_k\to\infty $;
		\item the inverse branch $ F_n $ sending $x_0 $ to $ x_n $ is well-defined in $ D(x_0,r) $;
		\item  $ \textrm{\em diam } F_n(D(x_0, r))\to 0$, as $ n\to\infty $.
	\end{enumerate}
\end{maintheorem}

Note that, in particular, for $ \omega_U $-almost every $ x_0\in\partial U $ there exists a backward orbit $ \left\lbrace x_n\right\rbrace _n $, and  inverse branches $ \left\lbrace F_n\right\rbrace _n $ of $f^n$, well-defined in $D(x_0, r)$, such that $ \left\lbrace x_n\right\rbrace _n $ is dense on $ \partial U$.

If we consider parabolic basins or Baker domains, the situation is even more unfavorable, since no harmonic measure on $ \partial U $ is $ f $-invariant. Nevertheless,  there exists a $ \sigma $-finite measure which is absolutely continuous with respect  to harmonic measure on $\partial U$, invariant under $ f $, recurrent and ergodic. By means of the first return map, we develope a similar result for parabolic basins and Baker domains.  As far as we are aware, this result is new even for parabolic basins of polynomials, in which case the assumptions are always trivially satisfied.
\begin{maintheorem}\label{teo:B}{\bf (Parabolic Pesin theory)}
	Let $ f\colon\mathbb{C}\to\mathbb{C}$ be a meromorphic function, and let $U$ be a simply connected parabolic basin or   Baker domain,  such that $ SV\cap U $ are compactly contained in $ U $. Let $ \omega_U $ be a harmonic measure on $ \partial U$, such that $ \log \left| f'\right| \in L^1(\omega_U) $ with $ \int_{\partial U} \log \left| f'\right| d\omega_U>0$. Assume  there exists $ \varepsilon >0 $ such that, if $\partial U_{+\varepsilon} \coloneqq \left\lbrace z\in\mathbb{C}\colon \textrm{\em dist}(z, \partial U)<\varepsilon \right\rbrace $, the set of singular values of $ f $ in $ \partial U_{+\varepsilon} $ is finite.

\noindent	Then, for  every countable collection of measurable sets $ \left\lbrace A_k\right\rbrace _k\subset\partial U $ with $ \omega_U(A_k)>0 $, and for $ \omega_U $-almost every $ x_0\in\partial U $ there exists a backward orbit $ \left\lbrace x_n\right\rbrace _n \subset\partial U$ and $ r>0 $ such that
\begin{enumerate}[label={\em (\alph*)}]
	\item   $ x_{n_k}\in A_k $ for some sequence $ n_k\to\infty $;
	\item the inverse branch $ F_n $ sending $x_0 $ to $ x_n $ is well-defined in $ D(x_0,r) $;
	\item  for every subsequence $\left\lbrace x_{n_j}\right\rbrace _j$ with $ x_{n_j}\in D(x_0, r) $, 
	$ \textrm{\em diam } F_{n_j}(D(x_0, r))\to 0$, as $ j\to\infty $.
\end{enumerate}
\end{maintheorem}

\begin{remark*}
	In the particular case when $ f $ is an entire function (polynomial or transcendental), instead of assuming that the set of singular values of $ f $ in $ \partial U_{+\varepsilon} $ is finite, it is enough to assume that the set of critical values in $ \partial U_{+\varepsilon} $ is finite (see Section \ref{subsect-entire-parabolic}).
\end{remark*}

Next we present two applications of the theorems above:  developing Pesin theory for centered inner functions, and finding periodic points for transcendental maps.

\subsubsection*{Application. Pesin theory for centered inner functions}
 Let $ \mathbb{D} $ denote the unit disk, and $\partial  \mathbb{D} $ the unit circle, and let $ \lambda $ be the normalized Lebesgue measure in $ \partial\mathbb{D} $.
An \textit{inner function} is, by definition, a holomorphic self-map of the unit disk, $ g\colon\mathbb{D}\to\mathbb{D} $, which preserves
 the unit circle $ \lambda $-almost everywhere in the sense of radial limits.  If, in addition, we have that $ g(0)=0 $, we say that the inner function is \textit{centered}. A point $ \xi\in\partial\mathbb{D} $ is called a {\em singularity} of $ g $ if $ g $ cannot be continued analytically to any neighbourhood of $ \xi $. Denote the set of singularities of $ g $ by $ E(g) $.
 
 It is well-known that the radial extension of a centered inner function preserves the Lebesgue measure $ \lambda $ and is ergodic (see e.g. \cite[Thm. A, B]{DoeringMane1991}). For these reasons, centered inner functions have been widely studied as measure-theoretical dynamical systems
 \cite{Aaronson, DoeringMane1991,Craizer,Craizer2, Aaronson97, ivrii2023inner, ivriiurbanskiNOU}. 
 
 An important subset of centered inner functions are the ones with finite entropy, or equivalently, when $ \log \left| g'\right|\in L^1(\partial\mathbb{D})$  \cite{Craizer}. Such a property translates to a greater control on the dynamics, from different points of view (see e.g. \cite{Craizer, Craizer2, ivrii2023inner, ivriiurbanskiNOU}).
In particular, centered  inner functions with finite entropy are natural candidates to apply the theory developed above. Moreover, due to its rigidity and symmetries, we will deduce some additional properties. 
In general, inner functions present a highly discontinuous behaviour in $ \partial\mathbb{D} $, so it is noteworthy the great control we achieve, only by assuming that $ \int_{\partial \mathbb{D}} \log\left| x-SV\right|^{-1} d\lambda(x)<\infty  $, where $ \lambda $ denotes the Lebesgue measure on $ \partial\mathbb{D} $.

Let us denote the \textit{radial segment} at $ \xi $ of length $ \rho>0 $ by \[R_{\rho}(\xi)\coloneqq \left\lbrace r\xi\colon r\in(1-\rho,1)\right\rbrace ,\]
and  the \textit{Stolz angle} at $ \xi $ of length $ \rho>0 $ and opening $ \alpha\in (0, \frac{\pi}{2}) $ by 	\[\Delta_{\alpha, \rho}(\xi)=\left\lbrace z\in\mathbb{D}\colon \left| \textrm{Arg }\xi- \textrm{Arg }(\xi-z)\right| <\alpha, \left| z \right| >1-\rho \right\rbrace .\] 

Using the same construction as in Theorem \ref{teo:A}, we deduce the following.

\begin{maincor}\label{corol:D}{\bf (Pesin theory for centered inner functions)}
		Let $ g\colon\mathbb{D}\to\mathbb{D} $ be a centered inner function,  such that $ \log \left| g'\right|\in L^1(\partial\mathbb{D})$ and $ \int_{\partial \mathbb{D}} \log\left| x-SV\right|^{-1} d\lambda(x)<\infty  $. Fix $ \alpha\in (0,\pi/2) $. 
		Then, for  every countable collection of measurable sets $ \left\lbrace A_k\right\rbrace _k\subset\partial \mathbb{D} $ with $ \lambda(A_k)>0 $, and for $ \lambda$-almost every $ \xi_0\in\partial \mathbb{D} $ there exists a backward orbit $ \left\lbrace \xi_n\right\rbrace _n \subset\partial \mathbb{D}$ and $ \rho_0>0 $ such that
		\begin{enumerate}[label={\em (\alph*)}]
			\item  $ \xi_{n_k}\in A_k$ for some sequence $ n_k\to\infty $;

		\item the inverse branch $ G_n $ of $ g^n $ sending $\xi_0 $ to $ \xi_n $ is well-defined in $ D(\xi_0,\rho_0) $;
		\item for all $ \rho\in (0,\rho_0) $, $ G_n(R_{\rho}(\xi_0))\subset \Delta_{  \alpha, \rho}(\xi_n)$.
	\end{enumerate}
	In particular, the set of singularities $ E(g) $ has zero $ \lambda $-measure.
\end{maincor}

\subsubsection*{Application. Periodic boundary points in transcendental dynamics}
One possible application of Pesin theory is to find periodic points. In complex dynamics, this idea was already exploited by F. Przytycki and A. Zdunik to find periodic points on the boundaries of basins for rational maps \cite{Przytycki-Zdunik}. Hence, we aim to apply Theorems \ref{teo:A} and \ref{teo:B} to find periodic boundary points in the transcendental setting. 
 
 To do so, we need a stronger assumption on the orbits of singular values inside $ U $. Recall that, given a simply connected domain $ U $, we say that $ C\subset U $ is a {\em crosscut}  if $ C $ is a Jordan arc such that $ \overline{C}= C\cup \left\lbrace a,b\right\rbrace  $, with $ a,b\in\partial U $, $ a\neq b $. Any of the two connected components of $ U\smallsetminus C $ is a {\em crosscut neighbourhood}.
We define the 
\textit{postsingular set} of $ f $ as \[ P(f)\coloneqq \overline {\bigcup\limits_{s\in SV}\bigcup\limits_{n\geq 0} f^n(s)}.\]

\begin{maincor}\label{corol:C}{\bf (Periodic boundary points are dense)}
	Under the hypotheses of Theorem \ref{teo:A} or \ref{teo:B}, assume, in addition, that there exists a crosscut neighbourhood $ N_C $ with $ N_C\cap P(f)=\emptyset $. 
	Then, periodic points are dense on $ \partial U $.
\end{maincor}

%Note that, although we need some control on the postsingular set in $ U $, we do not put any restriction on the postsingular set outside $ U $. Observe that the assumptions of Corollary \ref{corol:C} are always satisfied if singular values in $ U $ are compactly contained in $ U $, even if $ f|_{U} $ has infinite degree.

\subsubsection*{Lyapunov exponents of transcendental maps}
Finally, we note that one essential hypothesis  in our results is that $ \log \left| f'\right|  $ is integrable with respect to the harmonic measure $ \omega_U $, and hence the Lyapunov exponent \[\chi_{\omega_U} (f)=\int_{\partial U}\log \left| f'\right| d\omega_U\] is well-defined. We also require that $  \chi_{\omega_U}$ is positive. These facts are well-known for simply connected basins of attraction of rational maps \cite{Przytycki-Hausdorff_dimension, Przytycki-Lyapunovexponents}, but unexplored for transcendental maps. In this paper we give several conditions, concerning the order of growth of the function and the shape of the Fatou component, which imply that the Lyapunov exponent is well-defined and non-negative.

One of the main challenges that appears when considering transcendental maps is that $ \left| f'\right|  $ may not be bounded in $ \partial U $, even when taking the derivative with respect to the spherical metric. Indeed, $ \left| f'\right|  $ is not bounded around the essential singularity, and the growth can be arbitrarily fast. Thus, we introduce the following concept, which relates the growth of the function with the shape of the Fatou component.
\begin{defi*}{\bf (Order of growth in a sector for meromorphic functions)}
	Let $ f\colon\mathbb{C}\to\widehat{\mathbb{C}}$ be a transcendental meromorphic function, and let $ U\subset \mathbb{C} $ be an invariant Fatou component for $ f $. We say that $ U $ is {\em asymptotically contained in a sector of angle $ \alpha\in   \left( 0, 1\right) $ with order of growth $ \beta >0 $} if there exists $ R>0 $, $ \xi\in \partial\mathbb{D}$ and $ \alpha \in (0,1) $, such that, if
	\[ S_R=S_{R, \alpha}\coloneqq \left\lbrace z\in\mathbb{C}\colon \left| z\right| >R , \ \left|\textrm{Arg }\xi-\textrm{Arg }z\right| <\pi \alpha\right\rbrace \]
	then, 
	\begin{enumerate}[label={(\alph*)}]
		\item  $U\cap \left\lbrace z\in\mathbb{C}\colon \left| z\right| >R\right\rbrace \subset S_{R}$;
		\item $ f $ has order of growth $ \beta>0 $ in $  S_{R} $, i.e. there exists $ A,B>0 $ such that for all $ z\in S_{R}$,  \[A \cdot e^{B\cdot \left| z\right| ^{-\beta}}\leq \left| f'(z)\right| \leq A \cdot e^{B\cdot \left| z\right| ^{\beta}}.\]
	\end{enumerate}
\end{defi*}

Under this asusmption on the growth, we are able to prove the following.

\begin{mainprop}\label{prop:D}{\bf ($ \log \left| f'\right|  $ is $ \omega_U $-integrable)}\label{prop-lyapunov-exponents well-defined}
	Let $ f\colon\mathbb{C}\to\widehat{\mathbb{C}}$ be a meromorphic function, and let $ U $ be an invariant Fatou component for $ f $. Let $ \omega_U $ be a harmonic measure on $ \partial U $. Assume $ U $ is asymptotically contained in a sector of angle $ \alpha\in \left( 0, 1\right)  $, with order of growth $ \beta\in(0, \frac1{2\alpha}) $. Then, $ \log \left| f'\right| \in L^1(\omega_U) $. 
\end{mainprop}

\begin{mainprop}\label{prop:E}{\bf (Non-negative Lyapunov exponents)}\label{prop-non-negative-lyapunov}
		Let $ f\colon\mathbb{C}\to\widehat{\mathbb{C}}$ be a meromorphic function, and let $ U $ be a simply connected attracting basin, with fixed point $ p\in U $. Let $ \omega_U $ be the harmonic measure in $ \partial U $ with base point $ p $.  Assume\begin{enumerate}[label={\em (\alph*)}]
		\item $ U $ is asymptotically contained in a sector of angle $ \alpha\in \left( 0, 1\right)  $, with order of growth $ \beta\in(0, \frac1{2\alpha}) $;
		\item $ \int_{\partial U} \log\left| x-SV\right|^{-1} d\omega_U(x)<\infty  $.
	\end{enumerate}
	Then, \[\chi_{\omega_U}(f)=\int_{\partial U}\log \left| f'\right| d\omega_U\geq 0.\]
\end{mainprop}

\begin{remark*}
	The statements of Theorems \ref{teo:A} and \ref{teo:B}, and Corollary \ref{corol:D} are a simplified version of the ones proved inside the paper (respectively, Thms. \ref{thm-pesin}, \ref{thm-pesin-entire} and \ref{thm-pesin-inner}). These stronger statements are formulated in terms of the Rohklin's natural extension of the corresponding dynamical systems. Since this construction is not common in transcendental dynamics (although it is standard in ergodic theory), we chose to present our results in this simplified (and weaker) form.  For convenience of the reader, all the needed results about Rohklin's natural extension can be found in Section \ref{subsect-Rokhlin's-natural-ext}. 
	
	Although the results are stated here for meromorphic functions, we shall work in the more general class $ \mathbb{K} $ of functions with countably many singularities; in particular, this allows us to consider \textit{periodic} attracting basins of meromorphic maps, not only invariant ones. The technicalities that arise when working in class  $ \mathbb{K} $ are explained in Section \ref{subsect-components-Fatou-K}. 
\end{remark*}

\begin{remark*}
	It seems plausible to extend the previous results to multiply connected Fatou components, as long the harmonic measure is well-defined. This is always the case of Fatou components in class $ \mathbb{K} $ \cite{FerreiraJove}.
\end{remark*}

\begin{notation*}
	 Thorughout the paper, $ f^{-1}(z) $ denotes all the preimages under $ f $ of the point $ z $ (setwise). When we refer to the inverse branch, we write $ F_{n,z,w} $ meaning that $ F_{n,z,w}$ is an inverse branch   of $ f^{n} $ sending $ z $ to $ w $. When it is clear from the context, we  just write $ F_n $.
\end{notation*}

\vspace{0.2cm}
{\bf Structure of the paper.} In
Section \ref{section-preliminaries} we gather all the preliminary results used throughout the paper, including results on Rohklin's natural extension. Section \ref{sect3-setting} is devoted to comment on the setting and the hypotheses we work with.  In Sections \ref{section-pesinA} and \ref{sect-pesin-entire}  we develop  Pesin theory, proving Theorems \ref{teo:A} and \ref{teo:B}, respectively.  Corollary \ref{corol:D} is proven in Section \ref{sect-inner-functions}, while Corollary \ref{corol:C} is proven in Section \ref{sect-periodic-points}. Section \ref{sect-lyapunov} deals with Lyapunov exponents, proving Propositions \ref{prop:D} and \ref{prop:E}.

\vspace{0.2cm}
{\bf Acknowledgments.} First of all, I am indebted to my supervisor, Núria Fagella.
I also want to thank Lasse Rempe, for asking me the question which motivates this work. Besides, I am indebted with Anna Zdunik, for all her explanations and her encouragement for starting this project, as well as with Oleg Ivrii, for   interesting discussions and his valuable insights, and for pointing out the right assumption in Theorem \ref{teo:A}.
I also want to thank Phil Rippon for his help and his kindness for sharing with me  some of his knowledge about harmonic measure. I am also thankful to Jana Rodríguez-Hertz and the KTH in Stockholm, for a master class in Pesin theory, and Nikolai Prochorov and the Séminaire Rauzy  in Marseille, for interesting discussions on the topic.

\section{Preliminaries}\label{section-preliminaries}

In this section we gather the tools we use throughout the article, putting special emphasis on Rokhlin's natural extension.  %Although all the results in this section seem to be well-known, we include the proof of those for which we could not find a written reference.

\subsection{Distortion estimates for univalent maps}

We need the following result concerning the distortion for univalent maps.

\begin{thm}{\bf (Koebe's distortion estimates, {\normalfont \cite[Sect. 23.1]{UrbanskiRoyMunday3}})} \label{thm-Koebe} Let $ z\in\mathbb{C} $, $ r>0 $, and let $ \varphi\colon D(z,r)\to\mathbb{C} $ be a univalent map. Then, 
	\[ D\left( \varphi(z), \frac{1}{4}\cdot \left| \varphi'(z)\right| \cdot r\right)  \subset\varphi(D(z,r)).\]
Moreover, for all $ \lambda\in (0,1) $ and $ z\in \overline{D(x,\lambda r)} $, it holds
	\[\left| \varphi'(x)\right| \cdot \dfrac{1-\lambda}{(1+\lambda)^3}\leq\left| \varphi'(z)\right| \leq  { \left| \varphi'(x)\right| }\cdot \dfrac{1+\lambda }{(1-\lambda )^3},\]
\[  \varphi(D(x,\lambda r))\subset D\left( \varphi(x), r\cdot { \left| \varphi'(x)\right| }\cdot \dfrac{1+\lambda }{(1-\lambda )^3}\right) .\]
\end{thm}

\subsection{Abstract Ergodic Theory}
We recall some basic notions used in abstract ergodic theory and measure theory (for more details, see e.g. \cite{PrzytyckiUrbanski, Hawkins, UrbanskiRoyMunday1}).

\begin{lemma}{\bf (First Borel-Cantelli lemma, {\normalfont \cite[1.12.89]{Bogachev}})}\label{lemma-borel-cantelli}
	Let $ (X,\mathcal{A}, \mu) $ be a probability space, let $ \left\lbrace A_n\right\rbrace _n\subset \mathcal{A} $, and let \[B\coloneqq \left\lbrace x\in X\colon x\in A_n\textrm{ for infinitely many }n\textrm{'s}\right\rbrace =\bigcap\limits_{k=1}^\infty\bigcup_{n=k}^\infty A_n .\] Then, if $ \sum_{n=1}^{\infty}\mu (A_n)<\infty $, it holds $ \mu (B)=0 $.
\end{lemma}

\begin{defi}{\bf (Ergodic properties of measurable maps)}\label{def-ergodic-properties} Let $ (X,\mathcal{A}, \mu) $ be a measure space, and let $ T\colon X\to X $ be measurable. Then,\begin{itemize}
			\item $ T $ is {\em non-singular}, if, for every $ A\in\mathcal{A} $, it holds $ \mu (T^{-1}(A))=0 $ if and only if  $ \mu (A)=0 $;
		\item $ \mu $ is {\em $ T $-invariant} if $ T $ is measure-preserving, i.e. if $ \mu (T^{-1}(A))=\mu (A) $, for every $ A\in\mathcal{A} $;
		\item $ T $ is {\em recurrent} with respect to $ \mu $, if for every  $ A\in\mathcal{A} $ and $ \mu $-almost every $ x\in A $, there exists a sequence $ n_k\to \infty $ such that $ T^{n_k}(x)\in A $;
		\item $ T $ is {\em ergodic} with respect to $ \mu $, if $ T $ is non-singular and for every $ A\in\mathcal{A} $ with $ T^{-1}(A)=A $, it holds $ \mu(A)=0$ or $ \mu (X\smallsetminus A)=0 $.
	\end{itemize}
\end{defi}
In the sequel, if it is clear with which measure are we working with, we omit the dependence on the measure. Note that if $ T $ is invertible and ergodic, then $ T^{-1} $ is also ergodic.

\begin{thm}{\bf (Almost every orbit is dense, {\normalfont\cite[Prop. 1.2.2]{Aaronson97}})}\label{thm-almost-every-orbit-dense}
	Let $ (X,\mathcal{A}, \mu) $ be a measure space, and let $ T\colon X\to X $ be non-singular. Then, the following are equivalent.\begin{enumerate}[label={\em (\alph*)}]
		\item $ T $ is ergodic and recurrent with respect to $ \mu $.
		\item\label{thm-orbitesdenses-b}  For every  $ A\in\mathcal{A} $ with $ \mu (A)>0 $ and $ \mu $-almost every $ x\in X $, there exists a sequence $ n_k\to \infty $ such that $ T^{n_k}(x)\in A $.
	\end{enumerate}
\end{thm}

\begin{thm}{\bf (Birkhoff Ergodic Theorem, {\normalfont \cite[Sect. 4.1]{KatokHasselblat}})} \label{thm-birkhoff} Let $ (X,\mathcal{A},\mu) $ be a probability space together with a measure-preserving transformation $ T\colon X\to X $, and let $ \varphi\in L^1(\mu) $. Then, \[\lim\limits_{n}\frac{1}{n}\sum_{k=0}^{n-1}\varphi(T^k(x))\] exists for $ \mu $-almost every $ x\in X $. If $ T $ is an automorphism, the equality \[\lim\limits_{n}\frac{1}{n}\sum_{k=0}^{n-1}\varphi(T^k(x))=\lim\limits_{n}\frac{1}{n}\sum_{k=0}^{n-1}\varphi(T^{-k}(x))\] holds $ \mu $-almost everywhere.
	 Finally, if $ T $ is ergodic with respect to $ \mu $, then for $ \mu $-almost every $ x\in X $ it holds \[\lim\limits_{n}\frac{1}{n}\sum_{k=0}^{n-1}\varphi(T^k(x))=\int_X\varphi  d\mu.\]
\end{thm}

\subsection{Rokhlin's natural extension}\label{subsect-Rokhlin's-natural-ext}

A useful technique in the study of non-invertible measure-preserving tranformations is the so-called Rokhlin's natural extension \cite{Rohlin}, which allows us to construct a measure-preserving automorphism in an abstract measure space, mantaining its ergodic properties. However, this technique is often developed for {Lebesgue spaces} with invariant probabilities (see e.g \cite[Sect. 1.7]{PrzytyckiUrbanski} , \cite[Sect. 8.5]{UrbanskiRoyMunday1}). Since we  work also with $ \sigma $-finite measures, we sketch how we can develop the theory  in this more general case.

Let $ (X,\mathcal{A}, \mu) $ be a {\em Lebesgue space}, i.e. a  measure space isomorphic (in the measure-theoretical sense) to an interval (equipped with the Lebesgue measure) together with countably many atoms. Let   $ T\colon X\to X $ be measure-preserving. The measure $ \mu $ is either finite (and  we assume it is a probability measure), or  $ \sigma $-finite. 

Consider the space of backward orbits for $ T $
\[\widetilde{X}=\left\lbrace \left\lbrace x_n\right\rbrace _n\subset X\colon\hspace{0.15cm} x_0\in X, \hspace{0.15cm}T(x_{n+1})=x_n ,\hspace{0.15cm} n\geq0\right\rbrace,\] and define, in a natural way, the following  maps. On the one hand, for $ k\geq 0 $, let 
$ \pi_k\colon \widetilde{X}\to X $ be the projection on the $ k $-th coordinate of $ \left\lbrace x_n\right\rbrace _n $, that is $ \pi_k(\left\lbrace x_n\right\rbrace _n) =x_k$. On the other hand, we define {\em Rokhlin's natural extension} of $ T $ as 
$ \widetilde{T}\colon\widetilde{X}\to\widetilde{X} $, with \[\widetilde{T}(\left\lbrace x_n\right\rbrace _n)= \widetilde{T}(x_0x_1x_2\dots)= T(x_0)x_0x_1\dots\] 
It is clear that $  \widetilde{T}$ is invertible and $ \widetilde{T}^{-1} $ is the shift-map, i.e. \[\widetilde{T}^{-1}(\left\lbrace x_n\right\rbrace _n)= \widetilde{T}^{-1}(x_0x_1x_2\dots)= x_1x_2x_3\dots=\left\lbrace x_{n+1}\right\rbrace _n.\]
Moreover, for each $ k\geq 0 $, the following diagram commutes.
\[\begin{tikzcd}[row sep=-0.2em,/tikz/row 3/.style={row sep=4em}]
	\widetilde{X} \arrow{rr}{\widetilde{T}} & &	\widetilde{X} \\ 
	{\scriptstyle \left\lbrace x_{n+1}\right\rbrace _n} \arrow{ddddddd}{\pi_{k}}& &{\scriptstyle \left\lbrace x_{n}\right\rbrace _n} \arrow{ddddddd}{\pi_{k}} \\ \\ \\	\\ \\ \\ \\ X \arrow{rr}{T}  &&
X \\ {\scriptstyle x_{k+1} }&&{\scriptstyle x_{k} }
\end{tikzcd}
\]
Note that, up to here, the construction is purely symbolic and measures have not appeared yet. In fact, the next step in the construction is to endow the space $ \widetilde{X} $ with an appropriate $ \sigma $-algebra $ \widetilde{\mathcal{A}} $ and a measure $ \widetilde{\mu}$, which makes the previous projections $ \pi_k $ and the map $ \widetilde{T} $ measure-preserving. To do so, we will need the following more general result.

\begin{thm}{\bf (Kolmogorov Consistency Theorem, {\normalfont \cite[Thm. V.3.2]{Parthasarathy}})}\label{thm-consistency}
	Let $ (X_n, \mathcal{A}_n, \mu_n) $ be Lebesgue probability  spaces, and let $ T_{n}\colon X_{n+1}\to X_n $ be measure-preserving. Let \[\widetilde{X}=\left\lbrace \left\lbrace x_n\right\rbrace _n\colon\hspace{0.15cm} x_n\in X_n, \hspace{0.15cm}T_n(x_{n+1})=x_n ,\hspace{0.15cm} n\geq0\right\rbrace.\] 
	and let $ \pi_k\colon \widetilde{X}\to X_k $ be the projection on the $ k $-th coordinate. Then, there exists a $ \sigma $-algebra $ \widetilde{A} $ and a probability measure $ \widetilde{ \mu} $ in $ \widetilde{X} $ such that  $ (\widetilde{X},  \widetilde{A}, \widetilde{ \mu})  $ is a Lebesgue probability space and, for each $ k\geq 0 $, 
	$$ \widetilde{\mu}(\pi^{-1}_k(A))=\mu_k(A) , \hspace{0.5cm} A\in \mathcal{A}_k.$$
\end{thm}

Notice that the theorem above holds whenever $ (X_n, \mathcal{A}_n, \mu_n) $ are Lebesgue measure spaces with finite measure.
The $ \sigma $-algebra $  \widetilde{A}$ can be taken to be the smallest which makes each projection $ \pi_k\colon \widetilde{X}\to X_k $ measurable  \cite[Thm. V.2.5]{Parthasarathy}. Note that $ T_{k}\circ\pi_{k+1}= \pi_k$. Observe that now $ \widetilde{X} $ stands for the space of backward orbits under the sequence of maps $ \left\lbrace T_n\right\rbrace _n $. Hence, one has to think of $ \widetilde{X} $ as the infinite product of the spaces $ \left\lbrace X_n\right\rbrace _n $, since the spaces in $ \left\lbrace X_n\right\rbrace _n $ are {\em a priori} different, and hence there is no endomorphism $ \widetilde{T}\colon\widetilde{X}\to\widetilde{X } $ in general. However, we will use these extensions $ (\widetilde{X},  \widetilde{A}, \widetilde{ \mu})  $ of some appropriate spaces as building blocks for Rokhlin's natural extension for transformations with $ \sigma $-finite invariant measures.

\begin{thm}{\bf (Rokhlin's natural extension for $ \sigma $-finite invariant measures)}\label{thm-natural-extension}
 let $ (X,\mathcal{A}, \mu) $ be a Lebesgue space, and  let $ T\colon X\to X $ be a measure-preserving transformation. Assume $ \mu $ is a $ \sigma $-finite measure, and consider Rokhlin's natural extension $ \widetilde{T}\colon\widetilde{X}\to\widetilde{X} $. Then, there exists a $ \sigma $-algebra $\widetilde{ \mathcal{A}} $ and a $ \sigma $-finite measure $ \widetilde{\mu} $ such that the maps $ \pi_k $ and $ \widetilde{T} $ are measure-preserving. 
\end{thm}
\begin{proof}
	In the case of $ (X,\mathcal{A}, \mu) $ being a Lebesgue probability space, the statement follows from applying the Kolmogorov Consistency Theorem \ref{thm-consistency} with $ X_n=X $, for all $ n\geq0 $, as indicated in \cite[Thm. 8.4.2]{UrbanskiRoyMunday1}. 
	
	Otherwise,  let $ \left\lbrace X^j_0\right\rbrace _j $ be a partition of $ X $ such that $ \mu (X_0^j) $ is finite, for each $ j\geq 0 $. Without loss of generality, we assume $ \mu (X_0^j)=1 $, for each $ j\geq 0 $, to simplify the computations.
	Then, for all $ n\geq 0 $, $ \left\lbrace X_n^j\coloneqq T^{-n}(X_0^j)\right\rbrace _j $ is also a partition of $ X $ such that $ \mu (X_n^j)=1 $, for each $ j\geq 0 $, since $ T $ is measure-preserving and preimages of disjoint sets are disjoint. 
	
	If we write $ \mathcal{A}_n^j $ and  $ \mu_n^j $ for the restrictions of $ \mathcal{A} $ and $ \mu $ to $ X_n^j $, we have that, for each $ j \geq 0$, $ (X_n^j , \mathcal{A}_n^j, \mu_n^j ) $ is a Lebesgue probability space, and $ T\colon X^j_{n+1}\to X^j_n $ is measure-preserving. Hence, by Theorem \ref{thm-consistency}, there exists a  Lebesgue probability space $ (\widetilde{X^j }, \widetilde{\mathcal{A}^j}, \widetilde{\mu^j }) $ such that 
	\[\widetilde{X^j}=\left\lbrace \left\lbrace x_n\right\rbrace _n\colon\hspace{0.15cm} x_n\in X^j_n, \hspace{0.15cm}T(x_{n+1})=x_n ,\hspace{0.15cm} n\geq0\right\rbrace.\] 
	and the projections $ \pi^j_k\colon \widetilde{X^j}\to X_k ^j$ are measure-preserving. The space of backward orbits \[\widetilde{X}=\left\lbrace \left\lbrace x_n\right\rbrace _n\colon\hspace{0.15cm} x_n\in X, \hspace{0.15cm}T_n(x_{n+1})=x_n ,\hspace{0.15cm} n\geq0\right\rbrace\] 
 is the disjoint union of the $  \widetilde{X^j}$, $ j\geq 0 $. Let $ \mathcal{A} $ to be the $ \sigma $-algebra generated by $ \left\lbrace \widetilde{\mathcal{A}^j}\right\rbrace _j $, and the measure $ \widetilde{\mu} $ on $ (\widetilde{X}, \widetilde{A}) $ unambiguously determined by the $ \widetilde{\mu^j }$'s. It is clear that the maps $ \pi_k $ preserve the measure $ \widetilde \mu $, for all $ k\geq 0 $.
 	
 	It is left to see that $ \widetilde{T} $ is measure-preserving. To do so, note that we have the following measure-preserving commutative diagram.
 	\[\begin{tikzcd}[row sep=-0.2em,/tikz/row 3/.style={row sep=4em}]
 		...\arrow{r}{\widetilde{T}} &	\widetilde{X^{j_2}}\subset \widetilde{X} \arrow{rr}{\widetilde{T}} & &		\widetilde{X^{j_1}}\subset \widetilde{X}\arrow{rr}{\widetilde{T}} &&\widetilde{X^{j_0}}\subset \widetilde{X} \\ 	 
 		&	{\scriptstyle \left\lbrace x_{n+2}\right\rbrace _n} \arrow{ddddddd}{\pi_{0}^{j_2}}& &	{\scriptstyle \left\lbrace x_{n+1}\right\rbrace _n }\arrow{ddddddd}{\pi_{0}^{j_1}} & &	{\scriptstyle \left\lbrace x_{n}\right\rbrace _n }\arrow{ddddddd}{\pi_{0}^{j_0}}\\ \\ \\	\\ \\ \\ \\...\arrow{r}{T} & X^{j_2}_{2}\subset X \arrow{rr}{T}   && X^{j_1}_1\subset X \arrow{rr}{T} &&X^{j_0}_0\subset X \\ &{\scriptstyle x_2}&&{\scriptstyle x_1} &&{\scriptstyle x_0}
 	\end{tikzcd}
 	\]
 Since the sets $ \left\lbrace (\pi_n^j)^{-1}(A\cap X^j_n)\colon A\in \mathcal{A}\right\rbrace _{n,j} $ generate the $ \sigma $-algebra $ \widetilde{A} $,  it is enough to prove invariance for such sets. Thus, without loss of generality, let $ A\subset  X^j_n$, and then \begin{align*}
 	\widetilde{\mu} \circ \widetilde{T}^{-1}( (\pi_n^j)^{-1}(A))&= \widetilde{\mu} \circ (\pi_n^j\circ \widetilde{T})^{-1}(A)= \widetilde{\mu} \circ (T\circ \pi_n^{j})^{-1}(A)= \widetilde{\mu} \circ(\pi_n^{j})^{-1}\circ T^{-1}(A)\\&=\mu (T^{-1}(A))= \mu (A)= \widetilde{ \mu} ( (\pi_n^j)^{-1}(A)),
 \end{align*}
 as desired.

\end{proof}
It follows from the previous theorem that $ \widetilde{\mu} $ is a probability measure if and only if  $ \mu $ is also. Natural extensions share many ergodic properties with the original map, as shown in the following proposition for probability spaces.  

\begin{prop}{\bf (Ergodic properties of Rokhlin's natural extension)}\label{prop-ergodic-properties-natural-ext}
	Let $ (X,\mathcal{A},\mu) $ be a Lebesgue probability space, endowed with a measure-preserving transformation $ T\colon X \to X $, and consider its Rokhlin's natural extension $ \widetilde{T} $ acting on $ (\widetilde{X},\widetilde{\mathcal{A}},\widetilde{\mu}) $, given by Theorem \ref{thm-natural-extension}. Then, the following holds.
	\begin{enumerate}[label={\em (\alph*)}]
		\item 	$ \widetilde{T} $ is recurrent with respect to $ \widetilde{\mu} $.
		\item 	$ \widetilde{T} $ is ergodic with respect to $ \widetilde{\mu} $ if and only if $ {T} $ is ergodic with respect to $ {\mu} $.
	\end{enumerate}

\end{prop}
\begin{proof} Since $ \mu $ is assumed to be a probability measure,  $ \widetilde{\mu} $ is also a probability measure, and the recurrence of $ \widetilde{T} $ follows from Poincaré Recurrence Theorem, since $ \widetilde T $ is measure-preserving. For (b), see {\normalfont \cite[Thm. 8.4.3]{UrbanskiRoyMunday1}}.
\end{proof}

Under the assumption of ergodicity and recurrence, we can prove that every subset of positive measure in the phase space is visited by almost every backward orbit.
\begin{corol}{\bf (Almost every backward orbit is dense)}\label{prop-ergodic-properties-natural-ext2}
		Let $ (X,\mathcal{A},\mu) $ be a Lebesgue space, endowed with a measure-preserving transformation $ T\colon X \to X $, and consider its Rokhlin's natural extension $ \widetilde{T} $ acting in $ (\widetilde{X},\widetilde{\mathcal{A}},\widetilde{\mu}) $, given by Theorem \ref{thm-natural-extension}. Assume $ \widetilde{T} $  is ergodic and recurrent with respect to $ \widetilde{\mu} $,  and  $ A\subset X $ is a measurable set  with $ \mu(A)>0 $. Then, for $ \widetilde{\mu} $-almost every $ \left\lbrace x_n\right\rbrace _n\in \widetilde{X }$, there exists a sequence $ n_k\to \infty $ such that $ x_{n_k}\in A $ for all $k\in\mathbb{N}$.   
\end{corol}
\begin{proof}
	Since $ \widetilde{T} $  is ergodic and recurrent with respect to $ \widetilde{\mu} $,  by Theorem \ref{thm-almost-every-orbit-dense}, for every  $ \widetilde{A}\in\widetilde{\mathcal{A}} $ with $ \widetilde{\mu} (\widetilde{A})>0 $ and $ \widetilde{\mu} $-almost every $ \left\lbrace x_n\right\rbrace _n\in \widetilde{X }$, there exists a sequence $ n_k\to \infty $ such that $ \widetilde{T}^{-n_k}(\left\lbrace x_n\right\rbrace _n)\in \widetilde{A} $. Taking $ \widetilde{A} $ to be $ \pi_0^{-1}(A) $,  we have that $ \widetilde{\mu} (\widetilde{A})>0 $,  so for $ \widetilde{\mu} $-almost every $ \left\lbrace x_n\right\rbrace _n\in \widetilde{X }$, there exists a sequence $ n_k\to \infty $ with $$ \widetilde{T}^{-n_k}(\left\lbrace x_n\right\rbrace _{n\geq 0})= \left\lbrace x_n\right\rbrace _{n\geq n_k} \in \pi_0^{-1}(A)  .$$ Hence, $ x_{n_k}\in A $, as desired.
\end{proof}
\subsection{Harmonic measure}\label{subsect_harmonic_measure}
Let $ U \subset\widehat{\mathbb{C}}$ be a hyperbolic simply connected domain (i.e. $ U $ omits at least three points), and let 
 $ \varphi\colon\mathbb{D}\to U $ be a Riemann map. We are concerned with the extension of $ \varphi $ to the unit circle $ \partial\mathbb{D} $ given in terms of radial limits \[\varphi^*(\xi)\coloneqq \lim\limits_{t\to 1^-}\varphi(t\xi) ,\] which exist $ \lambda $-almost everywhere. The radial extension of its Riemann map  $\varphi^*\colon \partial \mathbb{D}\to \widehat{\partial} U$  is used to define a measure in $ \widehat{\partial} U $, the {\em harmonic measure},  in terms of the push-forward of the normalized Lebesgue measure on the unit circle $ \partial\mathbb{D} $.

\begin{defi}{\bf (Harmonic measure)}
	Let $ U\subset\widehat{\mathbb{C}} $ be a hyperbolic simply connected domain, $ z\in U $, and let $ \varphi\colon\mathbb{D}\to U $ be a Riemann map, such that $ \varphi(0)=z\in U $. Let $ (\partial\mathbb{D}, \mathcal{B}, \lambda) $ be the measure space on $ \partial \mathbb{D} $ defined by $ \mathcal{B} $, the Borel $ \sigma $-algebra of $ \partial \mathbb{D} $, and $ \lambda $, its normalized Lebesgue measure. Consider the measurable space $ (\widehat{\mathbb{C}}, \mathcal{B}(\widehat{\mathbb{C}})) $, where  $ \mathcal{B}(\widehat{\mathbb{C}}) $ is the Borel $ \sigma $-algebra of $ \widehat{\mathbb{C}} $. Then, given $ B\in  \mathcal{B}(\widehat{\mathbb{C}}) $, the \textit{harmonic measure at $ z $ relative to $ U $} of the set $ B$ is defined as\[\omega_U(z, B)\coloneqq\lambda ((\varphi^*)^{-1}(B)).\]
\end{defi}

We refer to \cite{harmonicmeasure2, Pommerenke} for equivalent definitions and further properties of the harmonic measure.

Let $ U\subset\widehat{\mathbb{C}} $ be a hyperbolic simply connected domain. We need the following simple facts.
\begin{itemize}
	\item 	Let $ B\in\mathcal{B}(\widehat{\mathbb{C}}) $.  If there exists $ z_0\in U $ such that $ \omega_U(z_0, B)=0 $ (resp. $ \omega_U(z_0, B)=1 $), then $ \omega_U(z, B)=0 $ (resp. $ \omega_U(z, B)=1 $) for all $ z\in U $.
	In this case, we say that the set $ B $ has {\em zero} (resp. {\em full}) {\em harmonic measure relative to $ U $}, and we write $ \omega_U(B)=0 $ (resp. $ \omega_U(B)=1 $).
	
	\item 	$\textrm{supp }\omega_U =\widehat{\partial }U$.
	That is, for all $ x\in\widehat{\partial} U $ and $ r>0 $, $ \omega_U(D(x,r))>0 $. 
	
	\item Harmonic measure is invariant under Möbius transformations. That is, if  consider the harmonic measure $ \omega_U(z, \cdot) $ and
	$ M $ is a Möbius transformation, $ M(U) $ is a hyperbolic simply connected domain, and, for all $ B\in\mathcal {B}(\widehat{\mathbb{C}})$ it holds
	\[\omega_U(z, B)=\omega_{M(U)}(M(z), M(B)).\]
\end{itemize}

\subsection{Inner functions. Inverse branches and distortion}
An {\em inner function} is, by definition, a holomorphic self-map of the unit disk, $ g\colon\mathbb{D}\to\mathbb{D} $, which preserves the unit circle $ \lambda $-almost everywhere in the sense of radial limits, i.e. \[g^*(\xi)\coloneqq \lim\limits_{t\to 1^-}g(t\xi) \in\partial\mathbb{D},\] for $ \lambda $-almost every $ \xi\in\partial\mathbb{D} $. If $ g $ is not conjugate to a rotation, it is well-known that iterates in $ \mathbb{D} $ converge locally uniformly to a distinguished point $ p\in\overline{\mathbb{D}} $, the {\em Denjoy-Wolff point}.

We need precise estimates on the distortion of the radial segment in terms of Stolz angles under inverse branches of an inner function (we refer to \cite{Jove} for a wider explanation).

\subsubsection*{Generalized radial arcs and Stolz angles}
Denote the \textit{radial segment} at $ \xi $ of length $ \rho>0 $ by $R_{\rho}(\xi)\coloneqq \left\lbrace r\xi\colon r\in(1-\rho,1)\right\rbrace$,
and  the \textit{Stolz angle} at $ \xi $ of length $ \rho>0 $ and opening $ \alpha\in (0, \frac{\pi}{2}) $ by 	\[\Delta_{\alpha, \rho}(\xi)=\left\lbrace z\in\mathbb{D}\colon \left| \textrm{Arg }\xi- \textrm{Arg }(\xi-z)\right| <\alpha, \left| z \right| >1-\rho \right\rbrace .\] 
A more flexible notion of radial segment and Stolz angle will be needed for our purposes.

\begin{defi}{\bf (Generalized radial arc and  Stolz angle) } \label{defi-radi-stolz-angle} Let $ p\in\overline{\mathbb{D}} $ and let $ \xi\in\partial\mathbb{D} $, $ \xi\neq p $. Let $ \rho>0 $ and $ \alpha\in (0, \pi/2 )$. \begin{itemize}
		\item If $ p\in \mathbb{D}$, consider the Möbius transformation $ M\colon\mathbb{D}\to\mathbb{D} $, $ M(z)=\dfrac{p-z}{1-\overline{p}z} $. Then,  the {\em (generalized) radial segment} $ R_{\rho}(\xi,p) $ of length $ \rho $ at $ \xi $ is defined as the preimage under $ M $ of the radial segment $ R_\rho(M(\xi)) $. Analogously, the {\em (generalized) Stolz angle} $ \Delta_{\alpha, \rho}(\xi, p) $ of angle $ \alpha $ and length $ \rho $  is  the preimage under $ M $ of the Stolz angle $ \Delta_{\alpha, \rho}(M(\xi)) $.
		That is, \[ R_{\rho}(\xi, p)\coloneqq M^{-1}(R_{\rho}(M(\xi))),\]\[ \Delta_{\alpha, \rho}(\xi, p)\coloneqq M^{-1}(\Delta_{\alpha,\rho}(M(\xi))).\]
		
		\item 	If $ p\in\partial\mathbb{D} $, consider the Möbius transformation $ M\colon\mathbb{D}\to\mathbb{H} $, $ M(z)=i\dfrac{p+z}{p-z} $. Then, the {\em (generalized) radial segment} and {\em Stolz angle} at $ \xi $ are defined as the preimages of the corresponding radial segment and Stolz angle at $ M(\xi) \in\mathbb{R}$. That is, \[ R_{\rho}(\xi,p)\coloneqq M^{-1}(R_{\rho}^\mathbb{H}(M(\xi)))\] \[ \Delta_{\alpha, \rho}(\xi,p)\coloneqq M^{-1}(\Delta_{\alpha,\rho}^\mathbb{H}(M(\xi))).\]
	\end{itemize}
\end{defi}

\begin{remark} See \cite[Sect. 2.3]{Jove} for details. In particular, the Lehto-Virtanen Theorem \cite[Sect. 4.1]{Pommerenke} justifies that, for meromorphic maps omitting three values, it is equivalent to take the limit along the radial segment, than along any generalized radial segment. Likewise, the angular limit can be computed along generalized Stolz angles. 
\end{remark}
\subsubsection*{Inverse branches for inner functions and distortion}
Given an inner function $ g\colon\mathbb{D}\to\mathbb{D} $, consider it as its maximal meromorphic extension \[g\colon \widehat{\mathbb{C}}\smallsetminus E(g)\to \widehat{\mathbb{C}}.\]The set of singular values of an inner function is determined by those in $ \mathbb{D} $, and the distortion at regular points is controlled, as shown in Proposition \ref{prop-radial-limits}.

\begin{prop}{\bf (Inverse branches of inner functions, {\normalfont\cite[Prop. 4.5]{Jove}})}\label{prop-radial-limits}
	Let $ g\colon\mathbb{D}\to\mathbb{D} $ be an inner function, with Denjoy-Wolff point $ p\in\overline{\mathbb{D}} $. Let $ \xi\in\partial\mathbb{D} $, $ \xi\neq p $.
	Assume there exists a a crosscut $ C $, with crosscut neighbourhood $ N_C $ and $ \xi\in \partial N_C $ such that $ SV(g)\cap N_C =\emptyset$. Then, there exists $ \rho_0>0$ such that $ D(\xi, \rho_0)\cap SV(g)\neq 0 $.  Moreover, for all $ 0<\alpha< \frac{\pi}{2}$, there exists $ \rho_1\coloneqq \rho_1(\alpha, \rho_0)<\rho_0 $ such that all branches $ G_1 $ of $ g^{-1} $ are well-defined in $ D(\xi, \rho_1) $ and, for all $ \rho<\rho_1 $,\[G_1(R_\rho(\xi, p))\subset \Delta_{\alpha,\rho}(G_1(\xi), p),\] where $ R_\rho(\cdot, p) $ and $ \Delta_{\alpha,\rho}(\cdot, p) $ stand for the generalized radial segment and Stolz angle with respect to $ p $ {\em (Def. \ref{defi-radi-stolz-angle})}. 
\end{prop}

\section{On the setting and the hypotheses}\label{sect3-setting}
In this section we shall discuss the setting we are working with (that is, iteration of functions of class $ \mathbb{K} $, and the dynamics on the boundary of Fatou components), and the hypothesis we use in our version of Pesin theory, that is  $ \int_{\partial U} \log\left| x-SV\right|^{-1} d\omega_U(x)<\infty  $. 
\subsection{Fatou components of functions of class $ \mathbb{K} $}\label{subsect-components-Fatou-K}

As mentioned in the introduction, consider $ f\in\mathbb{K} $, i.e. \[f\colon \widehat{\mathbb{C}}\smallsetminus E(f)\to \widehat{\mathbb{C}},\] where $ \Omega(f)\coloneqq\widehat{\mathbb{C}}\smallsetminus E(f)  $ is the largest set where $ f $ is meromorphic and $ E(f) $ is the set of singularities of $ f $, which is assumed to be closed and countable. Note that $ \Omega(f)$ is open. 

\begin{notation*}
	Once a function $ f\in\mathbb{K} $ is fixed, we denote $ \Omega(f) $ and $ E(f) $ simply by $ \Omega $ and $ E $, respectively. Given a domain $ U\subset\Omega $, we denote by $ \partial U $ the boundary of $ U $ in $ \Omega $, and we keep the notation $ \widehat{\partial} U $ for the boundary with respect to $ \widehat{\mathbb{C}} $.
\end{notation*}

The dynamics of functions in class $ \mathbb{K} $ was introductedin \cite{Bolsch-thesis, BakerDominguezHerring}. The Fatou set $ \mathcal{F}(f) $ is defined as the largest open set in which $ \left\lbrace f^n\right\rbrace _n $ is well-defined and normal, and the Julia set $ \mathcal{J}(f) $, as its complement in $ \widehat{\mathbb{C}} $. The standard theory of Fatou and Julia for rational or entire functions extends successfully to this more general setting. We shall need the following properties.

\begin{thm}{\bf (Properties of Fatou and Julia sets, {\normalfont \cite[Thm. A]{BakerDominguezHerring}})}
	Let $ f\in\mathbb{K} $. Then,
	\begin{enumerate}[label={\em (\alph*)}]
		\item\label{Fatou1} $ \mathcal{F}(f) $ is completely invariant in the sense that $ z\in\mathcal{F}(f) $ if and only if $ f(z)\in \mathcal{F}(f) $;
		\item\label{Fatou2} for every positive integer $ k $, $ f^k\in\mathbb{K} $, $ \mathcal{F}(f^k)=\mathcal{F}(f) $ and $ \mathcal{J}(f^k)=\mathcal{J}(f) $.
	\end{enumerate}
\end{thm}

By {\em\ref{Fatou1}}, Fatou components (i.e. connected components of $ \mathcal{F}(f) $) are mapped among themselves, and hence classified into periodic, preperiodic or wandering. By {\em\ref{Fatou2}}, the study of periodic Fatou components reduces to the invariant ones, i.e. those for which $ f(U)\subset U $.

In this paper, we focus on simply connected periodic Fatou components, which we assume to be invariant. 
Those Fatou components are classified into attracting basins, parabolic basins, Siegel disks,  and Baker domains \cite[Thm. C]{BakerDominguezHerring}. A \textit{Baker domain} is, by definition, a periodic Fatou component $ U $ of period $ k\geq 1 $ for which there exists $ z_0\in\widehat{\partial} U$ such that $ f^{nk}(z)\to z_0 $, for all $ z\in U $ as $ n\to\infty $, but $ f^k $ is not meromorphic at $ z_0 $. In such case, $ z_0 $ is accessible from $ U $  \cite[658]{BakerDominguezHerring}. Baker domains are classified according to its internal dynamics in doubly parabolic, hyperbolic and simply parabolic (see e.g. \cite{fh}).

\subsection*{Ergodic properties of the boundary map $ f\colon \partial U\to \partial U $}\label{subsection-ergodic-f}

To prove Theorems \ref{teo:A} and \ref{teo:B} we shall need the following ergodic results about the boundary map $ f\colon \partial U\to \partial U $. More precisely, let $ f\in\mathbb{K} $, and let $ U $ be an invariant Fatou component for $ f $, which we assume to be simply connected. Consider $ \varphi\colon\mathbb{D}\to U $ to be a Riemann map. Then,  $ f\colon U\to U $ is conjugate by $ \varphi $ to a holomorphic map $ g\colon\mathbb{D}\to\mathbb{D} $, i.e. the diagram
\[\begin{tikzcd}[row sep=large, column sep=large]
	U \arrow{r}{f} & U \\	
	\mathbb{D} \arrow{r}{g} \arrow{u}{\varphi} & \mathbb{D}\arrow[swap]{u}{\varphi}
\end{tikzcd}
\]
commutes.  
It is well-known that  $ g $ is an inner function (see e.g. \cite[Prop. 5.6]{Jove}).
We say that $ g $ is an {\em inner function associated to $ (f,U) $}. 
 
Ergodic properties of $ g^* |_{\partial\mathbb{D}}$, as well as their extension to $ f|_{\partial U} $,  have been widely studied \cite{Aaronson, DoeringMane1991, bfjk-escaping, Jove}. The following theorem summarizes these well-known results (see also \cite[Thm. 5.7]{Jove} and  references therein).

\begin{thm}{\bf (Ergodic properties of the boundary map)}\label{thm-ergodic-boundary-map}
	Let $ f\in\mathbb{K}$, and let $ U $ be an invariant simply connected Fatou component for $ f $. Let  $ g $ be an inner function associated to $ (f, U) $. Then, the following are satisfied. \begin{enumerate}[label={\em (\roman*)}]
		\item If $ U $ is either an attracting basin, a parabolic basin, or a Siegel disk, then $ g^*|_{\partial\mathbb{D}} $ is ergodic and recurrent with respect to the Lebesgue measure $ \lambda$. 
		\item If $ U $ is a doubly parabolic Baker domain, $ g^*|_{\partial \mathbb{D}} $ is ergodic with respect to $ \lambda$. If the Denjoy-Wolff point of $ g $ is not a singularity, $ g^*|_{\partial \mathbb{D}} $ is recurrent.
		\item If $ SV\cap U $ are compactly contained in $ U$, the Denjoy-Wolff point of $g$ is not a singularity.
		
		\item If $ g^*|_{\partial \mathbb{D}} $ is ergodic (resp. recurrent) with respect to $ \lambda $,  so is $ f|_{{\partial} U} $ with respect to $ \omega_U $. 
		
		%	\noindent If $ g^*|_{\partial \mathbb{D}} $ is recurrent with respect to $ \lambda $, then for $ \omega_U $-almost every point $ x\in\partial U $, $ \left\lbrace f^n(x)\right\rbrace _n $ is dense in $ \partial U$.
		
		\item  Let $ k $ be a positive integer. Then, the inner function associated to $ (f,U) $ has the same ergodic properties than the inner function associated to $ (f^k,U) $.
	\end{enumerate}
\end{thm}
The following result concern the existence of invariant measures for $ f|_{\partial U} $.

\begin{thm}{\bf (Invariant measures for $ f|_{\partial U} $)}\label{corol-invariant-measures}
	Let $ f\in\mathbb{K}$, and let $ U $ be an invariant simply connected Fatou component for $ f $. 
	\begin{enumerate}[label={\em (\roman*)}]
		\item If $ U $ is an attracting basin or a Siegel disk with fixed point $ p\in U $, the harmonic measure $ \omega_U(p,\cdot) $ is $ f $-invariant.
		\item If $ U $ is a parabolic basin or a doubly-parabolic Baker domain, with convergence point $ p\in\widehat{\partial}U $. Then,  the push-forward $ \mu\coloneqq (\varphi^*)_*\lambda_\mathbb{R} $ of the measure \[\lambda_\mathbb{R} (A)=\int_A \frac{1}{\left| w-1\right| ^2}d\lambda(w), \hspace{0.5cm}A\in \mathcal{B}(\partial\mathbb{D}),\] under a Riemann map  $ \varphi\colon \mathbb{D}\to U$, $ \varphi^*(1)=p $, is $ f $-invariant. The support of $ \mu $ is $ \widehat{\partial}U $. 
	\end{enumerate}
\end{thm}

\subsection{On the condition $ \int_{\partial U} \log\left| x-SV\right|^{-1} d\omega_U(x)<\infty  $}\label{subsect-sparseSV}

We need the following lemma, which gives a geometric intuition for the condition $ \int_{\partial U} \log\left| x-SV\right|^{-1} d\omega_U(x)<\infty  $, by showing it is equivalent to singular values are `not too dense' on $ \widehat{\partial} U $ with respect to harmonic measure.
\begin{lemma}\label{lemma-suma-mesura-disks}
	Let $ f\in\mathbb{K} $, and $ U $ be a Fatou component. Then, $ \int_{\partial U} \log\left| x-SV\right|^{-1} d\omega_U(x)<\infty  $ if and only if for any $ C>0 $ and $ t\in (0,1) $, 	  \[\sum_{n\geq 0} \omega_U \left( \bigcup\limits_{s\in SV} D(s, C\cdot t^n)\right) <\infty.\] 
\end{lemma}
\begin{remark}\label{lemma-finitely-manySV}
	We note that $ \int_{\partial U} \log\left| x-SV\right|^{-1} d\omega_U(x)<\infty  $ always holds if $ SV\cap\widehat{\partial}U $ is finite. Indeed, given any simply connected domain $U$, for every $ a\in \mathbb{C} $, $  \log \left| z-a\right|  \in L^1(\omega_U)$  \cite[Prop. 21.1.18]{Conway}.
\end{remark}
\begin{proof}[Proof of Lemma \ref{lemma-suma-mesura-disks}]
	First note that, since we are working with the spherical metric, $ \log\left| x-SV\right|^{-1} $ is uniformly bounded above. Hence, one has only to examine the previous integral close to singular values. Let $ 0<t<1 $, and  \[A_n\coloneqq \left\lbrace z\in\mathbb{C}\colon t^{n+1}\leq \left| z-SV\right| <t^n  \right\rbrace; \hspace{0.3cm}D_n\coloneqq \bigcup\limits_{s\in SV} D(s, C\cdot t^n)=\left\lbrace z\in\mathbb{C}\colon  \left| z-SV\right| <t^n  \right\rbrace. \] Then, \[\int_{\partial U} \log\left| x-SV\right|^{-1} d\omega_U(x) \geq - \sum_n\log(t^{n})\cdot \omega_U(A_n)=-\log t\cdot \sum_n n\cdot \omega_U(A_n).\] This already implies that $ \sum_{n\geq 0} \omega_U (D_n)=\sum_n (n+1)\cdot  \omega_U(A_n)<\infty $, for every $ t\in (0,1) $. Since for every $ C>0 $ and $ 0<t<1 $ exists $ 0<s<1 $ with $ C\cdot t^n<s^n $ for $ n  $ sufficiently large, the claim of the lemma follows.
	
	For the converse, note that 
$ \sum_{n\geq 0} \omega_U (D_n)=\sum_n (n+1)\cdot  \omega_U(A_n)<\infty $,
	implying that 
	\[\infty>-\log t \cdot \sum_n n\cdot  \omega_U(A_n)=-\sum_n\log(t^{n})\cdot \omega_U(A_n)\geq \int_{\partial U} \log\left| x-SV\right|^{-1} d\omega_U(x), \] as desired.
\end{proof}

\begin{remark} Note that Lemma \ref{lemma-suma-mesura-disks} already implies that $ SV\cap \widehat{\partial} U $ has zero harmonic measure. 
\end{remark}

\section{Pesin theory for attracting basins. Theorem \ref{teo:A}}\label{section-pesinA}

In this section, we take on the main challenge of this paper: developing Pesin theory for a simply connected attracting basin $ U $ of a function of class $ \mathbb{K} $, or, in other words, proving that generic infinite inverse branches are well-defined on $ \partial U $.

The easiest assumption one can make to  show that  generic infinite inverse branches are well-defined in $ \partial U $, is that there exists $ x\in\partial U $ and $ r>0 $ so that $ D(x,r)\cap P(f) =\emptyset$. Indeed, 
in such case, all iterated inverse branches are well-defined in $ D(x,r) $. Moreover, since $ f|_{\partial U} $ is ergodic and recurrent, and $ D(x,r) $ has positive harmonic measure, it follows that the forward orbit of $ \omega_U$-almost every $ y\in\partial U $ eventually falls in $ D(x,r) $, so all iterated inverse branches are well-defined around $ y $.

The previous method has a main limitation: it does not work when $ \partial U\subset P(f) $. Even in the case where $ f $ is a polynomial, one can find examples for which $ \partial U\subset P(f) $, or even $ \mathcal{J}(f)\subset P(f) $. Our goal is precisely to show that, even in the case where $ \partial U\subset P(f) $, if $ \int_{\partial U} \log\left| x-SV\right|^{-1} d\omega_U(x)<\infty  $, generic inverse branches are well-defined. Hence, one should work with each infinite backward orbit separately, and try to find a disk where the inverse branches corresponding to this backward orbit are well-defined, but other inverse branches may fail to be defined. Here is where Rohklin's natural extension plays a crucial role.

Therefore, let $ U $ be a simply connected attracting basin for a map $ f\in\mathbb{K} $ with fixed point $ p\in U $, and consider the measure-theoretical dynamical system given by $ (\partial U,  \omega_U, f) $, where $ \omega_U $ is the harmonic measure with basepoint $ p $. Note that, through this section, $ \omega_U $ stands for the harmonic measure with basepoint $ p $, although we do not write it explicitly. Recall that $ \omega_U $ is $ f $-invariant, ergodic and recurrent. Note also that we omit the dependence of the previous dynamical system on the $ \sigma $-algebra $ \mathcal{B}(\widehat{\mathbb{C}}) $, in order to lighten the notation. 

Now, consider the  natural extension of $ (\partial U,  \omega_U, f) $, denoted by $ (\widetilde{\partial U},  \widetilde{\omega_U}, \widetilde{f}) $, and given by the projecting morphisms $ \left\lbrace \pi_{U,n}\right\rbrace _n $. We note that $ (\partial U,  \omega_U, f) $ is indeed a Lebesgue probability space (in fact, it is isomorphic, in the measure-theoretical sense, to the unit interval), and hence Theorem \ref{thm-natural-extension} can be applied to guarantee the existence of Rokhlin's natural extension. Thus, $ \widetilde{\partial U} $ is the space of backward orbits $ \left\lbrace x_n\right\rbrace _n\subset \partial U $, with $ f(x_{n+1})=x_n $ for $ n\geq 0 $, and $ \widetilde{f}\colon  \widetilde{\partial U}\to \widetilde{\partial U}$ is the automorphism which makes  the following diagram commute.
\[\begin{tikzcd}[row sep=-0.2em,/tikz/row 3/.style={row sep=4em}]
...\arrow{r}{\widetilde{f}} &	\widetilde{\partial U} \arrow{rr}{\widetilde{f}} & &	\widetilde{\partial U} \arrow{rr}{\widetilde{f}} & &	\widetilde{\partial U}\arrow{r}{\widetilde{f}} &...\\ 
&	{\scriptstyle \left\lbrace x_{n+2}\right\rbrace _n} \arrow{ddddddd}{\pi_{U,n}}& &{\scriptstyle \left\lbrace x_{n+1}\right\rbrace _n} \arrow{ddddddd}{\pi_{U,n}}& &	{\scriptstyle \left\lbrace x_{n}\right\rbrace _n }\arrow{ddddddd}{\pi_{U,n}} \\ \\ \\	\\ \\ \\ \\...\arrow{r}{f} &\partial U \arrow{rr}{f}  &&
\partial U \arrow{rr}{f}  && \partial U \arrow{r}{f} &...\\ &{\scriptstyle x_{n+2} }&&{\scriptstyle x_{n+1} }&&{\scriptstyle x_n}
\end{tikzcd}
\]

Since the natural extension inherits the ergodic properties of the original dynamical system, we have that $ \widetilde{\omega_U} $ is an $ \widetilde{f }$-invariant, ergodic and recurrent probability (Prop. \ref{prop-ergodic-properties-natural-ext}). Moreover, for every measurable set $ A\subset \partial U $ with $ \mu(A)>0 $ and $ \widetilde{\omega_U} $-almost every $ \left\lbrace x_n\right\rbrace _n\in \widetilde{\partial U}$, there exists a sequence $ n_k\to \infty $ such that $ x_{n_k}\in A $ for all $k\in\mathbb{N}$ (Corol. \ref{prop-ergodic-properties-natural-ext2}). 

We shall rephrase Theorem \ref{teo:A} in terms of Rokhlin's natural extension as follows.
\begin{thm}{\bf (Inverse branches are well-defined almost everywhere)}\label{thm-pesin}
	Let $ f\in \mathbb{K}$, and let $U$ be a simply connected attracting basin for $ f $, with  fixed point $ p\in U $. Let $ \omega_U $ be the harmonic measure in $ \partial U $ with base point $ p $. Assume:
	\begin{enumerate}[label={\em (\alph*)}]
		\item\label{thm-pesin-hyp-lyap} $ \log \left| f'\right| \in L^1(\omega_U) $, and $ \chi_{\omega_U}(f)>0$;
		\item \label{thm-pesin-hyp-SV}  $\sum_n \omega_U (D(x,M^n ))<\infty $, for every $ M\in (0,1) $.
	\end{enumerate}
	Then, for $ \widetilde{\omega_U} $-almost every $ \left\lbrace x_n\right\rbrace _n \in \widetilde{\partial U}$, there exists $ r\coloneqq r(\left\lbrace x_n\right\rbrace _n)>0 $ such that
		\begin{enumerate}[label={\em (\roman*)}]
		\item for all $n\geq 0$ the inverse branch $ F_n $ sending $x_0 $ to $ x_n $ is well-defined in $ D(x_0,r) $;
		\item  for every $ \chi\in (-\chi_{\omega_U}, 0) $, there exists $ C>0 $ such that, for all $ n\in\mathbb{N}$,  $ \left| F'_n(x_0)\right|<C \cdot e^{\chi \cdot n} $;
		\item  for every $ r_0\in (0, r)$, there exists $ m\in\mathbb{N} $ such that $ F_m(D(x_0, r))\subset D(x_0, r_0) $. 
	\end{enumerate}
\end{thm}

We show now how to deduce Theorem \ref{teo:A} from Theorem \ref{thm-pesin}, and then we give the proof of the latter theorem. 

\begin{proof}[Proof of Theorem \ref{teo:A}] The assumptions of Theorems \ref{teo:A} and \ref{thm-pesin} are equivalent, by Lemma \ref{lemma-suma-mesura-disks}. We have to show that the conclusions of \ref{teo:A} can be derived from the ones of Theorem \ref{thm-pesin}. But this follows immediately from Corollary \ref{prop-ergodic-properties-natural-ext2}. Indeed, since $ \widetilde{f} $  is ergodic and recurrent with respect to $ \widetilde{\omega_U} $,  for any  $ A\subset X $ measurable set  with $ \omega_U(A)>0 $, for $ \widetilde{\omega_U} $-almost every $ \left\lbrace x_n\right\rbrace _n\in \widetilde{X }$, there exists a sequence $ n_k\to \infty $ such that $ x_{n_k}\in A $.  It follows that, for every countable collection of measurable sets $ \left\lbrace A_k\right\rbrace _k\subset\partial U $ with $ \omega_U(A_k) >0$, then for $ \widetilde{\omega_U} $-almost every $ \left\lbrace x_n\right\rbrace _n\in \widetilde{X }$, there exists a sequence $ n_k\to \infty $ such that $ x_{n_k}\in A_k $. 
\end{proof}

\begin{remark}\label{rmk-pesin-also-for-inner}
	Before starting the proof let us note that we are assuming $ f\in\mathbb{K} $ just because it is a larger class of functions in which Fatou components are defined. We do not use the fact that functions in class $ \mathbb{K}  $ have only countably many singularities, we only use that singular values are `not too dense'  on $ \partial U $ (hypothesis {\em \ref{thm-pesin-hyp-SV}}). 
\end{remark}

The remainder of the section is devoted to prove Theorem \ref{thm-pesin}.

\subsection{Proof of Theorem \ref{thm-pesin}}

Recall that $ \omega_U $ is an $ f $-invariant ergodic probability in $ \partial U $. 		
We fix $ M\in (e^{\frac{1}{4}\cdot\chi}, 1) $.

				\begin{sublemma}{\bf (Almost every backward orbit does not come close to singular values)}
			For $ \widetilde{\omega_U} $-almost every $ \left\lbrace x_n\right\rbrace _n\in\widetilde{\partial U} $, we have \begin{enumerate}[label={\em (1.\arabic*)}]
			\item\label{lemma2a} $ x_0 \notin \bigcup\limits_{s\in SV} \bigcup\limits_{n\geq0}f^n(s)$,
			\item\label{lemma2b} $ \lim\limits_{n}\dfrac{1}{n}\log \left| (f^n)'(x_n)\right| =\chi_{\omega_U}(f)$,
						\item\label{lemma2c} if  $ D_n\coloneqq \bigcup\limits_{s\in SV} D(s,  M^n) $, then $ x_n\in D_n$ only for a finite number of $ n $'s.
			\end{enumerate}
		\end{sublemma}
\begin{proof} Since the finite intersection of sets of full measure has full measure, it is enough to show that each of the conditions is satisfied in a set of full measure.
	
Condition {\em\ref{lemma2a}} follows from $ \omega_U(SV(f))=0$. Indeed, $ f $ is holomorphic, and hence absolutely continuous, we have $\omega_U( \bigcup_{s\in SV(f)} \bigcup_{n\geq0}f^n(s)) =0 $.

Requirement {\em\ref{lemma2b}} follows from the Birkhoff Ergodic Theorem \ref{thm-birkhoff} applied to the map $ \log \left| f'\right| $, which is integrable by assumption {\em \ref{thm-pesin-hyp-lyap}}. Indeed, for $ \widetilde{\omega_U }$-almost every $ \left\lbrace x_n\right\rbrace  _n\in\widetilde{\partial U} $, we have
\begin{align*}
	\chi_{\omega_U} (f)&=\int_{\partial U} \log \left| f'\right| d\omega_U=\lim\limits_{m}\frac{1}{m}\sum_{k=0}^{m-1}\log \left| f'(f^k(x_0))\right|=
	\\
&=	\lim\limits_{m}\frac{1}{m}\sum_{k=0}^{n-1}\log \left| f'(f^k(\pi_{U, 0}(\left\lbrace x_n\right\rbrace _n)))\right| =\lim\limits_{m}\frac{1}{m}\sum_{k=0}^{m-1}\log \left| f'(\pi_{U, 0}(\widetilde{f}^k(\left\lbrace x_n\right\rbrace _n)))\right|,
\end{align*}  where in the last two equalities we used the properties of Rokhlin's natural extension.

Now, $ \widetilde{f} $ is a measure-preserving automorphism, and, since  $ \log\left| f'\right| \in L^1(\omega_U)  $, $ \log\left| f'\circ \pi_{U, 0}\right| \in L^1(\widetilde{\omega_U})  $. Then, Birkhoff Ergodic Theorem yields
\begin{align*}
&	\lim\limits_{m}\frac{1}{m}\sum_{k=0}^{m-1}\log \left| f'(\pi_{U, 0}(\widetilde{f}^k(\left\lbrace x_n\right\rbrace _n)))\right|= \lim\limits_{m}\frac{1}{m}\sum_{k=0}^{m-1}\log \left| f'(\pi_{U, 0}(\widetilde{f}^{-k}(\left\lbrace x_n\right\rbrace _n)))\right|= 
	\\
=	&\lim\limits_{m}\frac{1}{m}\sum_{k=0}^{m-1}\log \left| f'(x_k)\right|=\lim\limits_{m}\frac{1}{m}\log ( \left| f'(x_0)\right|\dots  \left| f'(x_m)\right|)=\lim\limits_{m}\frac{1}{m}\log \left| (f^m)'(x_m)\right|.
\end{align*}
 Putting everyting together, we get that for $ \widetilde{\omega_U} $-almost every $ \left\lbrace x_n\right\rbrace _n $, we have \[\lim\limits_{n}\dfrac{1}{n}\log \left| (f^n)'(x_n)\right| =\chi_{\omega_U}(f).\]
	For condition {\em\ref{lemma2c}}, note that
by hypothesis {\em\ref{thm-pesin-hyp-SV}},  \[\sum_{n\geq 1}\widetilde{\omega_U}(\pi_{U,n}^{-1}(D_n))=\sum_{n\geq 1}\omega_U(D_n)<\infty,\] 
	Thus, by the Borel-Cantelli Lemma \ref{lemma-borel-cantelli}, for $ \widetilde{\omega_U} $-almost every $ \left\lbrace x_n\right\rbrace _n\in\widetilde{\partial U} $, $ x_n\in D_n $ for only finitely many $ n $'s, as desired. 
	This ends the proof of the Lemma.
\end{proof}

	Let us fix a backward orbit $ \left\lbrace x_n\right\rbrace _n $ satisfying the conditions of the previous lemma. By {\em\ref{lemma2c}}, there exists $ n_1\in\mathbb{N} $ such that, for $ n\geq n_1$, $x_n\notin \bigcup_{s\in SV} D(s, M^n)$.
			Moreover, by {\em\ref{lemma2b}}, there exists $ n_2\in\mathbb{N} $, $ n_2\geq n_1 $ such that, for $ n\geq n_2$, \[\left| (f^n)'(x_n)\right| ^{-\frac{1}{4}}<M^n<1.\] 
	Note that both $ n_1 $ and $ n_2 $ do not only depend on the starting point $ x_0 $, but on all the backward orbit $ \left\lbrace x_n\right\rbrace _n $. Two different backward orbits starting at $ x_0 $ may require different $ n_1 $ or  $ n_2 $.

Let \[b_n\coloneqq \left| (f^{n+1})'(x_n)\right| ^{-\frac14}, \hspace{0.5cm}P\coloneqq \prod_{n\geq 1}(1-b_n).\]
Observe that, since $ \sum_n b_n\leq\sum_n M^n<\infty $, the infinite product in $ P $ is convergent; in particular $P$ is positive.  Choose $ r\coloneqq r(\left\lbrace x_n\right\rbrace _n) >0$ such that \begin{enumerate}[label={(2.\arabic*)}]
	\item\label{lemma3a} $ 2rP<1$,
	\item\label{lemma3b} for $1\leq  n\leq n_2 $, the branch $ F_{n} $ of $ f^{-n} $ sending $ x_0 $ to $ x_n $ is well-defined in $ D(x_0, r) $,
	\item\label{lemma3c.euclidi}  $ F_{n_2}(D(x_0, r\prod_{m=1}^{n_2}(1-b_m)))\subset D(x_{n_2}, M^{n_2}) $.  
\end{enumerate}
The remaining inverse branches will be constructed by induction (Claim \ref{claim-diametres}), but first let us note that such a $ r>0 $ exists. Indeed,  it follows from the fact that the inverse branch $ F_{n_2} $ sending $ x_0$ to $ x_{n_2} $ is well-defined in an open neighbourhood of $ x_0$, since the set of singular values of $ f^{n_2} $ is closed, and $ x_0\notin SV(f^{n_2} ) $.

\begin{claim}\label{claim-diametres}
	For every $ n\geq n_2 $, there exists an inverse branch $ F_n $ sending $ x_0 $ to $ x_n$, defined in $ D(x_0, r\prod_{m=1}^n(1-b_m)) $, and such that \[ F_n(D(x_0, r\prod_{m=1}^n(1-b_m)) )\subset D(x_n, M^n).\]
\end{claim}

Note that proving the claim ends the proof of the theorem. Indeed,  letting $ n\to \infty $ we get that all inverse branches are well-defined in $ D(x_0, rP) $, i.e. in a disk centered at $ x_0 $ of positive radius. The estimate on the derivative follows from  $\left| (f^n)'(x_n)\right| ^{-\frac{1}{4}}<M^n<1$, with $ M\in (e^{\frac{1}{4}\chi}, 1) $, for $ n\geq n_2 $. 
\begin{proof}[Proof of the claim]
	Suppose the claim is true for some $ n\geq n_2 $, and let us see that it also holds for $ n+1 $.
	First, note that $ D(x_n, M^n)\cap SV=\emptyset  $ for all $ n\geq n_2 $ (by the choice of $ n_2 $).
	  Hence, there exists a branch $ F $ of $ f^{-1} $ satisfying $ F(x_n)=x_{n+1} $, well-defined in $ D(x_n, M^n)$. By the inductive hypothesis, there exists an inverse branch $ F_n $ sending $ x_0 $ to $ x_n$, defined in $D_n\coloneqq  D(x_0, r\prod_{m=1}^n(1-b_m)) $, and such that $ F_n(D_n )\subset D(x_n, M^n) $.
	Set $ F_{n+1}=F\circ F_n $. Then, $ F_{n+1} $   is well-defined in $ D_n$, and sends $ x_0 $ to $ x_{n+1} $.  
	
		Now we use Koebe's distortion estimates (Thm. \ref{thm-Koebe}) to prove the bound on the size of $  F_{n+1}( D_{n+1})$, where $ D_{n+1}\coloneqq  D(x_0, r\prod_{m=1}^{n+1}(1-b_m))\subset D_n$.  Note that $ F_{n+1} $ is well-defined in $ D_n$, which is strictly larger than $ D_{n+1}$, and the ratio between the two radii of both disks is $ (1-b_{n}) $. Since $ F_{n+1}|_{D_{n}}$ is univalent, we have $ F_{n+1}(D_{n+1})\subset D(x_{n+1}, R) $, where \begin{align*}
		R=r\cdot \prod_{m= 1}^{n+1}(1-b_m)\cdot \left| (F_{n+1})'(x_{0})\right| \cdot \dfrac{2}{b^3_n}\leq 2r\cdot\dfrac{\left| (f^{n+1})'(x_{n+1})\right|^{-1} }{\left| (f^{n+1})'(x_{n+1})\right| ^{-\frac34}}\leq\left| (f^{n+1})'(x_{n+1})\right| ^{-\frac14} \leq M^{n+1},
	\end{align*} as desired.
\end{proof}
As noted before, this last claim ends the proof of Theorem \ref{thm-pesin}.\hfill $ \square $
\section{Entire functions and the first return map. Theorem \ref{teo:B}}\label{sect-pesin-entire}

In this section, we extend Theorem \ref{teo:A} to parabolic and Baker domains of entire maps. The main challenge is that there does not exist an invariant probability  which is absolutely continuous  with respect to harmonic measure. However, the existence of an invariant $ \sigma $-finite measure in $ \partial U $ absolutely continuous  with respect to $ \omega_U $ will allow us to perform Pesin theory, by means of the first return map.

We shall start by constructing Rokhlin's natural extension (note that this is indeed possible due to the existence of the $ \sigma $-invariant measure). We do this by showing that Rokhlin's natural extension is compatible with the use of first return maps if the transformation we consider is recurrent. This allows us to move from our problem of finding inverse branches in a space endowed with a $ \sigma $-invariant measure to a probability space, where we can perform Pesin theory in a standard way. We do this construction of the first return map and Rokhlin's natural extension in Section \ref{subsect-FRMandRokhlin}, and finally we develop Pesin theory in Section \ref{subsection-Pesin's-theoryFRM}.

\subsection{The first return map and Rokhlin's natural extension}\label{subsect-FRMandRokhlin}

Assume  $ U $ is a parabolic basin or a  Baker domain, such that $ f|_{\partial U} $ is recurrent. The measure 
\[\lambda_\mathbb{R} (A)=\int_A \frac{1}{\left| w-1\right| ^2}d\lambda(w), \hspace{0.5cm}A\in \mathcal{B}(\partial\mathbb{D}),\] is invariant under the radial extension of the associated inner function $ g $ (taken such that 1 is the Denjoy-Wolff point)
and its push-forward $\mu=(\varphi^*)_*\lambda_\mathbb{R}$ is an infinite $ \sigma $-finite invariant measure supported in $ \widehat{\partial}U $ (see Thm. \ref{corol-invariant-measures}). 

One can consider the Rokhlin's natural extension.
Indeed, let $ \widetilde{\partial U} $ be the space of backward orbits $ \left\lbrace x_n\right\rbrace _n\subset \partial U $, with $ f(x_{n+1})=x_n $ for $ n\geq 0 $, and let $ \widetilde{f}\colon  \widetilde{\partial U}\to \widetilde{\partial U}$ be the automorphism which makes  the following diagram commute.
\[\begin{tikzcd}[row sep=-0.2em,/tikz/row 3/.style={row sep=4em}]
	...\arrow{r}{\widetilde{f}} &	\widetilde{\partial U} \arrow{rr}{\widetilde{f}} & &	\widetilde{\partial U} \arrow{rr}{\widetilde{f}} & &	\widetilde{\partial U}\arrow{r}{\widetilde{f}} &...\\ 
	&	{\scriptstyle \left\lbrace x_{n+2}\right\rbrace _n} \arrow{ddddddd}{\pi_{U,n}}& &{\scriptstyle \left\lbrace x_{n+1}\right\rbrace _n} \arrow{ddddddd}{\pi_{U,n}}& &	{\scriptstyle \left\lbrace x_{n}\right\rbrace _n }\arrow{ddddddd}{\pi_{U,n}} \\ \\ \\	\\ \\ \\ \\...\arrow{r}{f} &\partial U \arrow{rr}{f}  &&
	\partial U \arrow{rr}{f}  && \partial U \arrow{r}{f} &...\\ &{\scriptstyle x_{n+2} }&&{\scriptstyle x_{n+1} }&&{\scriptstyle x_n}
\end{tikzcd}
\]

One can get an equivalent construction of backward orbits by means of the first return map. Indeed,  let $ E \subset\partial U$ be a measurable set with $ \mu(E)\in (0, \infty) $ (we will fix $ E $ later). Consider the \textit{first return map}  to $ E$, i.e. \begin{align*}
	f_{E}\colon& E\longrightarrow E \\&	x\longmapsto f^{T(x)}(x),
\end{align*}
where  $ T(x) $ denotes the first return time of $ x $ to $ E$. 
We consider the measure-theoretical dynamical system $ (E, \mu_k, f_{X_k}) $, where \[\mu_E(A)\coloneqq \frac{\mu(A\cap E)}{\mu (E)},\] for every measurable set $A\subset  \partial U $. Note that $ (E, \mu_E) $ is a probability space.
The following properties of the first return map $ f_{E} $ will be needed. 
\begin{lemma}{\bf (First return map)}\label{lemma-return-map entire}
	Let $ f_{E}\colon E\to E$ be defined as above. Then, the following holds.\begin{enumerate}[label={\em(1.\arabic*)}]
		\item \label{1.1entire} $ \mu_E $ is invariant under $ f_{E} $. In particular, $ f_{E}$ is recurrent with respect to   $ \mu_E$.
		\item \label{1.2entire}$ f_{E} $ is ergodic with respect to   $ \mu_E $.
		\item\label{1.3entire} If  $ \log \left| f'\right| \in L^1(\mu) $ , then  $ \log \left| f'_{E}(x)\right|\coloneqq  \log \left| (f^{T(x)})'(x)\right| \in L^1(\mu_E) $ and \[ \int_{X_k} \log \left| f'_{E}\right| d\mu_k= \dfrac{1}{\mu (E)} \int_{\partial U} \log \left| f'\right| d\mu.\]
	\end{enumerate}
\end{lemma}
\begin{proof} The three claims are standard facts of measure-theoretical first return maps. More precisely, {\em \ref{1.1entire}} and {\em \ref{1.2entire}} follow from \cite[Prop. 10.2.1]{UrbanskiRoyMunday1} and \cite[Prop. 10.2.7]{UrbanskiRoyMunday1}, respectively. Statement {\em \ref{1.3entire}} comes from \cite[Prop. 10.2.5]{UrbanskiRoyMunday1}, applied to $ \varphi=\log \left| f'\right|  $ and $ \varphi_{E}=\log \left| f'_{E}\right|  $ .
\end{proof}

Since $ (X_E, \mu_E) $ is a Lebesgue probability space, and $ \mu_E $ is $ f_E$-invariant, we shall consider its Rohklin's natural extension, denoted by $ (\widetilde{E},  \widetilde{f_E}) $, and given by the projecting morphisms $ \left\lbrace \pi_{E_{n}}\right\rbrace _n $. Thus, $ \widetilde{E} $ is the space of backward orbits $ \left\lbrace x^E_n\right\rbrace _n\subset \partial U $, with $ f_{E}(x^E_{n+1})=x^E_n $ for $ n\geq 0 $, and $ \widetilde{f_{E}}\colon  \widetilde{E}\to \widetilde{E}$ is the automorphism which makes  the following diagram commute.
\[\begin{tikzcd}[row sep=-0.2em,/tikz/row 3/.style={row sep=4em}]
	...\arrow{r}{\widetilde{f_{E}}} &	\widetilde{E} \arrow{rr}{\widetilde{f_{E}}} & &	\widetilde{E} \arrow{rr}{\widetilde{f_{E}}} & &	\widetilde{E}\arrow{r}{\widetilde{f_{E}}} &...\\ 
	&	{\scriptstyle \left\lbrace x^E_{n+2}\right\rbrace _n} \arrow{ddddddd}{\pi_{E_{n}}}& &{\scriptstyle \left\lbrace x^E_{n+1}\right\rbrace _n} \arrow{ddddddd}{\pi_{E_{n}}}& &	{\scriptstyle \left\lbrace x^E_{n}\right\rbrace _n }\arrow{ddddddd}{\pi_{E_{n}}} \\ \\ \\	\\ \\ \\ \\...\arrow{r}{f_{E}} & E \arrow{rr}{f_{E}}  &&
 E \arrow{rr}{f_{E}}  && E \arrow{r}{f_{E}} &...\\ &{\scriptstyle x^E_{n+2} }&&{\scriptstyle x^E_{n+1} }&&{\scriptstyle x^E_n}
\end{tikzcd}
\]

Since the natural extension of a probability space inherits the ergodic properties of the original  system, we have that $ \widetilde{\mu_E} $ is  $ \widetilde{f_{E} }$-invariant, ergodic and recurrent (Prop. \ref{prop-ergodic-properties-natural-ext}). 

We claim that both constructions of spaces of backward orbits are essentially the same, with the only difference that, when considering the first return map, orbits starting at the set $ E $ are written `packed' according to their visits to $ E $.

Indeed, given a backward orbit $ \left\lbrace x^E_n\right\rbrace _n \subset E$ for $ f_{E} $, we can associate to it unambiguously a backward orbit $ \left\lbrace x_n\right\rbrace _n \subset \partial U$ for $ f$ as follows. Let $ x_0\coloneqq x^E_0 $, and let $ x_{T(x^E_1)}\coloneqq x^E_1 $. Since $ f_{E}(x^E_1)= f^{T(x^E_1)}(x^E_1)= x^E_0 $, for $ n=1,\dots, T(x^E_1)-1 $, let $ x_n \coloneqq f^{T(x^E_1)-n}(x^E_1)$. 
The rest of the backward orbit is defined recursively. We say that the $ f $-backward orbit $ \left\lbrace x_n\right\rbrace _n \subset \partial U$ is  \textit{associated} to the $ f_{E} $-backward orbit $ \left\lbrace x^E_n\right\rbrace _n \subset E$.   In the same way, if a $ f $-backward orbit visits $ E $ infinitely often, we can associate a $ f_E $-backward orbit to it.

As noted above, for every $ f_{E} $-backward orbit we can associate a $ f $-backward orbit. Moreover, due to recurrence, the converse is true $ \widetilde{ \mu} $-almost everywhere. Hence, it is enough to consider $ f_{E} $-backward orbits.

\begin{lemma}{\bf (Distribution of $ f_{E} $-backward orbits in $ \partial U $)}\label{lemma-FRM}
	Let  $ (\partial U, \mu, f) $ and  $ (E, \mu_E, f_{E}) $, and consider their natural extensions as before. Then, for $ \widetilde{\mu} $-almost every $ \left\lbrace x_n\right\rbrace _n\subset \partial U $ with $ x_0\in E $,  $ x_n\in E $ infinitely often, so we can associate a  $ f_{E} $-backward orbit to it.
\end{lemma}

Let us now fix the set $ E $, whose first return map will enjoy specific properties. 

\begin{prop}{\bf (The set $ E $)}\label{lemmaE}
		Let $ f\colon\mathbb{C}\to\mathbb{C}$ be a meromorphic function, and let $U$ be a simply connected parabolic basin or   Baker domain. Consider $ \varphi\colon \mathbb{H} \to U$, and let $ h\colon\mathbb{H}\to\mathbb{H} $ be the inner function associated with $ (f, U) $, with the Denjoy-Wolff point placed at $ \infty $, which we assume not to be a singularity. Let $ p^\pm_1 , p^\pm_2\dots$ be the (radial) preimages of $ \infty $, ordered such that $ p^-_1<p^-_2<\dots<p^+_2<p^+_1 $.

		Let $I\coloneqq \left[ p^-_1, p^+_1\right] $, and  $E\coloneqq \varphi^*(I) $. Then, as $ n\to\infty $,
		\[\mu(\left\lbrace x\in E\colon T(x)\geq n\right\rbrace )\sim \frac2{\sqrt n}.\]
\end{prop}
\begin{proof}
	The existence of the set $ I $ 
	is proven in
	\cite[Sect. 9.2]{ivrii2023inner}, together with the fact that $ \lambda(I^C\cap \left\lbrace T(x)=n\right\rbrace )\sim \frac2{\sqrt n} $. The standard fact for $ \sigma $-finite measures $ \lambda $ and sweep-out sets $ I $
	\[\lambda(I\cap \left\lbrace T(x)=n\right\rbrace )=\lambda(I^C\cap \left\lbrace T(x)>n\right\rbrace )\] (see e.g. \cite[p.19, Corol. 1]{notesInfErgodicTheory}) gives that $ \lambda(I\cap \left\lbrace T(x)>n\right\rbrace )\sim \frac2{\sqrt n} $.
	The estimates for the system $ (f|_{\partial U}, \mu) $ follow from the definition of the measure $ \mu $ as the push-forward of $ \lambda_\mathbb{R} $ under $ \varphi^* $.
\end{proof}
\subsection{Pesin theory for the first return map. Proof of Theorem \ref{teo:B}}\label{subsection-Pesin's-theoryFRM}

We shall start by rewriting  Theorem \ref{teo:B} in terms of the space of backward orbits given by Rokhlin's natural extension.

\begin{thm}{\bf (Inverse branches are well-defined almost everywhere)}\label{thm-pesin-entire}
		Let $ f\colon\mathbb{C}\to\mathbb{C}$ be a meromorphic function, and let $U$ be a simply connected parabolic basin or   Baker domain. 
	Assume 
	\begin{enumerate}[label={\em (\alph*)}]
		\item\label{thm-pesin-entire-recurrence} the Denjoy-Wolff point of the associated inner function is not a singularity;
		\item\label{thm-pesin-entire-lyap} $ \log \left| f'\right| \in L^1(\mu) $, and  $ \int_{\partial U} \log \left| f'\right| d\mu>0$;
		\item \label{thm-pesin-entire-hyp-SV}  there is $ \varepsilon >0 $ such that, if $\partial U_{+\varepsilon} \coloneqq \left\lbrace z\in\mathbb{C}\colon \textrm{\em dist}(z, \partial U)<\varepsilon \right\rbrace $,  $ SV\cap \partial U_{+\varepsilon} $ is finite.
	\end{enumerate}
		Then, for $ \widetilde{\mu} $-almost every backward orbit $ \left\lbrace x_n\right\rbrace _n \in \widetilde{\partial U}$, there is $ r_0\coloneqq r_0(\left\lbrace x_n\right\rbrace _n)>0 $ such that
	\begin{enumerate}[label={\em (\roman*)}]
		\item the inverse branch $ F_n $ sending $x_0 $ to $ x_n $ is well-defined in $ D(x_0,r_0) $;
		\item  for every $ r\in (0, r_0)$, there exists $ m\in\mathbb{N} $ such that $F_m(D(x_0, r_0))\subset D(x_0, r)$, and $ \textrm{\em diam } F^j_m(D(x_0, r))\to 0 $, as $ j\to\infty $. 
	\end{enumerate}
\end{thm}

It is clear that Theorem \ref{thm-pesin-entire} implies Theorem \ref{teo:B} (for hypothesis (a), see Thm. \ref{thm-ergodic-boundary-map}; for (b), see Prop. \ref{prop-log-int-parabolic-basins}). 
Going one step further, using the set $ E $ defined above,  we shall write Theorem \ref{thm-pesin-entire} in terms of the first return maps $ f_{E}\colon E \to E $ as follows.

\begin{prop}{\bf (Generic inverse branches are well-defined for the first return map)}\label{prop65}
	Under the assumptions of Theorem \ref{thm-pesin-entire}, consider the set $E$ as in Section \ref{subsect-FRMandRokhlin}. Then,  for $ \widetilde{\mu_E} $-almost every backward orbit $ \left\lbrace x^E_n\right\rbrace _n \in \widetilde{\partial U}$, there exists $ r_0\coloneqq r_0(\left\lbrace x^k_n\right\rbrace _n)>0 $ such that
	\begin{enumerate}[label={\em (\roman*)}]
		\item the inverse branch $ F^E_n $ sending $x^E_0 $ to $ x^E_n $ is well-defined in $ D(x^E_0,r_0) $;
		\item  for every $ r\in (0, r_0)$, there exists $ m\in\mathbb{N} $ such that \[F^E_m(D(x^E_0, r_0))\subset D(x^E_0, r).\]
	\end{enumerate}
\end{prop}

\subsection{Proof of Proposition \ref{prop65}}

We shall establish which are the `singular values' for the first return map $ f_{E}$. The only obstructions when defining inverse branches come from the singular values of $ f $. Note that, if there are no singular values of $ f $ in $ D(f_{E}(x), \varepsilon) $ and the first return time of $ x $ is 1, then the corresponding branch of $ f_{E} $ is well-defined in $ D(f_{E}(x), \varepsilon) $. Inductively,  if there are no critical values of $ f^n $ in $ D(f_{E}(x), \varepsilon) $ and the first return time of $ x $ is $ n $, then the corresponding branch of $ f_{E} $ is well-defined in $ D(f_{E}(x), \varepsilon) $. Hence, we observe an interplay between the points in the orbit of singular values of $ f $ and the first return times, as the limitation to define the inverse branches of $ f_{E} $.

Next we aim to give estimates on the first return times and the size of disks centered at `singular values of $ f_{E}$'. This is the content of Lemma \ref{lemma-SV-f_A}.

We use the following notation: let $ \left\lbrace v_1, \dots, v_N\right\rbrace  $ be the singular values of $ f $ in $\partial U_{+\varepsilon}$ (we assumed there are finitely many-- other singular values do not play a role in the considered inverse branches), and denote them by $ SV(f) $. $ T(x) $ stands for the first return time to $ E $ of $ x\in E $. 
\[A_n\coloneqq \left\lbrace x\in E\colon T(x)=n \right\rbrace \]
\[B_n \coloneqq \left\lbrace x\in E\colon T(x)\geq n \right\rbrace\]
\begin{sublemma}{\bf (Estimates on critical values and first returns)}\label{lemma-SV-f_A} In the previous setting, the following holds.
\begin{enumerate}[label={\em (2.\arabic*)}]
		\item\label{lemma1b-entire} $ \sum\limits_{n} \mu_E(B_{n^4})<\infty $.
	\item\label{lemma1c-entire} $ \sum_n \mu_E(D(CV(f^{n^4}), \varepsilon\cdot \lambda^n)) <\infty$, for any $ \lambda\in (0,1) $.
\end{enumerate}
\end{sublemma}
\begin{proof}
{\em \ref{lemma1b-entire}} follows directly from the estimate in Lemma \ref{lemmaE}.
	For {\em \ref{lemma1c-entire}}  note that, since $ E $ has finite measure, the measures $ \mu_E $ and $ \omega_U $ are comparable. Note also that $ f^{n^4} $ has $ n^4\cdot N $ singular values (where $ N $ stands for the number of singular values of $ f $). Then, applying a standard estimate of the harmonic measure of disks (see Lemma \ref{lemma-estimates-harmonic-measure}), we have
	\begin{align*}
		&\mu(E)\sum_n \mu_k(D(CV(f^{n^4}),\varepsilon \cdot \lambda^n))\lesssim\sum_n \mu (D(CV(f^{n^4}), \varepsilon\cdot \lambda^n))\\
		\lesssim&\sum_n \omega_U(D(CV(f^{n^4}), \varepsilon\cdot \lambda^n)) \leq  \sum_n \varepsilon^{1/2}\cdot  N\cdot n^4\cdot \lambda ^{n/2} <\infty.
	\end{align*}
\end{proof}

From here, the proof ends as the one of Theorem \ref{thm-pesin}:  proving that orbits under $ f_{E} $ do not come close to the `singular values of $ f_{E}$', and finally constructing inductively the required inverse branches of $ f_{E}$, which turn out to be a composition of inverse branches for the original map $ f $, as explained in Section \ref{subsect-FRMandRokhlin}. For convenience, we outline the steps of the proof, although not giving all the details as in Theorem \ref{thm-pesin}.

		Set \[\chi \coloneqq \int_{E} \log \left| f'_{E}\right| d\mu_E \in (0, +\infty), \] and let $ M\in (e^{\frac{1}{4}\cdot\chi}, 1) $.
\begin{sublemma}{\bf (Almost every orbit does not come close to singular values) }
	For $ \widetilde{\mu_E} $-almost every $ \left\lbrace x^E_n\right\rbrace _n\in\widetilde{E} $, we have \begin{enumerate}[label={\em (3.\arabic*)}]
		\item\label{lemma2a-entire} $ x^E_0 \notin \bigcup\limits_{s\in SV(f)} \bigcup\limits_{n\geq0}f^n(s)$,
		\item\label{lemma2b-entire} $ \lim\limits_{n}\dfrac{1}{n}\log \left| (f_{E}^n)'(x^E_n)\right| =\chi$,
		\item\label{lemma2c-entire} inverse branches of $ f_{E} $ are well-defined in $ D(x^E_n, \varepsilon\cdot M^n) $, except for finitely many $ n $'s.
	\end{enumerate}
\end{sublemma}
\begin{proof} Since the finite intersection of sets of full measure has full measure, it is enough to show that each of the conditions is satisfied in a set of full measure.
	
	For condition {\em\ref{lemma2a-entire}}, note that $ \bigcup_{s\in SV(f)} \bigcup_{n\geq0}f^n(s) $ is countable, and hence has zero $ \mu_E$-measure.
	Requirement {\em\ref{lemma2b-entire}} follows from the Birkhoff Ergodic Theorem \ref{thm-birkhoff} applied to the map $ \log \left| f_{E}'\right| $ (note that $ \mu_E$ is an ergodic probability).

	Condition {\em\ref{lemma2c-entire}} follows from Lemma \ref{lemma-SV-f_A} together with the first Borel-Cantelli Lemma \ref{lemma-borel-cantelli}. Indeed,    if 
$ D_n=(D(CV(f^{n^4}),\varepsilon\cdot  M^n)) $, then
\[\sum_{n\geq 1}\widetilde{\mu_E}(\pi_{U,n}^{-1}(D_n))=\sum_{n\geq 1}\mu_E(D_n)<\infty,\] implying that $ x^E_{n}\notin D_n $, for all $ n\geq n_0 $, for some $ n_0 $ and $ \widetilde{ \mu_E} $-almost every backward orbit. But according to {\em \ref{lemma1b-entire}} in Lemma \ref{lemma-SV-f_A}, \[\sum_{n\geq 1}\widetilde{\mu_E}(\pi_{U,n}^{-1}(B_{n^4}))=\sum\limits_{n} \mu_E(B_{n^4})<\infty,\] so $ x^E_{n+1}\notin B_{n^4}$, for all $ n\geq n_0 $ (maybe taking $ n_0 $ larger),  and $ \widetilde{ \mu_E} $-almost every backward orbit. Thus, for all $ n\geq n_0 $, the return time of $ x^E_{n+1} $ is less than $ n^4 $, so we only have to take into account the singular values of $ f^{n^4} $ in order to define the inverse branch from $ x^E_{n} $ to $ x^E_{n+1} $. Since  $ x^E_{n}\notin D_n $, the claim follows.
\end{proof}

Fix $ \left\lbrace x^E_n\right\rbrace _n $ satifying the conditions of Lemma \ref{proof-pesin-entire-blow-up}. By {\em\ref{lemma2b-entire}} and  {\em\ref{lemma2c-entire}}, there exists $ n_0\in\mathbb{N} $ such that, for $ n\geq n_0$, inverse branches of $ f_{E} $ are well-defined in $ D(x^E_n, \varepsilon\cdot M^n) $, and  
\[\left| (f^n_{E})'(x^E_n)\right|
 ^{-\frac{1}{4}}<M^n<1.\]

We set the following notation:
\[b_n\coloneqq \left| (f^{n}_{E})'(x^E_{n})\right|^{-\frac{1}{4}}, \hspace{0.8cm}P=\prod_{n\geq 1}(1-b_n)>0.\]

Choose $ r\coloneqq r(\left\lbrace x^E_n\right\rbrace _n) >0$ such that \begin{enumerate}[label={(4.\arabic*)}]
	\item\label{lemma3a-entire} $ 2rP<\varepsilon$,
	\item\label{lemma3b-entire}  the branch $ F^k_{n_0} $ of $ f_{E}^{-n_0} $ sending $ x^E_0 $ to $ x^E_{n_0} $ is well-defined in $ D(x_0, r) $,
	\item\label{lemma3c-entire} 
	$ F^E_{n_0}(D(x^E_0, r\prod_{m=1}^{n_0}(1-b_m)))\subset D(x^E_{n_0}, M^{n_0}) $.  
\end{enumerate}

Using the same procedure as in Theorem \ref{thm-pesin} (Claim \ref{claim-diametres}), one can prove inductively the following claim.

\begin{claim}{\bf ({Inductive construction of the inverse branches})}
	For every $ n\geq n_0 $, there exists a branch $ F^E_{n} $ of $ f^{-n}_{E} $ sending $ x^E_0 $ to $ x^E_n$, defined in $ D(x_0, r\prod_{m=1}^n(1-b_m)) $, and such that
	\[F^E_{n}(D(x^E_0, r\prod_{m=1}^{n}(1-b_m)))\subset D(x^E_{n}, \varepsilon\cdot M^{n}).\]
\end{claim}

 Letting $ n\to \infty $, we get that all inverse branches of $ f_{E}$ sending $ x^E_0 $ to $ x^E_n $ are well-defined in $ D(x_0, rP) $, with $ r>0 $.  Moreover, as $ n\to \infty $, $ \textrm{diam} \left( F^E_{n} ( D(x_0, rP)) \right) \leq \varepsilon \cdot M^{n}\to 0$.

This ends the proof of Proposition \ref{prop65}. \hfill $ \square $

\subsection{Parabolic Pesin theory for entire functions}\label{subsect-entire-parabolic}

We end this section by showing that, when we are dealing with an entire function, it is enough to ask that there are only finitely many critical values in $ \partial U_{+\varepsilon} $. To see this, it it enough to show that inverse branches for $ f^E$ are well-defined far from the orbit of critical values of $ f $ and from exceptional points (points with finite backwards orbit; any entire function has at most two exceptional points \cite[p. 6]{Bergweiler}). 

	\begin{lemma}\label{proof-pesin-entire-blow-up}
	Let $ x\in E $. Then, the inverse branch $ F^E_1 $ sending $ f^E(x) $ to $x$ is well-defined in $ D(f^E(x), r) $, $ r<\varepsilon $, as long as $ D(f^E(x), r) \cap CV(f^{T(x)})=\emptyset $,  and there are no exceptional points in $ D(f^E(x), r) $.
\end{lemma}

\begin{proof}
	If there are no exceptional points in $ D(f^E(x), r) $, it is easy to see that there exists $ R>0 $ such that $ f^{T(x)}|_{D(x,R) }$ is holomorphic, and $ f^{T(x)}(D(x,R))\supset D(f^{T(x)}(x),r) $ (see e.g. \cite[Corol. 14.2]{milnor}). Hence, all obstructions to define the inverse branch $ F^E_1 $ come from the critical values of $ f^{T(x)} $, but we assumed there are none.
\end{proof}

Lemma \ref{proof-pesin-entire-blow-up} is telling us that critical values together with exceptional points are the `singular values of $ f^E $', and there are only finitely many of them. Hence, one can prove a result analogous to Proposition \ref{prop65}.

\begin{thm}{\bf (Inverse branches are well-defined almost everywhere)}
	Let $ f\colon\mathbb{C}\to\mathbb{C}$ be an entire function, and let $U$ be a parabolic basin or   Baker domain. 
	Assume 
	\begin{enumerate}[label={\em (\alph*)}]
		\item the Denjoy-Wolff point of the associated inner function is not a singularity;
		\item$ \log \left| f'\right| \in L^1(\mu) $, and  $ \int_{\partial U} \log \left| f'\right| d\mu>0$;
		\item  there exists $ \varepsilon >0 $ such that the set of critical values of $ f $ in $ \partial U_{+\varepsilon} $ are finite.
	\end{enumerate}
	Then, for $ \widetilde{\mu} $-almost every backward orbit $ \left\lbrace x_n\right\rbrace _n \in \widetilde{\partial U}$, there is $ r_0\coloneqq r_0(\left\lbrace x_n\right\rbrace _n)>0 $ such that
	\begin{enumerate}[label={\em (\roman*)}]
		\item the inverse branch $ F_n $ sending $x_0 $ to $ x_n $ is well-defined in $ D(x_0,r_0) $;
		\item  $ \textrm{\em diam } F^j_m(D(x_0, r))\to 0 $, as $ j\to\infty $. 
	\end{enumerate}
\end{thm}

\section{Dynamics of  centered inner functions. Corollary \ref{corol:D}}\label{sect-inner-functions}

In this section we apply the techniques developed previously to a particular type of self-maps of the unit disk $ \mathbb{D} $, the so-called \textit{inner functions}. Recall that  $ g(0)=0 $, we say that $ g $ is a {\em centered inner function}.

Every inner function induces a measure-theoretical dynamical system $ g^*\colon\partial\mathbb{D}\to \partial\mathbb{D} $ defined $ \lambda $-almost everywhere. For 
centered inner functions, $ g^*|_{\partial\mathbb{D}} $ preserves the Lebesgue measure $ \lambda$ in $ {\partial\mathbb{D}}$, and $ g^*|_{\partial\mathbb{D}} $ is ergodic. Hence, the radial extension of a centered inner functions is a  good candidate  to perform Pesin theory. 
Therefore, we shall see Corollary \ref{corol:D} as an application of the work done in Theorem \ref{teo:A}, for a particular class of inner functions (centered inner functions with finite entropy, i.e. $ \log \left| g'\right| \in L^1(\lambda) $).

As in the previous sections, we rewrite Corollary \ref{corol:D} in terms of Rokhlin's natural extension (Thm. \ref{thm-pesin-inner}). 
 Indeed,  $ (\partial \mathbb{D}, \mathcal{B}(\partial\mathbb{D}), \lambda) $ is a Lebesgue space (it is isomorphic, in the measure-theoretical sense, to the unit interval), and hence Theorem \ref{thm-natural-extension} guarantees the existence of Rokhlin's natural extension.  Thus, $ \widetilde{\partial \mathbb{D}} $ is the space of backward orbits $ \left\lbrace \xi_n\right\rbrace _n\subset \partial \mathbb{D} $, with $ g^*(\xi_{n+1})=\xi_n $ for $ n\geq 0 $, and $ \widetilde{g^*}\colon  \widetilde{\partial \mathbb{D}}\to \widetilde{\partial \mathbb{D}}$ is the automorphism which makes  the following diagram commute.

\[\begin{tikzcd}[row sep=-0.2em,/tikz/row 3/.style={row sep=4em}]
	...\arrow{r}{\widetilde{g}} &	\widetilde{\partial \mathbb{D}} \arrow{rr}{\widetilde{g}} & &	\widetilde{\partial \mathbb{D}} \arrow{rr}{\widetilde{g}} & &	\widetilde{\partial \mathbb{D}}\arrow{r}{\widetilde{g}} &...\\ 
	&	{\scriptstyle \left\lbrace \xi_{n+2}\right\rbrace _n} \arrow{ddddddd}{\pi_{\mathbb{D},n}}& &{\scriptstyle \left\lbrace \xi_{n+1}\right\rbrace _n} \arrow{ddddddd}{\pi_{\mathbb{D},n}}& &	{\scriptstyle \left\lbrace \xi_{n}\right\rbrace _n }\arrow{ddddddd}{\pi_{\mathbb{D},n}} \\ \\ \\	\\ \\ \\ \\...\arrow{r}{g} &\partial \mathbb{D}\arrow{rr}{g}  &&
	\partial \mathbb{D} \arrow{rr}{g}  && \partial \mathbb{D}\arrow{r}{g} &...\\ &{\scriptstyle \xi_{n+2} }&&{\scriptstyle \xi_{n+1} }&&{\scriptstyle \xi_n}
\end{tikzcd}
\]

In this way, we can rephrase Corollary \ref{corol:D} as follows.
\begin{thm}{\bf (Pesin theory for centered inner function)}\label{thm-pesin-inner}
	Let $ g\colon\mathbb{D}\to\mathbb{D} $ be an inner function, such that $ g(0)=0 $, and $ \log \left| g'\right|\in L^1(\partial\mathbb{D})$. Fix $ \alpha\in (0,\pi/2) $.  Then, for $ \widetilde{\lambda }$-almost every backward orbit $ \left\lbrace \xi_n\right\rbrace_n \subset \partial\mathbb{D}$, there exists $ \rho >0 $ such that
	the inverse branch $ G_n $ of $ g^n $ sending $\xi_0 $ to $ \xi_n $ is well-defined in $ D(\xi_0,\rho) $, and, for all $ \rho_1\in (0, \rho) $, \[ G_n(R_{\rho_1}(\xi_0))\subset \Delta_{\alpha, \rho_1}(\xi_n).\]
	Moreover, the set of singularities $ E(g) $ has zero $ \lambda $-measure.
\end{thm}

Using that $ g^*|_{\partial\mathbb{D}} $ is ergodic and recurrent with respect to $ \lambda  $, it follows that for $ \widetilde{\lambda }$-almost every backward orbit $ \left\lbrace \xi_n\right\rbrace_n \subset \partial\mathbb{D}$ and every set $ A\subset\partial\mathbb{D} $ of positive measure, there exists  a sequence $ n_k\to\infty $ such that $ \xi_{n_k}\in A $ (Prop. \ref{prop-ergodic-properties-natural-ext}). Hence, it is clear that Theorem \ref{thm-pesin-inner} implies Corollary \ref{corol:D}.

\begin{proof}[Proof of Theorem \ref{thm-pesin-inner}]
	Proceeding exactly as in Theorem \ref{teo:A}, we find that, for $ \widetilde{\lambda }$-almost every backward orbit $ \left\lbrace \xi_n\right\rbrace_n \subset \partial\mathbb{D}$, there exists $ \rho_0 >0 $ such that
	the inverse branch $ G_n $ of $ g^n $ sending $\xi_0 $ to $ \xi_n $ is well-defined in $ D(\xi_0,\rho_0) $. Note that all inverse branches $ \left\lbrace G_n\right\rbrace_n  $ are well-defined in a disk of uniform radius, namely in $ D(\xi_0,\rho_0) $. Hence, we can apply Proposition  \ref{prop-radial-limits}, to see that, for all $ \alpha\in (0, \pi/2) $ there exists $ \rho<\rho_0 $ such that for all $ \rho_1\in (0, \rho) $, \[ G_n(R_{\rho_1}(\xi_0))\subset \Delta_{ \alpha, \rho_1}(\xi_n).\]
	
	It remains to see that the set of singularities has zero Lebesgue measure. Assume on the contrary that the set of singularities $ E(g) $ has positive measure. Then, 
	we can take $ \left\lbrace \xi_n\right\rbrace_n$ visiting $ E(g) $ infinitely often, and satisfying that the inverse branches $ \left\lbrace G_n\right\rbrace _n $ realizing such a  backward orbit are well-defined in $ D(\xi_0, \rho) $. Consider 
	\[ K\coloneqq \bigcup\limits_{n\geq1} G_n(D(\xi_0, \rho)). \] 
	
	We claim that no point in $ K $ is a singularity for $ g $. Indeed, for any $ \xi\in K $, there exists $ n\geq 1 $ such that $ \xi\in G_n(D(\xi_0, \rho)) $. Hence, \[g|_{G_n(D(\xi_0, \rho))}\colon G_n(D(\xi_0, \rho))\longrightarrow G_{n-1}(D(\xi_0, \rho))\] is univalent, so $ \xi $ cannot be a singularity for $ g $. This is a contradiction with the fact that $ K\cap E(g)\neq\emptyset $, and ends the proof of Corollary \ref{corol:D}.
\end{proof}

\section{Inner function associated to a Fatou component and Rokhlin's natural extension}\label{sect-inner-function-associatedFC}
Let $ f\in\mathbb{K} $, and let $ U $ be an invariant Fatou component for $ f $, which we assume to be simply connected. Consider $ \varphi\colon\mathbb{D}\to U $ to be a Riemann map. Then,  $ f\colon U\to U $ is conjugate by $ \varphi $ to a holomorphic map $ g\colon\mathbb{D}\to\mathbb{D} $, i.e. $ f\circ\varphi=\varphi\circ g$ (Sect. \ref{subsect-components-Fatou-K}).

The conjugacy $ f\circ\varphi=\varphi\circ g $ extends almost everywhere to $ \partial \mathbb{D} $ by means of the radial extensions $ \varphi^*\colon \partial \mathbb{D}\to\partial U $ and $ g^*\colon\partial\mathbb{D}\to\partial\mathbb{D} $. 
More precisely, consider
the following subsets of $ \partial \mathbb{D} $.
\[\Theta_E\coloneqq \left\lbrace \xi\in\partial\mathbb{D}\colon \varphi^*(\xi)\in E(f) \right\rbrace  \] 
\[\Theta_\Omega\coloneqq \left\lbrace \xi\in\partial\mathbb{D}\colon \varphi^*(\xi)\in\Omega(f) \right\rbrace \] 
Since $ E(f) $ is countable, $ \lambda(\Theta_E)=0 $, so $ \lambda(\Theta_\Omega)=1 $. Then, $ f\circ \varphi=\varphi\circ g $ extends for the radial extensions in $ \Omega_\Theta $, as shown in the following lemma.
\begin{lemma}{\bf(Radial limits commute, {\normalfont \cite[Lemma 5.5]{Jove}})}\label{lemma-radial-limit}
	Let $ \xi\in\Theta_\Omega $, then $ g^*(\xi) $ and $ \varphi^*(g^*(\xi)) $ are well-defined, and \[f(\varphi^*(\xi))=\varphi^*(g^*(\xi)).\] 
\end{lemma}
In this section we show that one can go further and relate backward orbits for the radial extension of the inner function $ g^* $ with backward orbits for the boundary map $ f|_{\partial U} $. Moreover, we will show how the natural extensions of $ (\partial \mathbb{D}, \lambda, g^*) $ and $ (\partial U, \omega_U, f) $ are related. 

To do so, first we have to establish, in the spirit of Lemma \ref{lemma-radial-limit}, a relation between backward orbits for $ g^* $ and backward orbits for $ f $. More precisely, we prove that backward orbits associated to a well-defined sequence of inverse branches indeed commute by the Riemann map, as long as the radial limit at the initial point exists.
\begin{prop}{\bf (Backward orbits commute)}\label{prop-backwards-orbits-commute}
	Let $ \left\lbrace \xi_n\right\rbrace _n\subset\partial\mathbb{D} $ be a backward orbit for $ g^* $. Assume $ \varphi^*(\xi_0 )$ exists. Then, $ \varphi^*(\xi_n) $ exists for all $ n\geq 1 $ and \[f(\varphi^*(\xi_{n+1}))=\varphi^*(g^*(\xi_{n+1}))=\varphi^*(\xi_{n}).\]
\end{prop}
\begin{proof}
	We note that, using an inductive argument, it is enough to prove that, if 	$ \varphi^*(\xi_0 )$ exists and $ \xi_1\in\partial\mathbb{D} $ is such that $ g^*(\xi_1)=\xi_0 $, then $ \varphi^*(\xi_1) $ is well-defined, and \[f(\varphi^*(\xi_{1}))=\varphi^*(g^*(\xi_{1}))=\varphi^*(\xi_{0}).\] Let $ R_{\xi_0} $ be the radius at $ \xi_0$. Then, $\varphi (R_{\xi_0}) $ is a curve landing  at $ \varphi^*(\xi_0) $, and there is a curve $ \gamma $ landing at $ \xi_1 $, with $ g(\gamma)= R_{\xi_0}$. Then, $ f(\varphi (\gamma))=\varphi(R_{\xi_0}) $ is a curve landing at  $ \varphi^*(\xi_0 )$. Since preimages of a point under a holomorphic map are discrete and the singularities of $ f $ are countable, $ \varphi (\gamma) $ lands at a point on $ \widehat{\partial} U $, which, by Lindelöf's theorem (see e.g. \cite[Thm. 2.2]{CarlesonGamelin}) coincides with $ \varphi^*(\xi_1) $ (which in particular is well-defined and satisfies $ f(\varphi^*(\xi_{1}))=\varphi^*(\xi_{0}) $).
\end{proof}

We are interested now in the interplay between the backward orbits for the associated inner function $ g $, and the backward orbits for $ f $ in the dynamical plane. 
According to Section \ref{subsect-Rokhlin's-natural-ext}, we can consider the natural extension $ (\widetilde{\partial\mathbb{D}}, \widetilde{\lambda}, \widetilde{g^*}) $ of $ (\partial\mathbb{D},  \lambda, g^*) $, given by the projecting morphisms $ \left\lbrace \pi_{\mathbb{D},n}\right\rbrace _n $, and the natural extension $ (\widetilde{\partial U}, \widetilde{\omega_U}, \widetilde{f}) $ of $ 	(\partial U,  \omega_U, f) $, given by the projecting morphisms $ \left\lbrace \pi_{U,n}\right\rbrace _n $. 
We are interested in relating both natural extensions.

In views of Proposition \ref{prop-backwards-orbits-commute}, it is clear that the  transformation 
\begin{align*}
	\widetilde{\varphi^*}&\colon\widetilde{\partial\mathbb{D}}\longrightarrow\widetilde{\partial U}
\\
&\left\lbrace \xi_n\right\rbrace _n\mapsto 	\left\lbrace \varphi^*(\xi_n)\right\rbrace _n
\end{align*}
is well-defined, and the following diagram commutes almost everywhere. 

\[\begin{tikzcd}[row sep=-0.2em,/tikz/row 2/.style={row sep=4em},/tikz/row 4/.style={row sep=4em},/tikz/row 6/.style={row sep=4em}]
	...\arrow{r}{\widetilde{g}} &	\widetilde{\partial \mathbb{D}} \arrow{rr}{\widetilde{g}} & &	\widetilde{\partial \mathbb{D}} \arrow{rr}{\widetilde{g}} & &	\widetilde{\partial \mathbb{D}}\arrow{r}{\widetilde{g}} &...\\ 
	&	{\scriptstyle \left\lbrace \xi_{n+2}\right\rbrace _n}\arrow[ddddd, shift right=3, bend right=12, swap, "\widetilde{\varphi^*}"] \arrow{d}{\pi_{\mathbb{D},n}}& &{\scriptstyle \left\lbrace \xi_{n+1}\right\rbrace _n}\arrow[ddddd, shift right=3, bend right=12, swap, "\widetilde{\varphi^*}"] \arrow{d}{\pi_{\mathbb{D},n}}& &	{\scriptstyle \left\lbrace \xi_{n}\right\rbrace _n }\arrow[ddddd, shift right=3, bend right=12, swap, "\widetilde{\varphi^*}"]\arrow{d}{\pi_{\mathbb{D},n}} \\ ...\arrow{r}{g} &\partial \mathbb{D}\arrow{rr}{g}  &&
	\partial \mathbb{D} \arrow{rr}{g} && \partial \mathbb{D}\arrow{r}{g} &...\\ &{\scriptstyle \xi_{n+2} }\arrow{d}{\varphi^*} &&{\scriptstyle \xi_{n+1} }\arrow{d}{\varphi^*} &&{\scriptstyle \xi_n}\arrow{d}{\varphi^*} 
	\\ ...\arrow{r}{f} &\partial U \arrow{rr}{f}  &&
	\partial U\arrow{rr}{f}  && \partial U \arrow{r}{f} &...\\ &{\scriptstyle \varphi^*(\xi_{n+2}) }&&{\scriptstyle \varphi^*(\xi_{n+1}) }&&{\scriptstyle \varphi^*(\xi_n)}
	\\ ...\arrow{r}{\widetilde{f}} &\widetilde{\partial U} \arrow{rr}{\widetilde{f}}\arrow[swap]{u}{\pi_{U,n}}   &&
	\widetilde{\partial U}\arrow{rr}{\widetilde{f}}\arrow[swap]{u}{\pi_{U,n}}  && \widetilde{\partial U} \arrow{r}{\widetilde{f}}\arrow[swap]{u}{\pi_{U,n}}  &...\\ &{\scriptstyle \left\lbrace \varphi^*(\xi_{n+2})\right\rbrace _n }&&{\scriptstyle \left\lbrace \varphi^*(\xi_{n+1})\right\rbrace _n }&&{\scriptstyle \left\lbrace \varphi^*(\xi_n)\right\rbrace _n}
\end{tikzcd}
\]

Now we claim that $ \widetilde{\varphi^*} $ is measure-preserving. Indeed, one may take a basis for the $ \sigma $-algebra in $ \widetilde{\partial U} $ made of sets of the form $ \pi^{-1}_{U,n} (A) $, where $ A\subset \partial U $ measurable, and $ n\geq 0 $.  It is enough to prove that $ 	\widetilde{\varphi^*} $ preserves the measure of these sets. Indeed, using that  $ \varphi^*\circ \pi_{\mathbb{D}, n}=  \pi_{{U}, n}\circ\widetilde{\varphi^*} $ $ \widetilde{\lambda} $-almost everywhere, we have
\[\widetilde{\omega_U}(\pi^{-1}_{U,n} (A))=\omega_U (A)=\lambda (\varphi^*(A))=\widetilde{\lambda}(\pi^{-1}_{\mathbb{D},n}\circ (\varphi^*)^{-1}(A))= \widetilde{\lambda}( (\widetilde{\varphi^*})^{-1}\circ\pi^{-1}_{U,n}(A)),\]where $ A\subset \partial U $ measurable, and $ n\geq 0 $, as desired. 
In other words, $ \widetilde{\omega_U} $ is the push-forward of $ \widetilde{\lambda} $ by $ \widetilde{\varphi^*} $.

Hence,  the following diagram
\[\begin{tikzcd}[column sep = large, row sep = large]
	(\partial\mathbb{D}, \lambda, g^*)  \arrow[swap]{d}{\varphi^*}& (\widetilde{\partial\mathbb{D}}, \widetilde{\lambda}, \widetilde{g^*})\arrow{d}{\widetilde{\varphi^*}} \arrow[swap]{l}{\left\lbrace \pi_{\mathbb{D},n}\right\rbrace _n} \\	
	(\partial U, \omega_U, f) & (\widetilde{\partial U},  \widetilde{\omega_U}, \widetilde{f})\arrow[swap]{l}{\left\lbrace \pi_{U,n}\right\rbrace _n} .
\end{tikzcd}
\] commutes almost everywhere.

\begin{prop}{\bf (Generic inverse branches commute)}\label{prop-generic-inverse-branches}
	Let $ f\in \mathbb{K}$, and let $U$ be an invariant simply connected Fatou component for $ f $. Let $ \varphi\colon\mathbb{D}\to U $ be a Riemann map, and let $ g\colon\mathbb{D}\to\mathbb{D} $ be the  inner function  associated to $ (f, U) $ by $ \varphi $. Assume the following conditions are satisfied.
	\begin{enumerate}[label={\em (\alph*)}]
		\item For $ \widetilde{\omega_U }$-almost every  backward orbit $ \left\lbrace x_n\right\rbrace _n \subset\partial U$, there exists $ r>0 $ such that the inverse branch $ F_n $ sending $x_0 $ to $ x_n $ is well-defined in $ D(x_0,r) $.
		\item For $ \widetilde{\lambda} $-almost every backward orbit $ \left\lbrace \xi_n\right\rbrace _n\subset\partial\mathbb{D} $  there exists $ \rho>0 $ such that the inverse branch $ G_n $ sending $\xi_0 $ to $ \xi_n $ is well-defined in $ D(\xi_0,\rho) $.
	\end{enumerate}
	Then, for $ \widetilde{\lambda} $-almost every backward orbit $ \left\lbrace \xi_n\right\rbrace _n\subset\partial\mathbb{D} $  there exists $ \rho, r>0 $ such that the inverse branch $ G_n $ sending $\xi_0 $ to $ \xi_n $ is well-defined in $ D(\xi_0,\rho) $, the inverse branch $ F_n $ sending $\varphi^*(\xi_0) $ to $ \varphi^*(\xi_n)  $ is well-defined in $ D(\varphi^*(\xi_0) ,r) $, and $ \varphi^*\circ G_n (\xi_0)=F_n\circ \varphi^*(\xi_0) $, for all $ n\geq 0 $.
\end{prop}

We note that, if $ \widetilde{\omega_U }$-almost every  backward orbit $ \left\lbrace x_n\right\rbrace _n \subset\partial U$ satisfies an additional property (such as the ones proved in \ref{teo:A}), then it is straightforward to see that, for $ \widetilde{\lambda} $-almost every backward orbit $ \left\lbrace \xi_n\right\rbrace _n\subset\partial\mathbb{D} $, the backward orbit $ \left\lbrace x_n\coloneqq \varphi^*(\xi_n)\right\rbrace _n $ satisfies this additional property.

\begin{proof}[Proof of Proposition \ref{prop-generic-inverse-branches}]
	The proof follows directly from the previous construction. Indeed, one can write the first assumption as: for $ \widetilde{\lambda} $-almost every backward orbit $ \left\lbrace \xi_n\right\rbrace _n\subset\partial\mathbb{D} $, there exists $ r>0 $ such that the inverse branch $ F_n $ sending $\varphi^*(\xi_0) $ to $ \varphi^*(\xi_n )$ is well-defined in $ D(\varphi^*(\xi_0),r) $. Since the intersection of sets of full measure has full measure, we have that inverse branches $ G_n $ and $ F_n $ are well-defined along the backward orbit of $ \xi_0 $ and $ \varphi^*(\xi_0) $.
	By Proposition \ref{prop-backwards-orbits-commute}, such inverse branches commute.
\end{proof}

\begin{remark*}\label{remark-phi-star}
	It follows from the previous construction that in Theorems A and B one can find first the backward orbit $ \left\lbrace \xi_n\right\rbrace _n\subset\partial\mathbb{D} $ and define the backward orbit in the dynamical plane as their image by $ \varphi^* $. Moreover, one can choose a countable collection of sets $ \left\lbrace K_k\right\rbrace _k\subset\partial\mathbb{D} $ and ask that there exists a sequence $ n_k\to \infty $ with $ \xi_{n_k}\in K_k $.
\end{remark*}
\section{Application: periodic boundary points. Corollary \ref{corol:C}}\label{sect-periodic-points}

One application of Pesin theory in holomorphic dynamics is to prove that for some invariant Fatou components, periodic points are dense in their boundary. This was done in the seminal paper of Przytycki and Zdunik \cite{Przytycki-Zdunik} for simply connected attracting basins of rational maps (note that in this paper it is proved that periodic points are dense in the boundary of {\em every} attracting or parabolic basin of a rational map, but the proof relies on a different technique).  In the spirit of \cite{Jove}, we aim to prove a similar result for transcendental maps.

The goal in this section is to prove Corollary \ref{corol:C}, which states that, under the hypotheses of either Theorem \ref{teo:A} or \ref{teo:B}, plus an extra hyptothesis on the singular values in $ U $, accessible periodic boundary points are dense. 

In view of the theory developed in the previous sections based on working in the space of backward orbits given by Rokhlin's natural extension, we shall formulate an alternative (and more natural) version of \ref{corol:C}, in terms of backward orbits.
\begin{thm}{\bf (Periodic points are dense)}\label{thm-punts-periòdics}
	Let $ f\in \mathbb{K}$, and let $U$ be an invariant simply connected Fatou component for $ f $. Let $ \varphi\colon\mathbb{D}\to U $ be a Riemann map, and let $ g\colon\mathbb{D}\to\mathbb{D} $ be the  inner function  associated to $ (f, U) $ by $ \varphi $. Assume the following conditions are satisfied.
\begin{enumerate}[label={\em (\alph*)}]
	\item For $ \widetilde{\omega_U }$-almost every  backward orbit $ \left\lbrace x_n\right\rbrace _n \subset\partial U$, there exists $ r>0 $ such that the inverse branch $ F_n $ sending $x_0 $ to $ x_n $ is well-defined in $ D(x_0,r) $, for every subsequence $\left\lbrace x_{n_j}\right\rbrace _j$ with $ x_{n_j}\in D(x_0, r) $, 
	$ \textrm{\em diam } F_{n_j}(D(x_0, r))\to 0$, as $ j\to\infty $.
	\item For $ \widetilde{\lambda} $-almost every backward orbit $ \left\lbrace \xi_n\right\rbrace _n\subset\partial\mathbb{D} $  there exists $ \rho>0 $ such that the inverse branch $ G_n $ sending $\xi_0 $ to $ \xi_n $ is well-defined in $ D(\xi_0,\rho) $.
\end{enumerate}
Then, accessible periodic points are dense in $ \partial U $.
\end{thm}

However, we aim to give a proof of the density of periodic boundary points which does not use Rohklin's natural extension. To do so, we state Theorem \ref{thm-punts-periòdics} in a slightly different (and stronger) way.

\begin{thm}{\bf (Periodic points are dense)}\label{thm-punts-periòdics2}
	Let $ f\in \mathbb{K}$, and let $U$ be an invariant simply connected Fatou component for $ f $. Let $ \varphi\colon\mathbb{D}\to U $ be a Riemann map, and let $ g\colon\mathbb{D}\to\mathbb{D} $ be the  inner function  associated to $ (f, U) $ by $ \varphi $. Assume that for every countable sequence of measurable sets $ \left\lbrace K_k\right\rbrace _k\subset\partial\mathbb{D} $ with $ \lambda(K_k)>0 $ and  $ \lambda $-almost every $ \xi\in\partial\mathbb{D}$, there exists a backward orbit $ \left\lbrace \xi_n\right\rbrace _n\subset \partial\mathbb{D} $, such that 
	\begin{enumerate}[label={\em (\alph*)}]
		\item $ \xi=\xi_0 $ and there exists $ \rho>0 $ such that the inverse branch $ G_n $ sending $\xi_0 $ to $ \xi_n $ is well-defined in $ D(\xi_0,\rho) $, and there exists $ n_k\to \infty $ with $ \xi_{n_k}\in K_k $;
		\item for the backward orbit $ \left\lbrace x_n\coloneqq \varphi^*(\xi_n)\right\rbrace _n \subset\partial U$, there exists $ r>0 $ such that the inverse branch $ F_n $ sending $x_0 $ to $ x_n $ is well-defined in $ D(x_0,r) $, for every subsequence $\left\lbrace x_{n_j}\right\rbrace _j$ with $ x_{n_j}\in D(x_0, r) $, 
		$ \textrm{\em diam } F_{n_j}(D(x_0, r))\to 0$, as $ j\to\infty $.
	\end{enumerate}
	Then, accessible periodic points are dense on $ \partial U $.
\end{thm}

According to Proposition \ref{prop-generic-inverse-branches}, it is clear that Theorem \ref{thm-punts-periòdics2} implies \ref{thm-punts-periòdics}. We show now how to deduce Corollary \ref{corol:C} from Theorem \ref{thm-punts-periòdics}, and later we give the proof of  Theorem \ref{thm-punts-periòdics2}.

\begin{proof}[Proof of Corollary \ref{corol:C}]
	On the one hand, it is clear that, by the conclusion of \ref{teo:A} and \ref{teo:B}, the second requirement of Theorem \ref{thm-punts-periòdics} holds.
	
	 On the other hand, we have to see the assumption of the existence of a crosscut neighbourhood $ N_C $ in $ U $ with $ N_C\cap P(f)=\emptyset $ implies (b). Indeed, $ \varphi^{-1}(N_C) $ is a crosscut neighbourhood in $ \mathbb{D} $ which contains no postsingular value for the inner function. Since $ g^*|_{\partial\mathbb{D}} $ is ergodic and recurrent, for $ \lambda $-almost $ \xi\in \partial\mathbb{D} $, there exists $ \rho>0 $ such that, for all $ n\geq 0 $, all inverse branches of $ g^n $ are well-defined in $ D(\xi, \rho) $. Denote this set of backward orbits by $ \widetilde{A} $. We have to show that $ \widetilde{A} $ has full $ \widetilde{\lambda} $-measure in $ \widetilde{\partial\mathbb{D}} $. Indeed, note that
	 $\widetilde{A}=\pi_{\mathbb{D},0}^{-1}(\pi_{\mathbb{D}, 0}(\widetilde{A}))$, since the set $ \widetilde{A} $ is made of {\em all} backward orbit with initial point in $ \pi_{\mathbb{D}, 0}(\widetilde{A})$. Since $ \lambda(\pi_{\mathbb{D}, 0}(\widetilde{A}))=1$
and $ \pi_{\mathbb{D}, 0} $ is measure-preserving,	  this already implies the requirement (b) in Theorem \ref{thm-punts-periòdics}. 
\end{proof}

\subsection{Proof of Theorem \ref{thm-punts-periòdics2}}
Let $ x\in\partial U $ and $ R>0 $, we have to show that $ f $ has a repelling periodic point in $ D(x, R)\cap\partial U $, which is accessible from $ U $. 

We split the proof in several intermediate lemmas. We start by proving the existence of a backward orbit $ \left\lbrace \xi_n\right\rbrace _n \subset \partial\mathbb{D}$ such that for both $ \left\lbrace \xi_n\right\rbrace _n$ and $ \left\lbrace \varphi^*(\xi_n)\right\rbrace _n$ the corresponding inverse branches are well-defined (and conformal), and certain estimates on the contraction are achieved.

In the sequel, we fix $ \alpha\in (0, \pi/2) $, and we take all Stolz angles of opening $  \alpha$, although in the notation we omit the dependence.

\begin{sublemma}\label{sublemma-puntperiodic1}
	There exists a backward orbit $ \left\lbrace \xi_n\right\rbrace _n \subset \partial\mathbb{D}$, and constants $ m\in\mathbb{N} $, $ 0<\rho_m\leq  \rho $, and $ r\in (0, R/2) $  such that:
	\begin{enumerate}[label={\em (1.\arabic*)}]
		\item\label{item-1} $ x_0\coloneqq \varphi^* (\xi_0)$ and $ x_m\coloneqq \varphi^* (\xi_m)$ are well-defined, and $ x_0\in D(x, R/2) $ and $ x_m\in D(x_0, r/3) $;
		\item\label{item-2}  the inverse branch $ F_m $ of $ f^m $ sending $ x_0 $ to $ x_m $ is well-defined in $ D(x_0, r) $, and 
$ \textrm{\em diam } F_m(D(x_0, r))<r/3 $;
\item \label{item-3} the inverse branch $ G_m $ of $ g^m $ sending $ \xi_0 $ to $ \xi_m $ is well-defined in $ D(\xi_0, \rho_m) $, and satisfies \[  G_m(R_{\rho_m}(\xi_0))\subset \Delta_{ \rho_m}(\xi_m);\] 
	\item \label{item-4} $ \Delta_{\rho}(\xi_0)\cap  \Delta_{ \rho}(\xi_m)\neq\emptyset$, and, if $ z\in\Delta_{ \rho} (\xi_0)\cup \Delta_{ \rho} (\xi_m)$,  then $ \varphi(z)\in D(x_0, r) $.
	\end{enumerate}
\end{sublemma}	

\begin{proof}
Let $ A_n= D(x^n, r_n) $ be a countable basis for $ D(x, R) $ with the Euclidean topology, where $ x^n\in\partial U $ and $ A_n\subset D(x,R) $. 

In order to apply the hyptothesis of the theorem, we shall construct an appropriate countable sequence of measurable sets $ \left\lbrace K_k\right\rbrace _k $ of $ \partial\mathbb{D} $. We do it as follows.

For all $ n\geq 0 $, let 
\[ K^{n}=\left\lbrace \xi\in\partial\mathbb{D}\colon \varphi^*(\xi)\in D(x^n, {r_n}/{2}) \right\rbrace . \] It is clear that $ \lambda(K^{n})>0 $. By the Lehto-Virtanen Theorem, the angular limit exists whenever the radial limit exists. Therefore, there exists $ \rho_n>0 $ small enough so that
\[ K^n_{\rho_n}=\left\lbrace \xi\in K^n\colon \Delta_{\rho_n}(\xi)\subset D(x^n, {r_n}/{2}) \right\rbrace  \]
has positive $ \lambda $-measure. We can assume that every point in $ K^n_{\rho_n} $  is a Lebesgue density point  for $ K^n_{\rho_n} $. Then, if we take
$ \xi^n \in  K^n_{\rho_n} $, there exists a circular interval  $ I_{\xi^n} $ around $ \xi^n $ such that for any $ \zeta_1, \zeta_2\in I_{\xi^n}  $,
\[\Delta_{\rho_n}(\zeta_1)\cap  \Delta_{ \rho_n}(\zeta_2)\neq\emptyset.\] Then, $ K^n_{\rho_n} \cap  I_{\xi^n} $ has positive $ \lambda $-measure.
Note that this property only depends on the length of the interval, as long as $  \xi^n $ is a Lebesgue density point for $ K^n_{\rho_n} $.
Then, it is clear that there exist finitely many circular intervals $ I_1^n, \dots, I^n_{i_n} $ with this property. 

Let \[K^{1,n}_{*, i} \coloneqq K^n_{\rho_n} \cap  I^n_{i},\hspace{0.5cm} i=1, \dots, i_n,\]
\[
K^{1,n}_*\coloneqq\left\lbrace K^{1,n}_{*,1},\dots, K^{1,n}_{*,i_n}\right\rbrace.
\]
Then, we define the set $ K^{j,n}_*$, as before, but replacing $ \rho_n $ by $ \rho_n /2^j$.

Having introduced all this notation of the sets $ \left\lbrace K^{j,n}_* \right\rbrace _{n,j}$, we arrange the sequence $ \left\lbrace K_k\right\rbrace _k $ as follows. We construct this sequence of sets inductively, adding at each step finitely many sets.
Indeed, let us start by putting the block $ K^{1,1}_*\coloneqq\left\lbrace K^{1,1}_{*,1},\dots, K^{1,1}_{*,i_1}\right\rbrace $ as the first elements of the sequence. Then,
for the $ k $-th step of the induction, we consider $ A_{k} $ and let $ A_{k_1}, \dots,  A_{k_n}$ be all the sets of $ A_1,\dots, A_n $ such that $ A_n\subset A_{k_i} $. Then, we add to the sequence the blocks
\[K^{1,k_1}_*, \dots K^{1,k_n}_*, \dots, K^{k,k_1}_*, \dots, K^{k,k_n}_*.\]

 Basically, the idea is that, when one set is in the sequence $ \left\lbrace K_k\right\rbrace _k $ for the first time, then it appears infinitely often. Moreover, the set of points in $ \left\lbrace K_k\right\rbrace _k $ has  measure $ \lambda ((\varphi^*)^{-1}(D(x,R))) $. Indeed, the set of points in $ \partial\mathbb{D} $ for which the radial limit exists has full measure. Let $ \zeta $ be one of such points. Then, $ \varphi^*(\zeta) \in A_j$, for some $ j $, and for $ \rho>0 $ small enough, $ \Delta_\rho (\zeta)\subset A_j $. Then, there exists $ n\geq 0 $ such that $ A_n\subset A_j $ and $ \rho<\rho_j/2^n $, so $ \zeta \in K^n_{*, k_n}$, as desired.

By the assumption of the theorem, for $ \lambda $-almost every $ \xi_0\in\partial\mathbb{D} $, 
there exists a backward orbit $ \left\lbrace \xi_n\right\rbrace _n $ such that  the hypothesis on the definition of the inverse branches for $ \left\lbrace \xi_n\right\rbrace _n $ and $ \left\lbrace x_n\coloneqq \varphi^*(\xi_n)\right\rbrace _n $ are accomplished, and there exists $ n_k\to \infty  $ with $ \xi_{n_k}\in K_k $. 

Without loss of generality, we assume $ \xi_0 $ is chosen so that $ x_0\in D(x, R/2) $. Let $ r>0 $ be such that the inverse branches realizing the backward orbit $ \left\lbrace x_n\right\rbrace _n $ are well-defined in $ D(x_0,r) $. There is no loss of generality on assuming $ r\in (0, R/2) $. 

On the one hand, since $ \left\lbrace A_n \right\rbrace _n$ is a basis for $ D(x,R) $, there exists $ n_0 $ such that $$ x_0\in A_{n_0}\subset D(x_0, r/3) , $$ and  $ \xi_0 \in K^{n_0}_*$,  by the previous remark. In particular, for $ \rho_{n_0} $,  
\[\Delta_{\rho_{n_0}}(\xi_0)\subset A_{n_0}\subset D(x_0, r/3).\]
On the other hand, 
by the construction of the sets $ \left\lbrace K_n\right\rbrace _n  $, the backward orbit visits $ D(x_0,r) $ infinitely many times. Let $ n_1 $ be large enough so that, for all $ n\geq n_1 $, if $ x_n\in D(x_0, r) $, then $ \textrm{diam }F_n(D(x_0, r))<r/3  $. 

 By the construction of the sets $ \left\lbrace K_n\right\rbrace _n  $, there exists $ m\geq \max \left\lbrace n_0,  n_1\right\rbrace  $ such that $ \xi_m \in K^{n_0}_*$. Hence, we  take $ r>0 $, $ \rho=\rho_{n_0} $, and $ \xi_0 $ and $ \xi_m $ as above, and define $ \rho_m>0 $ as the radius such that the inverse branch $ G_m $ sending $ \xi_0 $ to $ x_m $ is defined around $ \xi_0 $ (such a radius exists by our assumptions on the orbit $ \left\lbrace \xi_n\right\rbrace _n $). We have to check that, with these choices, the requirements are satisfied.

First, by the choice of $ \left\lbrace \zeta_n\right\rbrace _n $, $ \varphi^*(\xi_0)\eqqcolon x_0 $ and $ \varphi^*(\xi_m)\eqqcolon x_m $ are well-defined. Moreover, by the choice of $ r $, we have $ x_0\in D(x, R/2) $. Since $ \xi_0, \xi_m \in K^{n_0}_*$, we have 
\[\Delta_{\rho}(\xi_0)\cap  \Delta_{ \rho}(\xi_n)\neq\emptyset,\]
and $ \Delta_{\rho}(\xi_0), \Delta_{\rho}(\xi_m)\subset A_{n_0}\subset D(x_0, r/3) $. In particular, $ x_m\in D(x_0, r/3) $, so (1.1) and (1.4) hold. To see (1.2), note that $ r $ has been chosen so that the inverse branches corresponding to $ \left\lbrace x_n\right\rbrace _n $ are well-defined in $ D(x_0,r) $, and $ m $ is large enough so that $ \textrm{diam }F_m(D(x_0, r))<r/3  $, as desired.
 Requirement (1.3)  is directly satisfied by the choice of $ \rho_m $. Therefore, we have proved the lemma.
\end{proof}

	\begin{figure}[htb!]\centering
	\includegraphics[width=16cm]{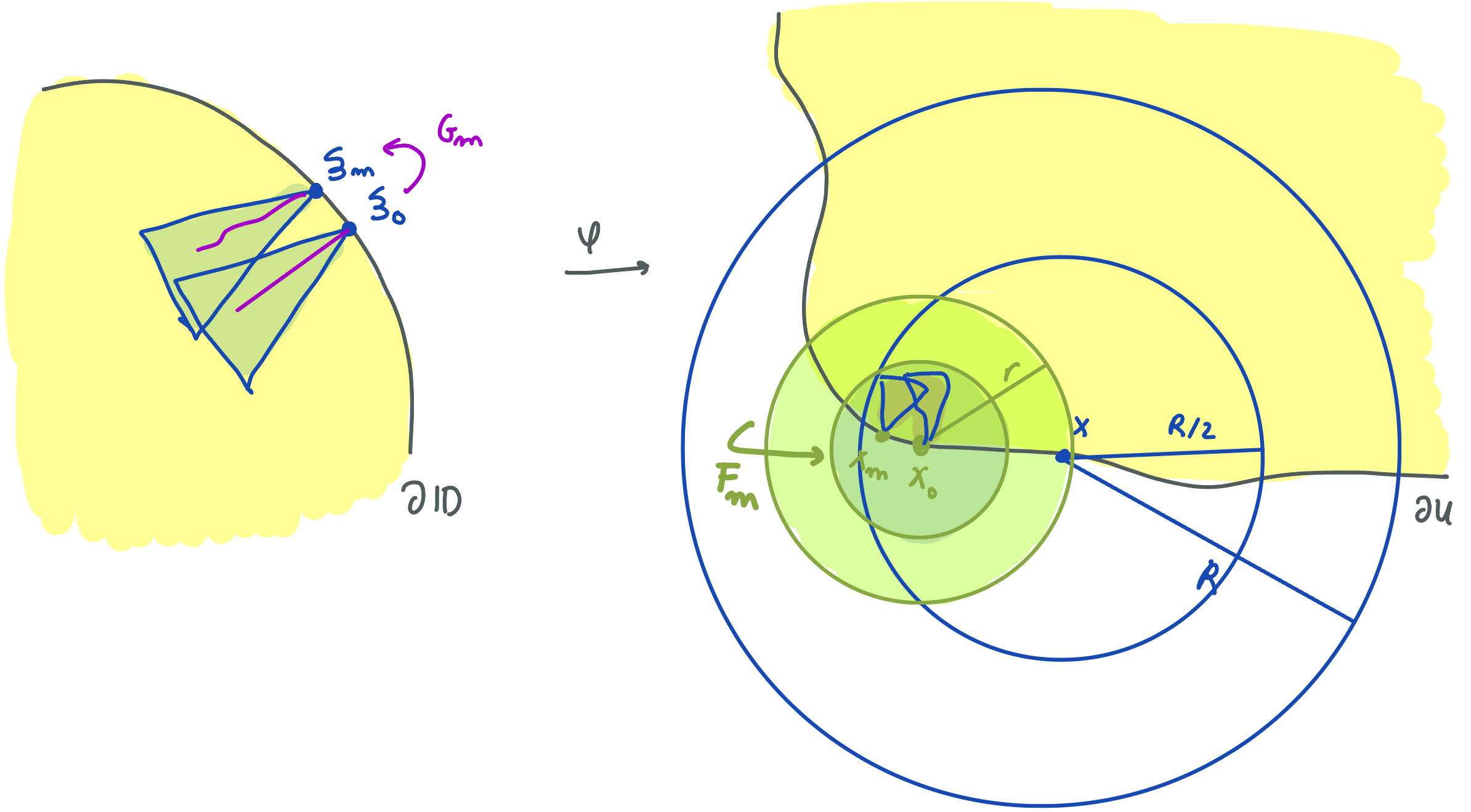}
	\setlength{\unitlength}{16cm}
     	\setlength{\unitlength}{16cm}
\put(-0.71, 0.27){$\xi_0$}
\put(-0.75, 0.35){$G_m$}
\put(-0.375, 0.2){$F_m$}
\put(-0.53, 0.225){$\varphi$}
\put(-0.7, 0.14){$\partial\mathbb{D}$}
\put(-0.4, 0.42){$\partial U$}
\put(-0.83, 0.365){$\xi_m$}
\put(-0.26, 0.23){\small$x_0$}
\put(-0.3, 0.22){\small$x_m$}
\put(-0.08, 0.135){$R$}
\put(-0.19, 0.105){\small $R/2$}
	\caption{\footnotesize Situation after Lemma \ref{sublemma-puntperiodic1}.}\label{fig-puntperiodic}
\end{figure}
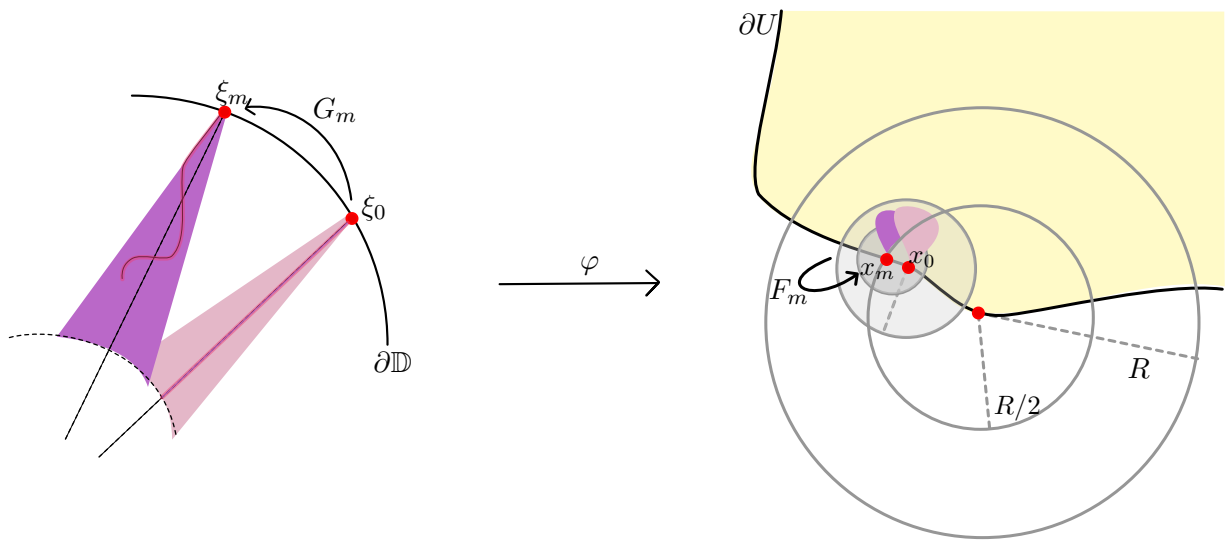

Next we prove the existence of a repelling periodic point in $ D(x_0, r) $. Note that, since $ D(x_0, r) \subset D(x, R)$, the proof of the next lemma ends the proof of the theorem.
\begin{sublemma}\label{sublemma-punt-periodic2}
	The map $ F_m $ has an attracting fixed point in $ D(x_0, r) $ which is accessible from $ U $. Hence, $ f $ has a repelling $ m $-periodic point in $ D(x_0, r) \cap\partial U$.
\end{sublemma}

\begin{proof}
First note that $ F_m(D(x_0, r)) \subset D(x_0, r)$. Indeed, by {\em\ref{item-1}} and {\em\ref{item-2}}, we have that  $ x_m\in D(x_0, r/3) $ and $ \textrm{diam } F_m(D(x_0, r))<r/3 $,  so \[ F_m(D(x_0,r))\subset D(x_m, 2r/3)\subset D(x_0, r).\] 
Therefore, by the Denjoy-Wolff Theorem, there exists a fixed point $ p\in D(x_0,r) $, which attracts all points in $ D(x_0, r) $ under the iteration of $ F_m $. Hence, it is repelling under $ f^m $ and thus belongs to $ \mathcal{J}(f) $.

    It is left to show that $ p $ is accessible from $ U $. To do so, first note that, by {\em \ref{item-3}}, the inverse  branch $ G_m$ of $ g^{-m} $  is well-defined in $ D(\xi_0, \rho_m) $, and we have that $$ \varphi\circ G_m= F_m\circ \varphi $$ in $ \Delta_{ \rho_m}(\xi_0) $. Moreover, we have that $$ G_m(R_{\rho_m}(\xi_0))\subset  \Delta_{	\rho_m}(\xi_m)\subset \Delta_{	\rho}(\xi_m)  .$$  By {\em \ref{item-4}}, $  \Delta_{	\rho}(\xi_0)\cup  \Delta_{ \rho}(\xi_m)$ is connected. Therefore, if we take $ z\in R_{\rho_m}(\xi_0) $, then $ G_m(z)\in  \Delta_{ \rho}(\xi_m)$, and 
     we can find a curve $ \gamma\subset \Delta_{	\rho}(\xi_0)\cup \Delta_{ \rho}(\xi_m)$ joining $ z $ and $ G_m(z) $. By {\em \ref{item-4}}, $ \varphi(\gamma)\subset D(x_0, r) $, and  joins $ \varphi(z) $ with $ F_m(\varphi(z)) $. See Figure \ref{fig-puntsperiodics}.
       	\begin{figure}[h]\centering
     	\includegraphics[width=16cm]{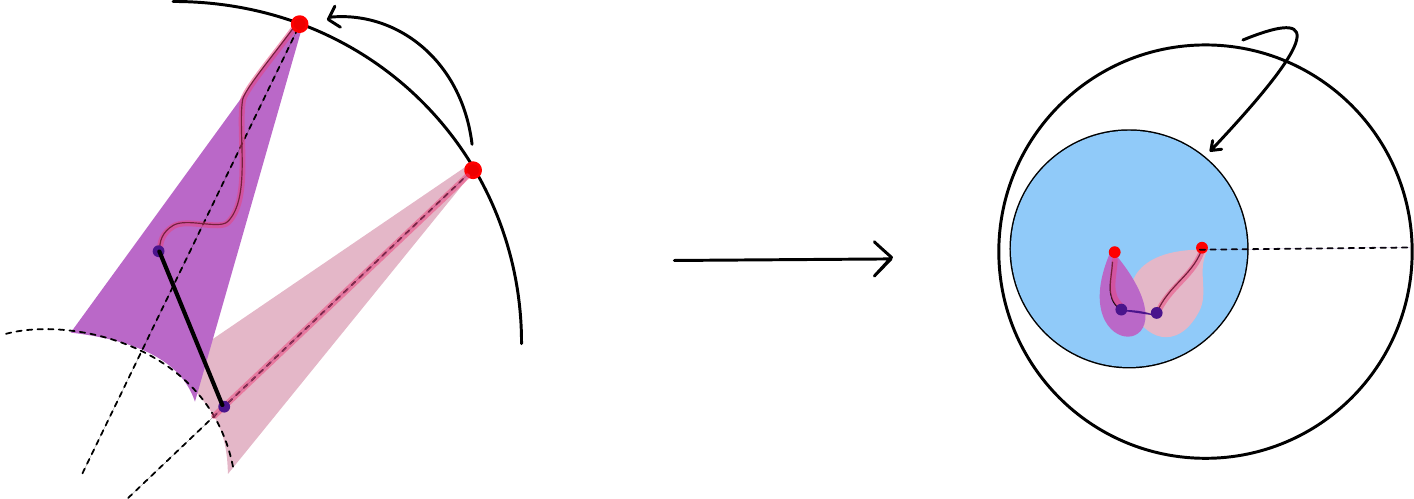}
     	\setlength{\unitlength}{16cm}
     	\put(-0.65, 0.23){$\xi_0$}
     	\put(-0.83, 0.05){$z$}
     	\put(-0.96, 0.15){$G_m(z)$}
     	\put(-0.68, 0.3){$G_m$}
     	\put(-0.08, 0.32){$F_m$}
     	\put(-0.865, 0.12){$\gamma$}
     	\put(-0.46, 0.18){$\varphi$}
     	\put(-0.65, 0.085){$\partial\mathbb{D}$}
     	\put(-0.8, 0.35){$\xi_m$}
     	\put(-0.22, 0.11){$\varphi(\gamma)$}
     	\put(-0.17, 0.18){$\varphi^*(\xi_0)$}
     	\put(-0.25, 0.18){$\varphi^*(\xi_m)$}
     	\put(-0.07, 0.16){$r$}
     	\caption{\footnotesize The construction of the curve $ \gamma $ in $ \mathbb{D} $, and its image $ \varphi(\gamma) $ in the dynamical plane.}\label{fig-puntsperiodics}
     \end{figure}
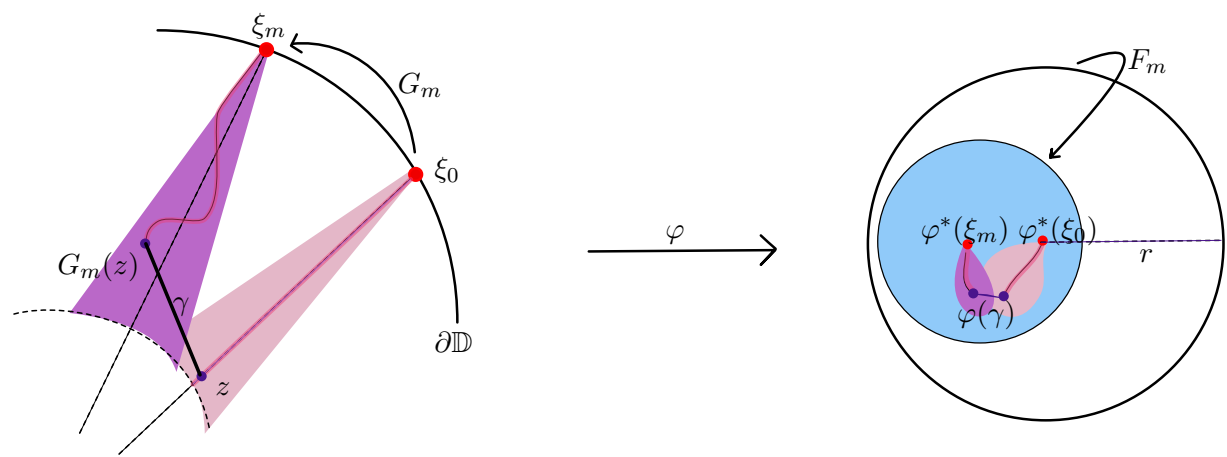
 
     Define \[\Gamma\coloneqq \bigcup\limits_{k\geq 0}F^k_m(\gamma).\]Then, $ \Gamma \subset\partial U$ lands at $ p $, ending the proof of Lemma \ref{sublemma-punt-periodic2}, and hence of Theorem \ref{thm-punts-periòdics}. 
\end{proof}

\section{Lyapunov exponents for transcendental maps}\label{sect-lyapunov}
 Let $ f\in\mathbb{K} $, and let $ X\subset \widehat{\mathbb{C}} $. Let $ \mu $ be a measure supported on $ X $, and assume $ \log \left| f'\right|  \in L^1(\mu)$. Then,
	\[ \chi_{\mu}\coloneqq \int_{X}\log\left| f'\right| d\mu\] is called the {\em Lyapunov exponent} of $ f $ (with respect to the measure $ \mu$).
In the previous sections, we were interested in the particular case where $ X $ is the boundary of an invariant Fatou component $ U $, and  $\mu= \omega_U $. We needed to assume $ \log \left| f'\right| \in L^1(\omega_U) $. It is well-known that this holds for rational maps \cite{Przytycki-Hausdorff_dimension}, but it is not clear what happens in the transcendental case. In Section \ref{subsection-integrabilitat-log f'}, we prove integrability of $ \log \left| f'\right| $ with respect to $ \omega_U $, under some assumptions on the shape of the Fatou component and the growth of the function.  In Section \ref{subsection-lyapunov-non-neg} we give conditions under which Lyapunov exponents are non-negative. Again, this is well-known for rational maps \cite{Przytycki-Lyapunovexponents}, but unexplored in the transcendental case. Finally, Section \ref{subsect-parabolic-basins-lyapunov} is devoted to extend some of the results to parabolic basins and Baker domains.

\subsection{Integrability of $ \log \left| f'\right|  $. Proposition \ref{prop:D}}\label{subsection-integrabilitat-log f'}

We examine the integrability of $ \log \left| f'\right|  $ with respect to harmonic measure. First observe that, by Harnack's inequality, if $ \log \left| f'\right| \in L^1(\omega_U(p,
\cdot)) $, then  $ \log \left| f'\right| \in L^1(\omega_U(q,
\cdot)) $, for all $ q\in U $. Hence,  we simply write  $\log \left| f'\right| \in L^1(\omega_U) $. 

To begin with, we prove that the integrability of $ \log \left| f'\right|  $ and the Lyapunov exponent is invariant under conjugating $ f $ by Möbius transformations. This is the content of the following lemma, which follows from  \cite[p. 165]{Przytycki-Hausdorff_dimension}. 

\begin{lemma}{\bf (Lyapunov exponent invariant under Möbius transformations)}\label{lemma-conjugate-mobius}
	Let $ f\in \mathbb{K} $, and let $ U $ be an invariant Fatou component for $ f $. Let $ M\colon\widehat{\mathbb{C}}\to \widehat{\mathbb{C}} $ be  a Möbius transformation, and let $ g\in\mathbb{K} $ be defined as $ g\coloneqq M\circ f\circ M^{-1} $. Then, $\log \left| f'\right| \in L^1(\omega_U) $ if and only if $\log \left| g'\right| \in L^1(\omega_{M(U)}) $. Moreover, if $ \omega_U $ is $ f $-invariant, then $ \omega_{M(U)} $ is $ g $-invariant, and \[\chi_{\omega_U} (f)=\chi_{\omega_{M(U)}} (g).\]
\end{lemma}

Observe that, for rational maps, $ \left| f'\right|  $ is bounded, so $ \chi_\mu (f) $ is well-defined (although \textit{a priori} may be equal to $ -\infty $). By a careful study of $ f $ around critical points, it is established that it is never the case, and in fact Lyapunov exponents are always non-negative (\cite{Przytycki-Lyapunovexponents}, see also \cite[Sect. 28.1]{UrbanskiRoyMunday3}). In the case of transcendental maps, $ \left| f'\right|  $ may not be bounded, and this is why we need an assumption on the growth.

To simplify the notation, in the sequel we shall assume $ \infty\in U $, hence $ \widehat{\partial}U $ is a compact subset of the plane, and that none of the singularities is placed at $ \infty $. 
\begin{defi}{\bf (Order of growth in sectors)}\label{defi-sectorS}
	Let $ f\in\mathbb{K} $, and let $ U\subset \widehat{\mathbb{C}} $ be an invariant Fatou component for $ f $. We say that $ U $ is {\em asymptotically contained in a sector of angle $ \alpha\in   \left( 0, 1\right)  $ with order of growth $ \beta >0 $} if there exists $ r>0 $, $ s_1, \dots, s_k\in\widehat{\partial} U \smallsetminus\left\lbrace \infty\right\rbrace $ and $ \xi_1, \dots, \xi_k\in\partial\mathbb{D} $, such that, if
	\[ S_{\alpha, r}=\bigcup\limits_{i=1}^k\left\lbrace z\in{\mathbb{C}} \colon \left| z-s_i\right| <r, \left|\textrm{Arg }\xi_i-\textrm{Arg }(z-s_i )\right| <\pi\alpha\right\rbrace \]satisfies
	\begin{enumerate}[label={(\alph*)}]
		\item $ \overline U \cap \bigcup\limits_{i=1}^k D(s_i, r) \subset S_{\alpha, r} $;
		\item $ f $ has order of growth $ \beta>0 $ in $  S_{\alpha, r} $, i.e. there exist $ A,B>0 $ such that for all $ z\in S_{\alpha, r}$,  \[A \cdot e^{B\cdot r^{\beta}}\leq \left| f'(z)\right| \leq A \cdot e^{B\cdot r^{-\beta}}.\]
	\end{enumerate}
\end{defi}

Geometrically, each of the sets 
\[ S_{i}=\left\lbrace z\in{\mathbb{C}} \colon \left| z-s_i\right| <r, \left|\textrm{Arg }\xi_i-\textrm{Arg }(z-s_i )\right| <\alpha\right\rbrace \] is a sector of angle $ \alpha \in (0,\pi)$, and side-length $ r>0 $, with vertex at $ s_i\in\widehat{\partial}U $.
Note that this notion of order of growth in sectors is invariant under conjugating $ f $ by a Möbius transformation $ M $, as long as $ M(U)\subset\mathbb{C} $ and $ M(s_i)\neq \infty $, $ i=1,\dots, k $. Indeed, $ M' $ is uniformly bounded around $ \widehat{\partial} U$ and hence distances are distorted in a controlled way when applying $ M $.

Next we check that, if $ U $ is asymptotically contained in a sector of angle $ \alpha\in ( 0, 1) $, with order of growth $ \beta\in(0, 1/2\alpha) $, then $ \log \left| f'\right| \in L^1(\omega_U) $. 
\begin{prop}{\bf ($ \log\left| f'\right|  $ is $ \omega _U$-integrable)}\label{prop-f'-integrable}
	Let $ f\in\mathbb{K} $, and let $ U $ be an invariant Fatou component for $ f $. Assume $ U $ is asymptotically contained in a sector of angle $ \alpha\in ( 0, 1) $, with order of growth $ \beta\in(0, \frac1{2\alpha}) $. Then, $ \log \left| f'\right| \in L^1(\omega_U) $. 
\end{prop}

Note that Proposition \ref{prop-f'-integrable} implies Proposition \ref{prop:D}. Before proving it, we need some estimates on the harmonic measure of sectors.

\subsubsection*{Estimates on the harmonic measure of sectors}

We start by recalling the following estimate on harmonic measure of disks for  simply connected domains, which follows from Beurling's Projection Theorem  \cite[Thm. 9.2]{harmonicmeasure2}.
\begin{thm}{\bf (Harmonic measure of disks, {\normalfont \cite[p. 281]{harmonicmeasure2}})}\label{thm-beurling}
	Let $ U \subset\widehat{\mathbb{C}}$ be a simply connected domain, such that $ \infty \in U $ and $  \textrm{\em diam}(\partial U)=2 $. Then, for all $ x\in\partial U $ and $ r>0 $, we have \[\omega_U(\infty, D(x,r))\leq \sqrt{r}.\]
\end{thm}

Assuming that $ U $ is contained in some sector \[ S_{\alpha, r}(x, \xi)=\left\lbrace z\in{\mathbb{C}} \colon \left| z-x\right| <r, \left|\textrm{Arg }\xi-\textrm{Arg }(z-x)\right| <\pi \alpha\right\rbrace ,\] with vertex at $ x\in{\partial}U $ (see Fig. \ref{fig-sector}), we obtain improved estimates of harmonic measure for disks centered at $ x $.

\begin{figure}[h]\centering
	\includegraphics[width=15cm]{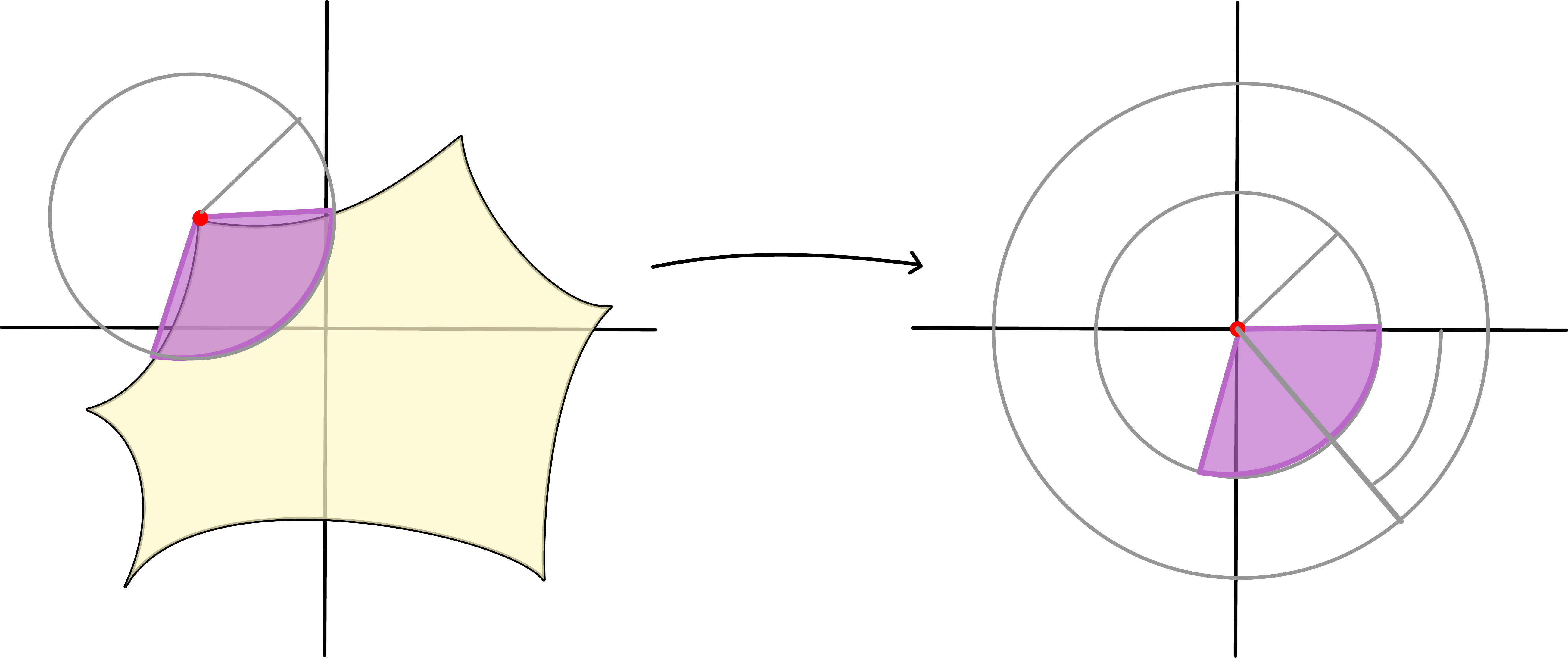}
	\setlength{\unitlength}{15cm}
	\put(-0.55, 0.26){\small $ z\mapsto z-x $}
	\put(-0.895, 0.28){\small $ x $}
	\put(-0.85, 0.32){\small $ r $}
	\put(-0.18, 0.25){\small $ r $}
	\put(-0.085, 0.15){\small $ \pi\alpha $}
	\put(-0.1, 0.07){\small $ \xi $}
	\put(-0.83, 0.19){\small $ S_{\alpha, r}  (x, \xi)$}
	\put(-0.75, 0.15){\small $ U $}
	\caption{\footnotesize  A visual representation of the definition of the sector $ S_{\alpha, r}  (x, \xi) $.}\label{fig-sector}
\end{figure}
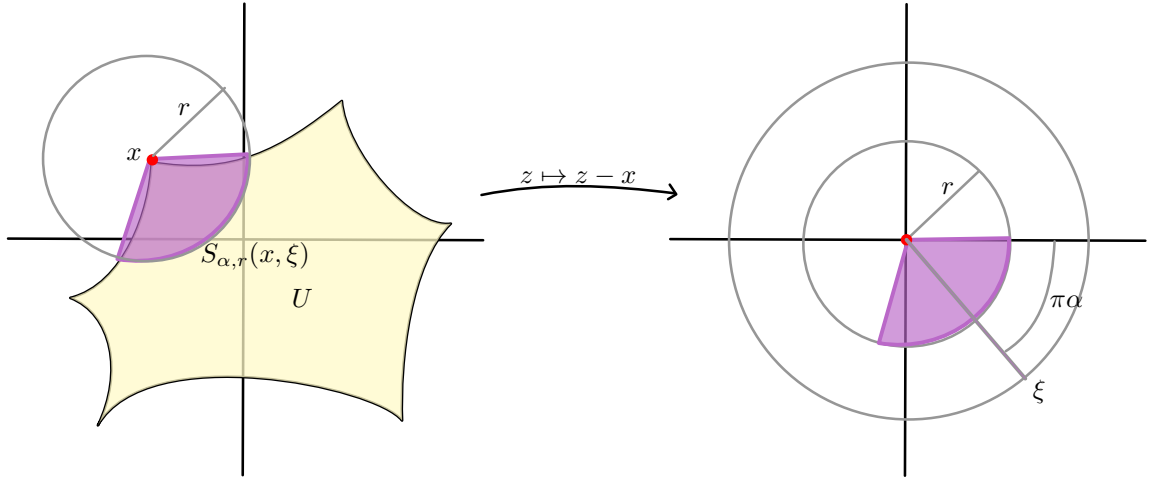

\begin{lemma}{\bf (Harmonic measure of sectors)}\label{lemma-estimates-harmonic-measure}
	Let $ U \subset \mathbb{C}$ be a simply connected domain, and let $ z_0\in  U $, $ x\in{\partial} U $. Assume there exists $ r_0>0 $, $ \alpha\in \left( 0, 1\right) $ and $ \xi\in \partial\mathbb{D} $, such that  $$ D(x,r_0)\cap U \subset S_{\alpha, r_0} (x, \xi).$$
	Then, there exists $ C>0$ and $ r_1\in (0, r_0) $ such that, for all $ r\in(0,r_1) $, 
	\[ \omega_U(z_0,D(x, r))\leq C\cdot r^\frac{1}{2\alpha} .\]
\end{lemma}
\begin{proof}
	Without loss of generality we assume $ z_0\notin  S_{\alpha, r_0} (x, \xi) $, and let $ r\in(0,r_0) $.
	First observe that, if $ V $ denotes the connected component of $ U\smallsetminus  D(x,r)$ that contains $ z_0 $, we have that 
	\[\omega_U(z_0, D(x,r))= \omega_U(z_0, D(x,r)\cap \partial U)\leq \omega_V (z_0, \partial D(x,r)\smallsetminus \partial U)\leq \omega_V (z_0, \partial D(x,r)), \] where in the first inequality we applied the Comparison Lemma \cite[Prop. 21.1.13]{Conway} (note that we apply it to the complements, and hence the inequality is reversed), and the second follows from the inclusion of the measured sets.
	
	Next we observe that, without loss of generality, we can assume that $$ V\subset  S_{\alpha}(x, \xi)=\left\lbrace z\in{\mathbb{C}} \colon  \left|\textrm{Arg }\xi-\textrm{Arg }(z-x)\right| <\pi \alpha\right\rbrace.$$ Indeed, since we want to estimate the harmonic measure of disks $ D(x,r) $ centered at $ x \in\partial U$ (which is a local property of the boundary around the point $ x $), and $ D(x,r_0)\cap U \subset S_{\alpha, r_0} (x, \xi) $, for $ r>0 $ small enough, we can disregard the part of $ \partial U $ outside $ D(x, r_0) $.
	
	Therefore, up to composing by appropriate Möbius transformations, it is left to show that, if 
	\[ S= \left\lbrace z\in\mathbb{C}\colon \left| \textrm{Arg }z\right|<\pi\alpha\right\rbrace , \] then, for some constant $ C>0 $, we have 
	\[\omega_{S\smallsetminus D(0,r)} (1, \partial D(0,r))\leq C \cdot r^{\frac{1}{2\alpha}}.\]
	However, since the length of a circumference of radius $ r $ is $ 2\pi r $ (i.e. proporcional to the radius), it is enough to see that $ \omega_{S} (1, D(0,r)) $  decays to 0 like $ r^{\frac{1}{2\alpha}} $, when $ r\to 0 $ (see Fig. \ref{fig-hmsector}).  But, since $ M(z)=z^{\frac{1}{2\alpha}} $ is a conformal map from $ S $ to the right half-plane fixing 1, this follows immediately (see again Fig. \ref{fig-hmsector})
	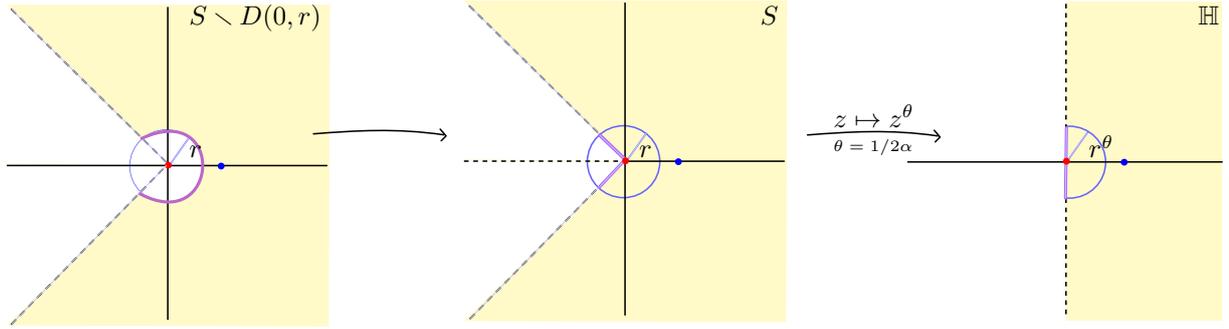
\begin{figure}[htb!]\centering
		\includegraphics[width=16cm]{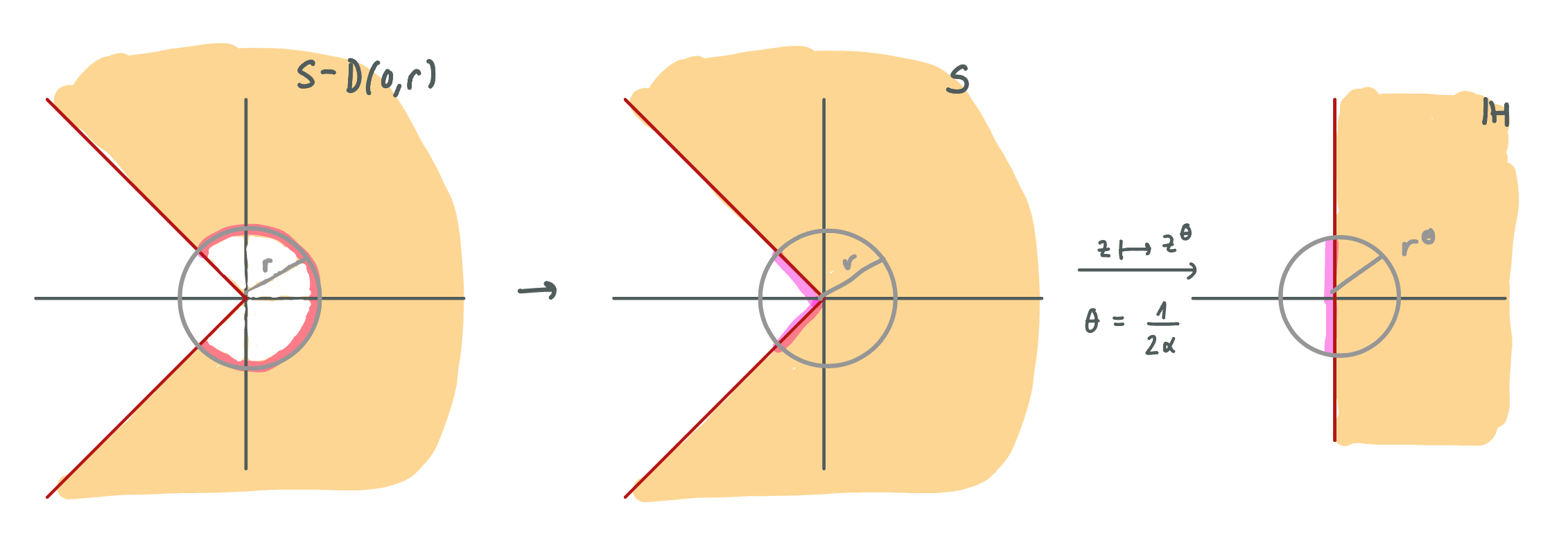}
		\setlength{\unitlength}{16cm}
		\put(-0.85, 0.25){\small$S\smallsetminus D(0,r)$}
		\put(-0.38, 0.25){\small$S$}
		\put(-0.02, 0.25){\small$\mathbb{H}$}
		\put(-0.85, 0.14){\small$r$}
		\put(-0.48, 0.14){\small$r$}
		\put(-0.11, 0.14){\small$r^\theta$}		
		\put(-0.32,0.165){\small $ z\mapsto z^\theta $}
		\put(-0.32,0.145){\tiny $ \theta=1/2\alpha $}				
		\caption{\footnotesize  A visual scheme to approximate harmonic measure of sectors. }\label{fig-hmsector}
	\end{figure}
	
\end{proof}

\subsection*{Proof of Proposition \ref{prop-f'-integrable}}

\begin{proof}[Proof of Proposition \ref{prop-f'-integrable}]
	By conjugating by a Möbius transformation if needed, we can assume $ \infty\notin U $. 
Note that $ \log \left| f'\right| $ is integrable with respect to harmonic measure when restricted to compact subsets of the domain $ \widehat{\mathbb{C}}\smallsetminus\left\lbrace s_1,\dots, s_k\right\rbrace  $. Indeed, the only difficulty is to see that $ \log \left| f'\right| $ is integrable around critical points. It is easy to check this by considering the Taylor expansion of $ f $ around the critical point, and using that $ \log \left| z-a\right|  $ is integrable with respect to $ \omega_U $ for all $ a\in\mathbb{C} $ (see e.g. \cite[Sect. 11.2]{PrzytyckiUrbanski}).
	
	It is left to check integrability near the singularities, and here is where we use the estimates on the growth. Let us use the notation
	\[\log^+\left| f'(z)\right| \coloneqq\max(0, \log \left| f'(z)\right| ), \hspace{0.7cm}\log^-\left| f'(z)\right| \coloneqq-\min(0, \log \left| f'(z)\right| ),\]
	so that
	$\left| \log\left| f'\right|\right| =\log^+\left| f'\right| +\log^-\left| f'\right| $ .
	Since $ \log^+\left| f'\right| $ and $ \log^-\left| f'\right|   $ satisfy analogous estimates, we check only that $ \log^+\left| f'\right| \in L^1(\omega_U) $. In fact, we only need to check integrability near $ s_i $, say in a disk $ D(s_i,r) $ for some $ r>0 $. Write $ D_n=D(s_i, 1/n) $, for $ n $ small enough. We have 
	\begin{align*}
		\int_{\partial U\cap D(s_i, r)} \log^+\left| f'\right| d\omega_U& \lesssim \sum_nn^\beta \left( \omega_U(D_n)-\omega_U(D_{n+1})\right) 
\lesssim \sum_n ((n+1)^\beta-n^\beta)\omega_U(D_{n+1})\\&\lesssim \sum_n n^{\beta-1}\cdot \frac1{n^{2\alpha}}=\sum_n \frac1{n^{-\beta+1+2\alpha}}.
	\end{align*}
The hypothesis $ \beta\in(0, \frac1{2\alpha}) $ guarantees the convergence of the sum, and hence of the integral, as desired.
\end{proof}

\subsection{Non-negative Lyapunov exponents. Proposition \ref{prop:E}}\label{subsection-lyapunov-non-neg}
Next, we give conditions under which $ \chi_{\omega_U} $ is non-negative. Our result is inspired by \cite[Lemma 9.1.2, Corol. 9.1.3]{KotusUrbanski}, but we remark that we do not assume that $ f $ extends holomorphically (in fact, not even continuously)  to a neighbourhood of $ \widehat{\partial}U $.

\begin{prop}{\bf (Lyapunov exponents are non-negative)}\label{prop-lya-non-neg}
	Let $ f\in\mathbb{K} $, and let $ U $ be an invariant Fatou component for $ f $, such that $ \omega_U $ is $ f $-invariant. Assume\begin{enumerate}[label={\em (\alph*)}]
		\item $ U $ is asymptotically contained in a sector of angle $ \alpha\in \left( 0, 1\right) $, with order of growth $ \beta\in(0, \frac1{2\alpha}) $;
		\item\label{hipotesib} $ \int_{\partial U} \log\left| x-SV\right|^{-1} d\omega_U(x)<\infty  $.
	\end{enumerate}
	Then, \[\chi_{\omega_U}=\int_{\partial U}\log \left| f'\right| d\omega_U\geq 0.\]
\end{prop}

Note that Proposition \ref{prop-lya-non-neg} implies Proposition \ref{prop:E}.

\begin{proof}
	If $ \omega_U $ is $ f $-invariant, then $ U $ is either an attracting basin or a Siegel disk, and $ \omega_U $ is precisely the harmonic measure with basepoint the fixed point  $ p\in U $. In particular, $ f|_{\partial U} $ is ergodic with respect to $ \omega_U $.
	
	\vspace{0.3cm}
	\noindent
	{\em 1.  Asymptotic contraction of $ f^n|_{\partial U} $, $ \omega_U $-almost everywhere.}
	By Proposition \ref{prop-f'-integrable}, the integral $\chi_{\omega_U}=\int_{\partial U}\log \left| f'\right| d\omega_U$ is well-defined. Since $ f|_{\partial U} $ is ergodic, Birkhoff Ergodic Theorem \ref{thm-birkhoff}, for $ \omega_U $-almost every $ x\in\partial U $, \[\lim\limits_n \frac{1}{n}\log \left| (f^n)'(x)\right| =\chi_{\omega_U}(f).\] We want to show that $  \chi_{\omega_U}(f)\geq 0$. We shall assume, on the contrary, that $  \chi_{\omega_U}(f)< 0$, and seek for a contradiction.
	
	Since $  \chi_{\omega_U}(f)< 0$, it follows that there exists $ M\in (e^{\frac{\chi_{\omega_U}}{4}}, 1) $ and $ n_0\coloneqq n_0(x) \in\mathbb{N}$ such that, for all $ n\geq n_0 $, \[\left| (f^n)'(x)\right| ^\frac{1}{4}\leq M^n<1.\]
	We fix $ x\in \partial U $ satisfying the previous property, and we denote by $ \left\lbrace x_n\right\rbrace _n $ its forward orbit.
	
	\vspace{0.3cm}
	\noindent
	{\em 2. Shrinking domains where $ f|_{\partial U} $ is univalent, $ \omega_U $-almost everywhere.}
	Let $ M \in (0,1)$ be the constant fixed in the previous step, and let $ x_n=f^n(x) $, for $ n\geq 0 $.
	
	\begin{sublemma}\label{lemma-univalent}
		For $ \omega_U $-almost every $ x\in\partial U $ and $ \lambda\in (M,1) $, there exists $ n_1\geq n_0 $ such that $ f|_{D(x_n, \lambda^n)} $ is univalent, for all $ n\geq n_1 $.
	\end{sublemma}
	In particular, since $ \lambda>M $, $ f|_{D(x_n, M^n)} $ is univalent, for all $ n\geq n_1 $.
	
	\begin{proof}
		Since $ \beta<\frac{1}{2\alpha} $, we can choose $ \gamma\in (\beta, \frac{1}{2\alpha}) $. Then, applying the estimates of Lemma \ref{lemma-estimates-harmonic-measure}, we have 
		\[\omega_U\left( S_{\alpha, n^{-\frac{1}{\gamma}}}\right) \leq C\cdot n^{-\frac{1}{\gamma\cdot 2\alpha}},\] and therefore \[ \sum_{n\geq 1}\omega_U\left( S_{\alpha, n^{-\frac{1}{\gamma}}}\right)\leq \sum_{n\geq 1} C\cdot n^{-\frac{1}{\gamma\cdot 2\alpha}} <+\infty.\]
		
		By the assumption on the growth, for all $ z\notin S_{\alpha, n^{-\frac{1}{\gamma}}} $ and $ n\in\mathbb{N} $ large enough, we have $$ \left| f'(z)\right| \leq C\cdot e^{n^{\frac{\beta}{\gamma}}} ,$$ for some constant $ C>0 $. Then, it is easy to see  that there exists  $ C'>0 $ such that, for $ z\notin S_{\alpha, n^{-\frac{1}{\gamma}}} $ and $ n\in\mathbb{N} $ large enough, $$ \left| f'(z)\right| \leq C'\cdot \lambda ^{-n/4},$$ where $ \lambda\in (M,1) $ is the constant given in the statement of the lemma. 
		
		Now, using the previous computations and assumption {\em \ref{hipotesib}}, the first Borel-Cantelli Lemma \ref{lemma-borel-cantelli} yields that, for $ \omega_U $-almost every $ x\in \partial U $ and $ n $ large enough (depending on $ x $), we have \begin{enumerate}[label={(1.\arabic*)}]
			\item\label{lemma35a} $ x_{n+1}\notin \bigcup\limits_{s\in SV}D(s, \lambda^{(n+1)/2}) $, 
			\item $ x_{n}\notin   S_{\alpha, n^{-\frac{1}{\gamma}}}$.
		\end{enumerate}
		
		By \ref{lemma35a}, all inverse branches of $ f $ are well-defined in $ D(x_{n+1}, \lambda^{(n+1)/2}) $ and are univalent. Denote by $ F $ the inverse branch of $ f $ defined in  $ D(x_{n+1}, \lambda^{(n+1)/2}) $ such that $ F(x_{n+1})=x_n $. By Koebe's distortion estimates \ref{thm-Koebe}, we have 
		\[F(D(x_{n+1}, {\lambda^{(n+1)/2}}))\supset D(x_n, R),\]where \[R=\frac14\cdot \left| F'(x_{n+1})\right|\cdot {\lambda^{(n+1)/2}}
		=\frac{\lambda^{(n+1)/2}}{4}\cdot\frac1 {\left| f'(x_n)\right|}\geq \frac{\lambda^{(n+1)/2}}{4}\cdot \lambda^{n/4}=K\cdot \lambda^\frac{3n}4,\]for some constant $ K>0 $. 
		It follows that there exists $ n_1\coloneqq n_1(x)$ large enough so that, for $ n\geq n_1 $,
		\[F(D(x_{n+1}, {\lambda^{(n+1)/2}}))\supset D(x_n, \lambda^n).\]
		Hence, $ f|_{D(x_n, \lambda^n)} $ is univalent, for all $ n\geq n_1 $.
	\end{proof}
	
	Hence, we fix a point $ x\in \partial U $ such that its forward orbit $ \left\lbrace x_n\right\rbrace _n $ satisfies the following conditions, with $ M\in (0,1) $ and $ n_1\coloneqq n_1(x) $ as above:
	\begin{enumerate}[label={(2.\arabic*)}]
		\item  $ \left| (f^n)'(x)\right| ^\frac{1}{4}\leq M^n<1 $, for all $ n\geq n_1 $;
		\item $ f|_{D(x_n, M^n)} $ is univalent, for all $ n\geq n_1 $.
	%	\item $ x_n\notin SV $, for all $ n\in\mathbb{N} $.
	\end{enumerate}
%	The third condition follows from the fact that singular values have zero harmonic measure.
	
	\vspace{0.3cm}
	\noindent  {\em 3. Quantitative contraction of $ f^n|_{D(x, \lambda^n) }$, for $ n $ large enough.}
	Let \[b_n\coloneqq \left| (f^{n+1})'(x)\right| ^{\frac14}, \hspace{0.5cm}P\coloneqq \prod_{n\geq 1}(1-b_n).\]
	Observe that, since $ \sum b_n\leq\sum M^n<\infty $, $ P $ is convergent. For all $ n $, let $$ D_n \coloneqq D(x, r\cdot \prod_{m= 1}^{n}(1-b_m)).$$
	We can choose $ r\coloneqq r(x)>0 $ small enough so that
	$ 2r<1$,
	$ f^{n_1}|_{D_{n_1}}$ is univalent,
	and $ f^{n_1}({D_{n_1}})\subset D(x_{n_1}, M^{n_1})$.
	
	\begin{claim}
		For $ n\geq n_1$, $ f^{n}|_{D_{n}}$ is univalent, and 
		$ f^n(D_n)\subset D(x_n, M^n) $.
	\end{claim}
	
	 It follows from the claim that, for all $ n\in\mathbb{N} $, $ f^n $ is univalent in $ D(x, rP) $ and $ f^n|_{D(x, rP)}\subset D(x_n, M^n) $.
	\begin{proof}
		We prove the claim inductively: assume the claim is true for $ n\geq n_1 $, and let us see that it also holds for $ n+1 $.
		
		First note that, since 	$ f^n(D_n)\subset D(x_n, M^n) $ (by inductive assumption) and $ f $ is univalent in $ D(x_n, M^n)  $ (by Lemma \ref{lemma-univalent}), it follows that $ f^{n+1}|_{D_{n}}$ is univalent. In particular, since $ D_{n+1}\subset D_n $, we have that $ f^n|_{ D_{n+1}} $ is univalent.
		
		Now we use Koebe's distortion estimates (Thm. \ref{thm-Koebe}) to prove the bound on the size of $  f^{n+1}( D_{n+1})$. Indeed, since $ f^{n+1}|_{D_{n}}$ is univalent, we have $ f^{n+1}(D_{n+1})\subset D(x_{n+1}, R) $, where \begin{align*}
			R=r\cdot \prod_{m\geq 1}^{n}(1-b_m))\cdot \left| (f^{n+1})'(x)\right| \cdot \dfrac{2}{b^3_n}\leq 2r\cdot\dfrac{\left| (f^{n+1})'(x)\right| }{\left| (f^{n+1})'(x)\right| ^{\frac34}}\leq\left| (f^{n+1})'(x)\right| ^{\frac14} \leq M^{n+1},
		\end{align*} as desired.
	\end{proof}
	\vspace{0.2cm}
	\noindent{\em 4. Contradiction with the blow-up property of the Julia set.}
	Let $ R>0 $ be small enough so that $ D(p, R) \subset U$, where $ p $ is the fixed point of $ f $ in $ U $. Let $ n_2\geq n_1 $ be such that $ M^{n_2}<\frac{R}{2} $ (recall that $ M\in (0,1) $, so such $ n_2 $ exists). 
	
	Then, $ f^{n_2}(D(x, rP)) $ is a neighbourhood of $x_{n_2}= f^{n_2}(x)\in\mathcal{J}(f) $. By the previous step, \[\bigcup_{n\geq n_2} f^n(D(x, rP))\subset \bigcup_{n\geq n_2} D(x_n, M^n) \subset \bigcup_{n\geq n_2} D(x_n, M^{n_2})\subset \widehat{\mathbb{C}}\smallsetminus D(p, R/2).\]This is a contradiction of the blow-up property of the Julia set. Notice that the contradiction comes from assuming $ \chi_{\omega_U}<0 $. Therefore, $ \chi_{\omega_U}\geq 0 $, and this ends the proof.
\end{proof}

\subsection{The Lyapunov exponent for parabolic basins and Baker domains}\label{subsect-parabolic-basins-lyapunov}
The boundary of parabolic basins and doubly parabolic Baker domains do not support invariant probabilities which are absolutely continuous with respect to the harmonic measure $ \omega_U $. However,  the measure 
\[\lambda_\mathbb{R} (A)=\int_A \frac{1}{\left| w-1\right| ^2}d\lambda(w), \hspace{0.5cm}A\in \mathcal{B}(\partial\mathbb{D}),\] is invariant under the radial extension of the associated inner function $ g $ (taken such that 1 is the Denjoy-Wolff point)
and its push-forward $\mu=(\varphi^*)_*\lambda_\mathbb{R}$ is an infinite invariant measure supported on $ \widehat{\partial}U $.

Hence, in the case of parabolic basins and doubly parabolic Baker domains, we shall consider the Lyapunov exponent of $ f $ with respect to $ \mu $ \[ \chi_{\mu}(f)\coloneqq \int_{\partial U}\log\left| f'\right| d\mu.\] We show that, for parabolic basins, if $ \log \left| f'\right|  $ is integrable with respect to harmonic measure, then it is also integrable with respect to the invariant measure $ \mu $. Note also that Lemma \ref{lemma-conjugate-mobius} and Proposition \ref{prop-f'-integrable} give conditions for $ \log \left| f'\right| \in L^1(\omega_U) $ and do not assume that $ \omega_U $ is invariant, so they  still hold in the parabolic setting.

\begin{prop}{\bf (Parabolic Lyapunov exponents)}\label{prop-log-int-parabolic-basins}
	Let $ f\in\mathbb{K} $, and let $ U $ be a parabolic basin. If $ \log \left| f'\right| \in L^1(\omega_U) $, then $ \log \left| f'\right| \in L^1(\mu) $. If, in addition, $ \chi_{\omega_U}> 0 $, then $ \chi_{\mu}>0 $.
\end{prop}
\begin{proof}
	For the first statement note that, since $ \omega_U $ and $ \mu $ are comparable except in a neighbourhood of the parabolic fixed point $ p\in\partial U $, it is enough to check that 
	\[\int_{D(p,r)}\log \left| f'\right| d\mu<\infty,\] for some $ r>0 $. We note that, in contrast with the situation considered in Proposition \ref{prop-f'-integrable}, $  \log \left| f'\right| $ achieves a (finite) maximum and minimum around $ p $  (since $ f'(p)=1 $), but now the difficulty comes from the fact that $ \mu $ is an infinite measure.

	On the one hand,  around the parabolic fixed point (which we assume to be the origin), we have the following normal form,    
	$f(z)=z+az^q+\dots$, with $ a\in\mathbb{C} $ and $ q\geq 2 $ (see e.g. \cite[Sect. 10]{milnor}). Therefore, 
	$\log \left| f'(z)\right| \sim \log (1+qa\left| z\right|^{q-1})$.

	On the other hand, by Lemma \ref{thm-beurling}, we have that 
	\[\lambda ((\varphi^*)^{-1}(D(p,r)))=\omega_U(D(p,r))\leq C\cdot\sqrt r.\] 
	Therefore, setting $ D_n=D(p, 1/n) $, \begin{align*}
		\int_{D(p,r)}\log \left| f'\right| d\mu&\lesssim \sum_{n=n_0}^\infty\log\left| 1+qa \frac{1}{n^{q-1}}\right|( \mu(D_n)-\mu(D_{n+1}))
		\lesssim \sum_{n=n_0}^\infty\left( \frac{1}{n ^{q-1}}-\frac{1}{{(n+1)}^{q-1}}\right)  \cdot \mu(D_n)\\&\lesssim  \sum_{n=n_0}^\infty\frac{n^{q-2}}{n^{2(q-1)}}\cdot \sqrt n<\infty,
	\end{align*}
	as desired.
	
	For the second statement, applying the Leau-Fatou Flower Theorem (see e.g. \cite[Sect. 10]{milnor}), we have that $ \log \left| f'\right| >0 $ in $ D(p,r)\cap \partial U $, for $ r $ small enough. Then the statement follows directly.
\end{proof}
\begin{remark} Proposition \ref{prop-log-int-parabolic-basins} is stated only for parabolic basins, and its proof used the normal form around a parabolic fixed point. For a Baker domain, there is no longer a normal form around the convergence point, since it is an essential singularity for $ f $, and hence the argument cannot be applied in general. However, for some explicit Baker domains, similar estimates can be obtained and the argument may work \textit{ad hoc}.
	Indeed, consider for instance the Baker domain of the map $ f(z)=z+e^{-z} $ (see \cite{FagellaJove}). Since it is contained in a strip and $ f $ has finite order, by Proposition \ref{prop-f'-integrable}, $ \log \left| f'\right| \in L^1(\omega_U) $. To see that $ \log \left| f'\right| \in L^1(\mu) $, it is enough to check integrability in a neighbourhood of infinity. Note that $ f'(z)=1-e^{-z} $, so the estimates on $ \left| f' \right| $ are even better than in the parabolic case, and the same argument can be applied.
	Moreover, $ \left| f'\right| >1 $ in a neighbourhood of $ \partial U $ \cite[Prop. 3.6]{FagellaJove}, so $ \chi_\mu (f)>0 $.
\end{remark}
\bibliographystyle{amsalpha}
{\footnotesize\bibliography{referencies.bib}}

\end{document}